\documentclass[11pt]{amsart}
\usepackage{amssymb}
\usepackage{mathrsfs}
\usepackage{enumerate}
\usepackage{hyperref}
\usepackage{color}

\newtheorem{theorem}{Theorem}[section]
\newtheorem{definition}{Definition}[section]
\newtheorem{proposition}[theorem]{Proposition}
\newtheorem{lemma}[theorem]{Lemma}
\newtheorem{corollary}[theorem]{Corollary}
\newtheorem{claim}[theorem]{Claim}
\newtheorem{remark}[theorem]{Remark}

\newtheorem{theoremalph}{Theorem}

 \def\BB{{\mathbb B}} 
\def\DD{{\mathbb D}}

 \def\NN{{\mathbb N}}

\def\UU{{\mathbb U}}  

 \def\ZZ{{\mathbb Z}}

\def \F{\EuScript{F}}

\def \B{\mathcal{B}}

\def \M{\mathcal{M}}

\def\M {{\mathcal M}}
\def\F {{\mathcal F}}
\def\B {{\mathcal B}}
\def\P{\mathcal P}

\def\U{{\mathcal U}}
\def\V{{\mathcal V}}

\def\W{{\mathcal W}}

\begin{document}

\title[Mostly expanding and mostly contracting centers]{Statistical stability for diffeomorphisms with mostly expanding and mostly contracting centers}
\author[Zeya Mi]{Zeya Mi}

\address{School of Mathematics and Statistics,
Nanjing University of Information Science and Technology, Nanjing 210044, Jiangsu, P.R. China.}
\email{\href{mailto:mizeya@163.com}{mizeya@163.com}}

\author[Yongluo Cao]{Yongluo Cao}

\address{Department of Mathematics, Shanghai key laboratory of PMMP, East China Normal University, Shanghai 200062, P.R. China.}
\address{Department of Mathematics, Center for Dynamical Systems and Differential Equations, Soochow University, Suzhou 215006, Jiangsu, P.R. China.}

\email{\href{mailto:ylcao@suda.edu.cn}{ylcao@suda.edu.cn}}
%\urladdr{\href{http://math.rice.edu/~zg2/}{http://math.rice.edu/$\sim$zg2/}}

%\urladdr{\href{http://math.rice.edu/~hk7/}{http://math.rice.edu/$\sim$hk7/}}

\thanks{Z.Mi\ was supported by NSFC 11801278 and The Startup Foundation for Introducing Talent of NUIST(Grant No. 2017r070)}
\thanks{Y.Cao\ was supported by NSFC (11771317,11790274).}

\date{\today}


\begin{abstract}
For partially hyperbolic diffeomorphisms with mostly expanding and mostly contracting centers, we establish a topological structure, called \emph{skeleton}--a set consisting of finitely many hyperbolic periodic points with maximal cardinality for which there exist no heteroclinic intersections. We build the one-to-one corresponding between periodic points in \emph{any} skeleton and physical measures. By making perturbations on skeletons, we 
study the continuity of physical measures with respect to dynamics under $C^1$-topology. 
\end{abstract}

\maketitle

%\tableofcontents

\section{Introduction}
Given a $C^1$ diffeomorphism $f$ on a smooth compact Riemannian manifold $M$, an invariant measure $\mu$ is called a \emph{physical measure} if it is physical observable \cite{ER}, in the sense that the set of its \emph{basin}
$$
B(\mu,f)=\Big\{x\in M: \frac{1}{n}\sum_{i=0}^{n-1}\delta_{f^ix}\xrightarrow{weak^*} \mu ~\textrm{as} ~n \to +\infty\Big\}
$$
has positive Lebesgue measure. Physical measures were discovered by Sinai, Ruelle and Bowen \cite{Sin72, Rue76, Bow75, BoR75}  in 1970's for uniformly hyperbolic systems, which can reflect the chaotic properties of dynamical systems from different viewpoints, such as entropies, Lyapunov exponents and equilibrium states (see \cite{pa,y02}). 
%In this work, for $C^{1+}$ diffeomorphism, we say that an invariant measure is an SRB measure if it has positive Lyapunov exponents and its conditional measures along Pesin unstable manifolds are absolutely continuous with respect to Lebesgue measures on these manifolds.

When the existence of physical measures is established, one expects to describe how the physical measures bifurcate as dynamics changes, which known as the subject of \emph{statistical stability}. Uniformly hyperbolic systems are well-known to be statistically stable \cite{ru97}. In recent decades, much progress has been made on the statistical stability for systems with certain weak forms of hyperbolicity (see e.g. \cite{A0,AV02,A04,cv,AV17}).
In this work, we mainly study the statistical stability for a rich family of diffeomorphisms beyond uniform hyperbolicity. We complete this picture by finding a geometrical combinatorial structure, which is called skeleton. By establishing the one-to-one corresponding between physical measures and elements of the skeleton, we explore the statistical properties and the stability for physical measures via skeletons in geometrical viewpoint.

\smallskip
A diffeomorphism $f$ is \emph{partially hyperbolic} if there exists a $Df$-invariant splitting $TM=E\oplus F$ of the tangent bundle such that
\begin{itemize}
\item $TM=E\oplus F$ is a \emph{dominated splitting}: $\|Df|_{E(x)}\|\cdot \|Df^{-1}|_{F(f(x))}\|<1/2$ for every $x\in M$, rewrite $TM=E\oplus_{\succ} F$ to indicate this domination;
\smallskip
\item either $E$ is uniformly expanding or $F$ is uniformly contracting.
\end{itemize}
Throughout, we write $E^u$ and $E^s$ to indicate that they are uniformly expanding and uniformly contracting invariant sub-bundles, respectively.
%Physical measures have been largely studied for partially hyperbolic diffeomorphisms in recent decades (e.g. \cite{ABV00,bv,CY05, bdpp, AV15, MCY17, CMY18}). 
For $C^{1+}$ partially hyperbolic diffeomorphisms\footnote{We say that $f$ is a $C^{1+}$ diffeomorphism if it is $C^1$ and $Df$ is H\"{o}lder continuous.} exhibiting $E^u$,
Pesin and Sinai \cite{ps82} introduced the notion of \emph{Gibbs $u$-states}, an invariant measure $\mu$ is a Gibbs $u$-state when its conditional measures along \emph{strong} unstable manifolds are absolutely continuous with respect to Lebesgue measures along these manifolds. From \cite{L84}, we know that $\mu$ is a Gibbs $u$-state if and only if it satisfies the following entropy formula
\begin{equation}\label{guu}
h_{\mu}(f,\F^u)=\int \log |{\rm det}Df|_{E^u}|d\mu,
\end{equation}
where $h_{\mu}(f,\mathcal{F}^u)$ denotes the partial entropy along strong unstable foliation $\F^u$ tangent to $E^u$(see Definition \ref{dgu}).
It is worth noting that the above entropy formula can be defined for $C^1$ partially hyperbolic diffeomorphisms. Thus, as proposed by Croviser-Yang-Zhang \cite{CYZ18}, in this paper we call $\mu$ a Gibbs $u$-state for $f$ if (\ref{guu}) holds for $\mu$. 
%We recall that $\mu$ is said to be an \emph{SRB measure} of a $C^{1+\alpha}$ diffeomorphism if it has positive Lyapunov exponents almost everywhere; and its disintegration along Pesin unstable manifolds are absolutely continuous w.r.t. Lebesgue measures along these manifolds (see \cite{bp02} for Pesin theory). 
%We have the following well known facts: 
%\begin{itemize}
%\item for partially hyperbolic diffeomorphism exhibiting $E^u$, any SRB measure is a Gibbs $u$-states;
%\item every physical measure is a Gibbs $u$-states; 
%\item every ergodic hyperbolic SRB measure is a physical measure. 
%\end{itemize}
The set of Gibbs $u$-states plays a crucial role in the study of physical measures for partially hyperbolic diffeomorphisms (see e.g. \cite{bv,bdpp,CY05,VY13, AV15, AV17, CMY18}). 

\smallskip
In this paper, we shall mainly investigate the topological and measure-theoretic behaviors for an open class of partially hyperbolic diffeomorphisms
introduced in \cite{MCY17}.
Let $PH^1_{EC}(M)$ be the set of $C^1$ partially hyperbolic diffeomorphisms on $M$ satisfying that for every $f$ in it, there exists a partially hyperbolic splitting $TM=E^u\oplus_{\succ} E^{cu}\oplus_{\succ} E^{cs}$ for which
\begin{itemize}
\smallskip
\item\label{P1} $E^{cu}$ is \emph{mostly expanding}: $\mu$ has only positive Lyapunov exponents along $E^{cu}$ for every Gibbs $u$-state $\mu$;
\smallskip
\item\label{P2} $E^{cs}$ is \emph{mostly contracting}: $\mu$ has only negative Lyapunov exponents along $E^{cs}$ for every Gibbs $u$-state $\mu$.
\end{itemize}
%{\color{red}Unless otherwise stated, we say $f\in \mathcal{PH}^{1}_{\mathcal{EC}}(M)$ with partially hyperbolic splitting $TM=E^u\oplus E^{cu}\oplus E^{cs}$, meaning that the partially hyperbolic splitting must satisfy properties (P\ref{P1}) and (P\ref{P2}).}
%In order to study SRB measures, one needs diffeomorphisms in $\mathcal{PH}^1_{\mathcal{EC}}(M)$ to have higher regularity. 
Denote by $PH^{1+}_{EC}(M)$ the set of diffeomorphisms in $PH^1_{EC}(M)$ of class $C^{1+}$.
Let us list some properties for this kind of partially hyperbolic diffeomorphisms:
\begin{itemize}
\item Following the result of \cite[Theorem B]{yjg16}, we know that $PH^1_{EC}(M)$ is an open subset of $C^1$ diffeomorphisms.
\smallskip
\item $PH^{1+}_{EC}(M)$ contains both the diffeomorphisms with mostly expanding center and diffeomorphisms with mostly contracting center.
%\footnote{By using the Gibbs $u$-states in $C^1$ type, one can also define diffeomorphism with mostly expanding centers and diffeomorphisms with mostly contracting centers. However, some striking characters of these diffeomorphism will be lost},
Both of them have been studied from different aspects of interests recently (see e.g.\cite{bv, d00, cas02, cas04, And10, VY13, AV15, yjg16, dvy16,yjg19}).
\smallskip
\item By \cite[Theorem A]{MCY17}, any diffeomorphism of $PH^{1+}_{EC}(M)$ has finitely many ergodic physical measures, which has the basin covering property\footnote{The union of the basins of all physical measures has a full Lebesgue measure.}. 
%\footnote{Which means that Lebesgue almost every point of $M$ belongs to some basin of physical measures}.
\end{itemize}

\smallskip
%\subsection{Skeletons and physical measures}
In \cite{dvy16}, to study the structure of physical measures for diffeomorphisms with mostly contracting center, inspired by \cite{h11}, Dolgopyat-Viana-Yang introduced a topological structure, which is known as skeleton. More recently, Yang \cite{yjg19} slightly change the original definition of skeleton for diffeomorphisms with mostly expanding center. 
%Before \cite{yjg19}, P. Mehdipour-A. Tahzibi \cite{MT16} studied the number of SRB measures for endomorphisms by analyzing skeletons.
By constructing skeletons, they studied certain geometric structures of physical measures, which include the number for basins and supports of physical measures. Moreover, the study on stability of skeletons allow one can describe how physical measures bifurcate as the diffeomorphism changes under $C^1$-topology. 
The question naturally arises whether one can find skeletons for diffeomorphisms in $PH^{1}_{EC}(M)$ and use them to study the physical measures in geometrical viewpoint?

For diffeomorphisms in $PH^{1}_{EC}(M)$, due to the mixture behavior occurs on centers, we propose a new definition of the skeleton which is different from those in both \cite{dvy16} and \cite{yjg19}.

\begin{definition}\label{S}
Let $f$ be a diffeomorphism in $PH^{1}_{EC}(M)$ with dominated splitting $TM=E^u\oplus_{\succ} E^{cu}\oplus_{\succ} E^{cs}$. A set $\mathcal{S}$ consisting of finitely many hyperbolic periodic points of stable index ${\rm dim}E^{cs}$ is called a skeleton of $f$ if
\begin{enumerate}[(S1)]
\item\label{S1} for any $C^1$ disk $\gamma$ transverse to $E^{cs}$\footnote{Meaning that $T_x\gamma\oplus E^{cs}(x)=T_xM$ for every $x\in \gamma$.}, there exists $p\in \mathcal{S}$ such that 
$$
\gamma \pitchfork W^s({\rm Orb}(p,f))\neq \emptyset;
$$
\item\label{S2} $W^s({\rm Orb}(p,f))\pitchfork W^u({\rm Orb}(q,f))=\emptyset$ whenever $p,q\in \mathcal{S}$ with $p\neq q$.
\end{enumerate}
In addition, we say that $\mathcal{S}$ is a pre-skeleton of $f$ if it is a collection of finitely many hyperbolic periodic points of stable index ${\rm dim}E^{cs}$ and satisfies condition (S\ref{S1}).
\end{definition}

The main difference between the present definition and those in \cite{dvy16,yjg19} occurs on condition (S\ref{S1}). In mostly contracting case \cite{dvy16}, instead of (S\ref{S1}), there requires that every strong unstable manifold must intersect some stable manifold of periodic orbits of skeleton. The mostly expanding \cite{yjg19} case requires the denseness of the union of stable manifolds of periodic orbits from skeleton, which is weaker than condition (S\ref{S1}). 
%However, one can check that our definition of skeletons coincides with those in \cite{dvy16,yjg19} for diffeomorphisms with mostly contracting center and with mostly expanding center, respectively.}

In this paper, we show the existence of skeletons for diffeomorphisms in $PH^1_{EC}(M)$, and obtain a one-to-one correspondence between elements of a skeleton and physical measures for diffeomorphisms in $PH^{1+}_{EC}(M)$. 
Given a hyperbolic periodic point $p$ of a diffeomorphism $f$, recall that the \emph{homoclinic class} $H(p,f)$ of $p$ is the closure of the set of transversal intersection points between the stable and unstable manifolds of ${\rm Orb}(p,f)$.

\begin{theoremalph}\label{TheoB}
Every $f\in PH^1_{EC}(M)$ admits skeletons.
Moreover, if $f\in PH^{1+}_{EC}(M)$, denoted by $\P(f)=\{\mu_1,\cdots,\mu_{k}\}$ the finitely many physical measures of $f$. Then for any skeleton $\mathcal{S}(f)=\{p_1,\cdots, p_{\ell}\}$ of $f$, there exists a bijective map $i\mapsto j(i)$ such that for every physical measure $\mu_i\in \P(f)$ and the corresponding $p_{j(i)}\in \mathcal{S}(f)$, we have
$$
{\rm supp}(\mu_{i})=\overline{W^u({\rm Orb}(p_{j(i)},f))}=H(p_{j(i)},f).
$$
In particular, we have $k=\ell$, i.e., the number of physical measures of $f$ is equal to the cardinality of its skeleton.
\end{theoremalph}

Similar results were obtained for both mostly contracting case \cite[Theorem A]{dvy16} and mostly expanding case \cite[Theorem A]{yjg19}. Let us emphasize that the approaches of both works depend heavily on the uniform behaviors on either stable manifolds or unstable manifolds tangent to strong stable or unstable sub-bundles. In contrast to \cite{dvy16,yjg19}, for diffeomorphisms in $PH^1_{EC}(M)$, we have neither $E^{cu}=E^u$ nor $E^{cs}=E^{s}$, thus both $E^{cu}$ and $E^{cs}$ have non-uniformly behaviors. Therefore we use a different approach in the proof of Theorem \ref{TheoB}. 
The following ingredients play central roles in the construction of skeletons.
\begin{itemize}
\item Hyperbolic periodic points in skeletons will be created near Pesin blocks admitting large weight associated to Gibbs $u$-states.  
\smallskip
\item We show that the forward orbit of typical point enters to small neighborhoods of periodic points obtained above at \emph{infinitely many} hyperbolic times.
\end{itemize}

We mention that during the proof of Theorem \ref{TheoB}, by utilizing structures of Gibbs $cu$-states established in this paper, we provide a new proof on the existence and basin covering property of physical measures, which is different from the
original argument \cite{MCY17}.
%Among which, we would like to mention that for the mostly expanding case \cite{yjg19}, the existence of skeletons relies on \cite[Proposition 2.2]{yjg19}, which asserts that for Lebesgue almost every point of $x$, any limit measure of $\omega_{\M}(x,f)$ is a Gibbs $cu$-state. Theorem \ref{TheoA} guarantees that this full Lebesgue measure set can be found in any disk transverse to $E^{cs}$, with this improvement, one can construct skeletons with transversal property as in (S\ref{S1}), not only with denseness property on the union of stable manifolds. Observe that this denseness property on skeleton is inadequate in constructing one-to-one property of skeletons and physical measures, mainly due to the lack of uniformity on size of stable manifolds in our systems.
%
Let us remark that both of \cite{dvy16} and \cite{yjg19} stress that each closure of stable manifold of periodic orbits of skeleton coincides with the basin or essential basin of some physical measure. We do not know if it is possible to establish this property for diffeomorphisms in $PH^{1+}_{EC}(M)$. 

The idea of using skeletons or homoclinic classes to analyze typical measures, such as SRB measures, physical measures and maximal entropy measures appears in recent works (see e.g. \cite{MT16,BCS18,UVY19}).

%As we see in Theorem \ref{TheoB}, the cardinality of skeletons is equal to the number of SRB/physical measures. In general,  P. Mehdipour-A. Tahzibi \cite[Theorem 1.3]{MT16} has proved that for any  endomorphisms, the number of SRB measures is bounded above by the cardinality of skeletons (with the same index), though the skeleton defined there is slightly different to us as well.
%
%\begin{remark}
%Furthermore, they also obtained the one-to-one property of the closure of stable manifold of periodic orbits of skeletons and the basins or essential basins of physical/SRB measures. We do not know if it is possible to establish this one-to-one property for diffeomorphisms in $\mathcal{PH}^{1+}_{\mathcal{EC}}(M)$.
%\end{remark}

\smallskip
Since $PH^{1+}_{EC}(M)$ is $C^1$ open, 
%i.e., for every $f\in \mathcal{PH}^{1+}_{\mathcal{EC}}(M)$, there exists a sequence of diffeomorphisms $f_n\in\mathcal{PH}^{1+}_{\mathcal{EC}}(M)$ such that $f_n\to f$ in $C^1$-topology. In view of this, 
one would like to study how physical measures vary with respect to diffeomorphisms of $PH^{1+}_{EC}(M)$ under $C^1$-topology. Theorem \ref{TheoB} suggests that one can achieve this goal via analyzing the perturbed skeletons. More precisely, we have

%\begin{theoremalph}\label{TheC}
%Both the cardinality of skeleton and the number of physical measures depend upper semi-continuously in $C^1$- topology among $\mathcal{PH}^{1+}_{\mathcal{EC}}(M)$. 
%Moreover, for any $f\in \mathcal{PH}^{1+}_{\mathcal{EC}}(M)$, there exists a $C^1$ neighborhood $\V$ of $f$ such that the number of physical measures is locally constant and the physical measures vary continuously in the weak*-topology on a $C^1$ open and dense subset $\mathcal{V}_0\subset \V$. 
%
%In addition, restricted to any subset of $\mathcal{V}$ where the number of physical measures is constant, the supports of the physical measures depends lower semi-continuously with dynamics, in the sense of Hausdorff topology.
%\end{theoremalph}

\begin{theoremalph}\label{TheC}
Assume that $f\in PH^{1+}_{EC}(M)$ and $\mathcal{S}(f)=\{p_1,\cdots p_k\}$ is a skeleton of $f$. Then there exists a $C^1$ neighborhood $\V$ of $f$ such that for any $g\in \V$, the corresponding continuation $\mathcal{S}(g)=\{p_1(g),\cdots,p_k(g)\}$ is a pre-skeleton of $g$. If $g\in \V$ is of class $C^{1+}$, then the number of its physical measures is bounded above by the number of physical measures of $f$. Moreover, these two numbers coincide if and only if there exist no heteroclinic intersections between different elements of $\mathcal{S}(g)$. For this situation, every physical measure of $g$ is close to a physical of $f$, in weak$^*$-topology.

In addition, restricted to any subset of $\mathcal{V}$ where the number of physical measures is constant, the supports of the physical measures depends lower semi-continuously with dynamics, in the sense of Hausdorff topology.
\end{theoremalph}

By involving more technical details, we will prove Theorem \ref{TheC} in a more precisely manner in Section \ref{sec7}.
Theorem \ref{TheC} gives the upper semi-continuity of the number of physical measures and the cardinality of skeleton. As a consequence, these numbers are locally constant on an open and dense subset of diffeomorphisms in $PH^{1+}_{EC}(M)$. As we mentioned, \cite{dvy16} and \cite{yjg19} obtained similar results for diffeomorphisms with mostly contracting and diffeomorphisms with mostly expanding center. Note that \cite{AV17} showed similar result for diffeomorphisms with mostly expanding center under $C^{r}$($r>1$)-topology, in which they didn't construct the skeletons. 

\smallskip
%Let us recall the notions of (weakly) statistical stability.
Let $\F$ be a family of diffeomorphisms such that each $f\in \F$ admits finitely many physical measures, we say that $\F$ is \emph{weakly statistically stable} if for every $f\in \F$, for every sequence $f_n\in \F$ that converges to $f$ in $C^1$-topology, and for every sequence $\mu_n$ of physical measures of $f_n$, any limit measure of $\{\mu_n\}_{n\in \NN}$ is in the convex hull of finitely many physical measures for $f$. In particular, if $\F$ consists of diffeomorphisms that each of which admits a unique physical measure, then we say that $\F$ is \emph{statistically stable} if it is weakly statistically stable, or equivalently, the map $\F\ni f\longmapsto \mu_f$ is continuous, where $\mu_f$ denotes the unique physical measure of $f$.

\begin{theoremalph}\label{sta}
$PH^{1+}_{EC}(M)$ is weakly statistically stable. If $f\in PH^{1+}_{EC}(M)$ admits a unique physical measure, then there exists a $C^1$ neighborhood $\U$ of $f$ such that $\U\cap PH^{1+}_{EC}(M)$ is statistically stable.
%\begin{enumerate}
%\item there exists a $C^1$ open neighborhood $\U$ such that $\U\cap \mathcal{PH}^{1+}_{\mathcal{EC}}(M)$ is weakly statistically stable. 
%\item when $f$ admits a unique physical measure, then there exists a $C^1$ open neighborhood $\U$ such that $\U\cap \mathcal{PH}^{1+}_{\mathcal{EC}}(M)$ is statistically stable.
%\end{enumerate}
\end{theoremalph}

To the best of our knowledge, we give a partial list related to this result.
\begin{itemize}
\item Dolgopyat \cite{d00} proved the statistical stability for some diffeomorphisms with mostly contracting center on three dimensional manifold.
\smallskip
\item For diffeomorphisms with mostly contracting center, \cite{And10,dvy16} achieved the (weak)statistical stability.
\smallskip
\item Recently, \cite{AV17} and \cite{yjg19} demonstrated the (weak) statistical stability for diffeomorphisms with mostly expanding center. 
\end{itemize}
The method here is rather different to \cite{d00,And10,AV17}, but with the same spirit of \cite{dvy16,yjg19}. Indeed, we deduce Theorem \ref{sta} from Theorem \ref{TheC}. We point out \cite{AV17} used the transitivity assumption to guarantee the uniqueness of the physical measure. However, for diffeomorphisms admitting mostly contracting center, transitivity is insufficient to get the uniqueness of the physical measure, mainly due to the existence of intermingled basins\footnote{For two physical measures $\mu$ and $\nu$ with respect to $f$, we say that they are intermingled if ${\rm Leb}(U\cap B(\mu,f))>0$ and ${\rm Leb}(U\cap B(\nu,f))>0$ for any open set $U$ of $M$.}, which occurs even in robustly transitive settings. One can see \cite[$\S$ 11.1.2]{BDV05}, \cite{Kan,O15,CGS18} for this phenomenon. 

\bigskip
The remainder of this paper is organized as follows. 
Section \ref{sec2} consists of some notions and results used in this paper. 
In Section {\ref{sec3}}, we show the abundance of hyperbolic periodic points near Pesin blocks, these periodic points are candidates of the elements of skeletons. 
Section \ref{sec4} is devoted to study the structure of the set of Gibbs $cu$-states, such as compactness and the upper semi-continuity with respect to diffeomorphisms. 
%More importantly, we prove that for any $C^1$ disk $D$ transverse to $E^{cs}$, any limit measure of $\{\frac{1}{n}\sum_{i=0}^{n-1}\delta_{f^ix}\}_{n\in \NN}$ is a Gibbs $cu$-state. 
Some preliminary result on properties of skeletons are given in Section \ref{sec5}, which will be used in next two sections. 
%Let us remake that all dynamics we studied from Section {\ref{sec3}} to Section \ref{sec5} belong to $PH^1_{EC}(M)$.
In Section \ref{sec6}, we introduce the set $PH^{1}_{C}(M)$ of partially hyperbolic diffeomorphims for which
every diffeomorphism in $PH^{1}_{EC}(M)$ has its iterates in $PH^{1}_{C}(M)$. 
For diffeomorphisms in $PH^{1}_{C}(M)$, we obtain the existence of skeletons, and the one-to-one corresponding between physical measures and periodic points of the skeleton (Theorem \ref{TheoF}).
In Section \ref{sec7}, going back to $PH^{1}_{EC}(M)$, we give the proofs of Theorems \ref{TheoB}, \ref{TheC} and \ref{sta}. 
We provide a detailed proof of Proposition \ref{t11} in Appendix \ref{appendix}, which is stated in Section {\ref{sec3}} and used almost throughout this paper. 

\bigskip
{\bf{Acknowledgements.}} We are grateful to D. Yang and R. Zou for their suggestions and discussions.

\section{Preliminary}\label{sec2}\label{sec2}
Throughout, let $M$ be a compact Riemannian manifold with metric $d$. Let $B(x,r)$ be the $r$-ball of $x\in M$ defined by $B(x,r)=\{y\in M: d(x,y)<r\}$. Given a smooth embedded sub-manifold $D$, let $d_D(x,y)$ be the distance from $x$ to $y$ along $D$; and for $\delta>0$, define $B_D(x,\delta)=\{y\in D: d_D(x,y)<\delta\}$. We use ${\rm Leb}_{D}$ to represent the induced Lebesgue measure on $D$. We say that $D$ is an \emph{unstable disk} if it is a disk inside an unstable manifold.
 
Denote by $\M$ the space of Borel probability measures supported on $M$.
Let us fix a metric compatible with the weak*-topology on $\M$ as follows: for $\mu$ and $\nu$ in $\M$, define
$$
{\rm dist}(\mu,\nu)=\sum_{n\in \NN}\frac{|\int \varphi_nd\mu- \int \varphi_n d\nu|}{2^n\sup_{x\in M}|\varphi_n(x)|},
$$
where $\{\varphi_n\}_{n\in \NN}$ is a countable dense subset of the space of continuous functions on $M$.
For a homoemorphism $f$ on $M$, let $\mathcal{M}(f)$ be the set of $f$-invariant measures supported on $M$. 

Let $f$ be a diffeomorphism and $E$ a $Df$-invariant sub-bundle, for every $\mu \in \M(f)$, define
$$
\chi_{\mu}(f,E)=\int \log \|Df|_{E(x)}\|d\mu(x).
$$

\subsection{Plaque families for diffeomorphisms with dominated splitting}\label{dd}
Let $f$ be a $C^1$ diffeomorphism with dominated splitting $TM=E\oplus F$.
%It is known that the domination is robust in the $C^1$-topology. More precisely, there exists a $C^1$-open neighborhood $\U$ of $f$ so that every $g\in \U$ admits the dominated splitting 
%$
%TM=E_g(x)\oplus_{\succ} F_g(x).
%$
%Note that both $E_g(x)$ and $F_g(x)$ depend continuously on $g$ and $x$. In particular, $E_f(x)=E(x)$ and $F_f(x)=F(x)$ for every $x\in M$.
Given $\theta > 0$, define the \emph{cone field} $\mathcal{C}_{\theta}^{E}=\{\mathcal{C}_{\theta}^{E}(x): x\in M\}$ associated to $E$ of width $\theta$ by
$$
\mathcal{C}_{\theta}^E(x)=\left\{v=v^E+v^F\in E(x)\oplus F(x): \|v^F\|\le {\theta} \|v^E\|\right\}
$$
for every $x\in M$. By symmetry, we can define the cone field $\mathcal{C}_{\theta}^F$ associated to $F$.
We say that a smooth embedded sub-manifold $D$ is \emph{tangent to $\mathcal{C}_{\theta}^E$} if
${\rm dim}D = {\rm dim}E$ and $T_xD\subset \mathcal{C}_{\theta}^E(x)$ for every $x\in D$. In particular, one says that $D$ is tangent to $E$ if $T_xD=E(x)$ for every $x\in D$.
%The domination property implies the following result.

%\begin{lemma}\label{cone}
%Given $a>0$, and a $C^1$ diffeomorphism $f$ with dominated splitting $TM=E\oplus_{\succ}F$. Then for any $C^1$ disk $D$ transverse to $F$, there exists $n_0\in \NN$ such that $f^n(D)$ is tangent to $\mathcal{C}_a^E$ whenever $n\ge n_0$.
%\end{lemma}

%Let us take $\delta_M>0$ sufficiently small such that for any $x\in M$ the inverse of
%exponential map ${\rm exp}_x^{-1}$ is well defined on the $\delta_M$ open neighborhood of $x$. We then identify this neighborhood with the corresponding neighborhood on tangent space. 
%Let $f$ be a diffeomorphism of class $C^{1+}$ with dominated splitting $TM=E\oplus_{\succ}F$, given $\xi\in (0,1)$ and a $C^1$ disk $D$ tangent to $\mathcal{C}_a^{E}$, define its \emph{H\"{o}lder curvature} w.r.t. $\xi$ as 
%$$
%\mathscr{K}_{\xi}(D)=\inf\{C>0: TD~ \textrm{is}~ (C,\xi)-\textrm{H\"{o}lder continuous}\}, 
%$$
%where we say $TD$ is $(C,\xi)$-H\"{o}lder if $dist(T_xD,T_yD)\le Cd(x,y)^{\xi}$ for all $x,y\in B(x,\delta_M)\cap D$ and $x\in D$.
%Let us recall the following fact.
%\begin{lemma}\cite[Corollary 4.2]{ABV00}\label{cur}
%Let $f$ be a $C^{1+}$ diffeomorphism with dominated splitting $TM=E\oplus_{\succ}F$.
%There exist $\xi>0$ and $C>0$ such that for any $C^2$ disk $D$ transverse to $F$, there is $n_0\in \NN$ such that 
%$\mathscr{K}_{\xi}(f^n(D))\le C$ for every $n\ge n_0$.
%\end{lemma}
%

It is known that the domination is robust in the $C^1$-topology. More precisely, if $C^1$ diffeomorphism $f$ admits the dominated splitting $TM=E\oplus F$, then there exists a $C^1$ neighborhood $\U$ of $f$ so that every $g\in \U$ admits the dominated splitting 
$
TM=E_g\oplus F_g\footnote{$E_g(x)$ and $F_g(x)$ depend continuously on $g$ and $x$. In particular, $E_f(x)=E(x)$ and $F_f(x)=F(x)$ for every $x\in M$. Sometimes, we drop the reference to $g$ for simplicity.}.
$
Denote by $E_g(x,\rho)$ the $\rho$-neighborhood of $0_x$ in $E_g(x)$, define $F_g(x,\rho)$ similarly. For diffeomorphisms with dominated splitting, one has the Plaque family theorem stated as follows.

\begin{proposition}\cite[Theorem 5.5]{HPS77}\label{plaque}
Let $f$ be a $C^1$ diffeomorphism with dominated splitting $TM=E\oplus F$. Then for every $\theta>0$, there exists a $C^1$ open neighborhood $\U$ of $f$ and constants $0<\delta<\rho$ such that for any $g\in \U$, there is a continuous family of $C^1$ embedded sub-manifolds $\{\mathcal{F}^E_{\rho} (g,x)\}_{x\in M}$ 
given by
$$
\mathcal{F}^E_{\rho} (g,x) = \exp_x\{(u, \phi_x(u)) : u\in  E_g(x,\rho)\}
$$
where $\phi_x: E_g(x,\rho)\to F_g(x)$ is a $C^1$ map such that $\phi_x(0_x)=0_x$ and $D\phi_x(0_x)=0$. Similarly, one can get the continuous family of $C^1$ embedded sub-manifolds $\{\mathcal{F}^F_{\rho} (g,x)\}_{x\in M}$.
For every $\ast \in \{E,F\}$ and $x\in M$ one has
\begin{itemize}
\item $\F_{\rho}^{\ast}(g,x)$ is tangent to $\mathcal{C}_{\theta}^{\ast}(x)$; 
\item $g(\F_{\delta}^{\ast}(g,x))\subset \F_{\rho}^{\ast}(g,g(x))$ and $g^{-1}(\F_{\delta}^{\ast}(g,x))\subset \F_{\rho}^{\ast}(g,g^{-1}(x))$;
\item for $g_n\xrightarrow{C^1} g$ in $\U$ and $x_n\to x$, we have $\F_{\rho}^{\ast}(g_n,x_n)\to \F_{\rho}^{\ast}(g,x)$.
\end{itemize}
\end{proposition}

\subsection{Hyperbolic time and Liao-Gan's shadowing Lemma}

For $A\subset \NN$, define the \emph{lower-density} $\mathcal{D}_{L}(A)$ and the \emph{upper-density} $\mathcal{D}_{U}(A)$ of $A$ as follows:
$$
\mathcal{D}_{L}(A):=\liminf_{n\to +\infty}\frac{1}{n}\#\{A\cap [1,n]\},\quad \mathcal{D}_{U}(A):=\limsup_{n\to +\infty}\frac{1}{n}\#\{A\cap [1,n]\}.
$$
%and define the \emph{upper-density of $\JJ$} as
%$$
%\mathfrak{D}_{\mathcal{U}}(\mathbb{J}):=\limsup_{n\to +\infty}\frac{1}{n}\#\{\mathbb{J}\cap [1,n]\}.
%$$

Now we give the definition of hyperbolic times associated to some invariant sub-bundles.
\begin{definition}
Let $f$ be a $C^1$ diffeomorphism with a $Df$-invariant sub-bundle $E$. Given $\sigma\in (0,1)$, we call $n\in \NN$ a $(E,\sigma)$-hyperbolic time of $x\in M$ if
$$
\prod_{i=n-k+1}^{n}\|Df^{-1}|_{E(f^i(x))}\|<\sigma^k \quad \textrm{for}~ k=1,\cdots, n.
$$
\end{definition}

For diffeomorphisms with dominated splitting, one has the backward contracting property for disks tangent to the cone field at hyperbolic times, see \cite[Lemma 2.7]{ABV00} for a proof.
\begin{lemma}\label{bcd}
Let $f$ be a $C^1$ diffeomorphism admitting the dominated splitting $TM=E\oplus_{\succ} F$. 
Given $0<\sigma_0<\sigma_1<1$, there exist $a>0, \delta>0$ such that, for any $C^1$ embedded disk $D$ tangent to $\mathcal{C}^E_a$ that contains $x$, if $n$ is a $(E, \sigma_0)$-hyperbolic time of $x$, then
$$
d_{f^i(D)}(f^{i}(y), f^{i}(z))\le \sigma_1^{(n-i)}d_{f^n(D)}(f^n(y), f^{n}(z))
$$
for any $y,z$ in $D$ such that $f^n(y), f^n(z)\in B_{f^n(D)}(f^n(x), \delta)$. 
\end{lemma}

%Recall the classical Pliss Lemma, see \cite{pliss} and \cite[Lemma 3.1]{ABV00}.
%\begin{lemma}\label{Lem:numberPliss}
%For any $C_0\geq C_1>C_2\geq0,$ there is $\theta=\theta(C_0,C_1,C_2)>0$ such that for any $N\in \NN$ and real numbers $\{a_i\}_{i=1}^N$ satisfying
%$$
%\sum_{j=1}^N a_j\geq C_1N, ~~~a_j\leq C_0,~~~\forall~ 1\le j\le N,
%$$
%then there is an integer $\ell>\theta N$ and a subsequence $\{n_i: 1\le i \le \ell\}\subset \{1,\cdots,N\}$ such that for every $1\le i \le \ell$,
%$$
%\sum_{j=n+1}^{n_i} a_j\geq C_2(n_i-n)\quad \textrm{for every}~0 \leq n<n_i.
%$$
%\end{lemma}
%
As an application of Pliss Lemma (see  \cite{pliss} and \cite[Lemma 3.1]{ABV00}), one has

\begin{proposition}\cite[Corollary 3.2]{ABV00}\label{hh}
Suppose $\alpha_0>\alpha>0$, and let $f$ be a $C^1$ diffeomorphism admitting dominated splitting $TM=E\oplus F$, then there exists $\theta=\theta(\alpha_0,\alpha,f)\in (0,1)$ and a $C^1$ neighborhood $\U$ of $f$ such that for any $g\in \U$, 
if $x$ is $\alpha_0$-nonuniformly expanding along $E_g$ , i.e.,
$$
\limsup_{n\to +\infty}\frac{1}{n}\sum_{i=1}^n \log \|Dg^{-1}|_{E_g(g^i(x))}\|<-\alpha_0,
$$
then there exist infinitely many $(E_g, {\rm e}^{-\alpha})$-hyperbolic times $n_1<n_2<\cdots $ of $x$ with 
lower density larger than $\theta$.
\end{proposition}

Let us recall the Liao-Gan's shadowing lemma \cite{L89,G02}.

\begin{lemma}\label{Gan}
Let $f$ be a $C^1$ diffeomorphism with dominated splitting $TM=E \oplus_{\succ} F$. For $\lambda \in (0,1)$, there exists  $\delta_0>0$ and $L_0>0$ such that for any orbit segment $(x,f(x),\cdots,f^n(x))$ with following properties:
\begin{itemize}
\item $$d(x,f^n(x))\le \delta_0;$$
\item for any $1\le k \le n$, we have
$$\prod_{i=0}^{k-1}\|Df^{-1}|_{E(f^{n-i}(x))}\|<\lambda^k,\quad \prod_{i=0}^{k-1}\|Df|_{F(f^{i}(x))}\|<\lambda^k,$$
\end{itemize}
there exists a hyperbolic periodic point $p$ with stable index ${\rm dim}F$ such that
$$
d(f^i(x), f^i(p))\le L_0 d(x,f^n(x)) \quad \textrm{for any} ~~0\le i \le n-1.
$$ 
%
%such that
%\begin{itemize}
%\item $p$ is hyperbolic with stable index ${\rm dim}F$ and unstable index ${\rm dim}E$;
%\smallskip
%\item $d(f^i(x), f^i(p))\le L_0 d(x,f^n(x))$ for any $0\le i \le n-1$;
%\smallskip
%\item there exists $\rho>0$ depending only on $\lambda$ such that the local stable (unstable) manifold of $p$ contains a disk with radius $\rho$.
%\end{itemize}
\end{lemma}

\subsection{Entropy along unstable foliations}
%In this subsection, we focus on the entropy of an invariant measure along unstable foliations, which has become a powerful tool in the study of partially hyperbolic diffeomorphisms in recent years, see \cite{CYZ18,HHW,hyy18,yjg19} for instance. 
Let $f$ be a $C^1$ partially hyperbolic diffeomorphism with $E^u$, denote by $\F^u$ the unstable foliation whose leaves are strong unstable manifolds tangent to $E^u$. 

Let $\mu$ be a Borel probability. A measurable partition $\xi$ is \emph{$\mu$-subordinate to} $\mathcal{F}^u$, if for $\mu$-almost every $x$, one has
\begin{itemize}
\item $\xi(x)\subset \mathcal{F}^u(x)$, where $\xi(x)$ is the element of $\xi$ containing $x$;
\smallskip
\item $\xi(x)$ contains an open neighborhood of $x$ in $\mathcal{F}^u$;
\smallskip
\item $\xi(f(x))\subset f(\xi(x))$.
\end{itemize}

Note that the existence of measurable partitions $\mu$-subordinate to unstable foliation was given by
\cite[Lemma 3.1]{LS82} and \cite[Lemma 3.2]{yjg16}. Moreover, we have the following result which is contained in \cite[Lemma 3.1.2]{ly}.

\begin{lemma}\label{a}
For any $\mu \in \M(f)$ and any measurable partitions $\xi_1$ and $\xi_2$  that are $\mu$-subordinate to $\mathcal{F}^u$, we have
$$
h_{\mu}(f,\xi_1)=h_{\mu}(f,\xi_2).
$$
\end{lemma}

By Lemma \ref{a}, one can define the entropy along $\F^u$ as follows:
\begin{definition}\label{dgu}
For every $\mu\in \M(f)$, the entropy of $\mu$ along $\mathcal{F}^u$ is defined by
$$
h_{\mu}(f,\mathcal{F}^u)=h_{\mu}(f,\xi),
$$
where $\xi$ is any measurable partition $\mu$-subordinated to $\mathcal{F}^u$.
\end{definition}
The entropy $h_{\mu}(f,\mathcal{F}^u)$ has become a powerful tool in the study of partially hyperbolic diffeomorphisms in recent years, see \cite{CYZ18,HHW,hyy18,yjg19} for instance.

%For $C^{1+\alpha}$ partially hyperbolic diffeomorphisms,
%Pesin and Sinai \cite{ps82} introduced the notion of Gibbs $u$-states, which requires that their conditional measures w.r.t measurable partitions  have the absolutely continuous property along strong unstable manifolds. This property is known ( see \cite{L84} for instance ) to be equivalent to the following equality
%\begin{equation}\label{eg}
%h_{\mu}(f,\F^u)=\int \log |{\rm det}Df|_{E^u}|d\mu.
%\end{equation}
%It is worth noting that the equality above can be defined for $C^1$ partially hyperbolic diffeomorphisms. Thus
%
%
%
\begin{definition}\label{gu}
Let $f$ be a $C^1$ partially hyperbolic diffeomorphism with $E^u$. We call $\mu\in \M(f)$ a Gibbs $u$-state if
$$
h_{\mu}(f,\F^u)=\int \log |{\rm det}Df|_{E^u}|d\mu.
$$
\end{definition}
%\begin{lemma}\cite[Theorem A]{yjg16}\label{gu}
%Let $f_n$ be a sequence of $C^1$ diffeomorphisms that converges to $f$ in $C^1$-topology, and assume that $\mu_n$ be a sequence of $f_n$-invariant measures which converges to the $f$-invariant measure $\mu$, then we have
%$$
%\limsup_{n\to+\infty}h_{\mu_n}(f_n, \F_n^u)\le h_{\mu}(f,\F^u),
%$$
%where $\F_n^u$ is the unstable foliation of $f_n$.
%\end{lemma}

%\begin{lemma}\cite[Theorem A]{wwz18}
%Let $f$ be a $C^1$ partially hyperbolic diffeomorphism with strong unstable direction $E^u$. Then any $f$-invariant measure $\mu$ satisfy 
%$$
%h_{\mu}(f,\F^u)\le \int \log |{\rm det}Df|_{E^u}|d\mu.
%$$ 
%\end{lemma}

%For $C^{1+\alpha}$ partially hyperbolic diffeomorphisms,
%Pesin and Sinai \cite{ps82} introduced the notion of Gibbs $u$-states, which requires that they have the absolutely continuous property along strong unstable manifolds. This property is known ( see \cite{L84} for instance ) to be equivalent to the following equality
%\begin{equation}\label{eg}
%h_{\mu}(f,\F^u)=\int \log |{\rm det}Df|_{E^u}|d\mu.
%\end{equation}
%It is worth noting that the equality above can be defined for $C^1$ partially hyperbolic diffeomorphisms. Thus, as proposed by Croviser-Yang-Zhang, in this paper we say $\mu$ is a Gibbs $u$-state for $f$ it satisfies (\ref{eg}). 
As we have mentioned, if $f$ is $C^{1+}$, then $\mu$ is a Gibbs $u$-state if and only its conditional measures along strong unstable manifolds are absolutely continuous w.r.t. Lebesgue measures on these manifolds. Denote by $G^u(f)$ the space of Gibbs $u$-states of $f$. 
\begin{lemma}\label{gibbsp}\cite[Corollary 2.15]{CYZ18}
Let $f$ be a $C^1$ partially hyperbolic diffeomorphism with $E^u$. Then
\begin{enumerate}
\item \label{up1} $G^u(f)$ is nonempty, compact and convex subset of $\M(f)$.
\smallskip
\item \label{up2}Every ergodic component of any $\mu \in G^u(f)$ is a Gibbs $u$-state.
\end{enumerate}
\end{lemma}

Moreover, we have the upper semi-continuity of the map $f\mapsto G^u(f)$.
\begin{lemma}\cite[Proposition 3.5]{hyy18}\label{upp}
Let $\{f_n\}_{n\in \NN}$ and $f$ be $C^{1}$ partially hyperbolic diffeomorphisms exhibiting $E^u$, and assume that $f_n$ converges to $f$ in $C^1$-topology, then
$$
\limsup_{n\to +\infty}G^u(f_n)\subset G^u(f).
$$ 
\end{lemma}

\section{Abundance of hyperbolic periodic points}\label{sec3}
Let $f$ be a $C^1$ diffeomorphism with dominated splitting $TM=E\oplus_{\succ} F$.  
For $\ell\in \NN, \alpha>0$, we introduce the following \emph{Pesin block}:
\begin{align*}
& \quad \Lambda_{\ell}(f,\alpha,E,F)=\big\{x: \frac{1}{n\ell}\sum_{i=0}^{n-1}\log \|Df^{-\ell}|_{E(f^{-i\ell}(x))}\|\le -\alpha,\\
 &\quad \quad \quad\quad~\frac{1}{n\ell}\sum_{i=0}^{n-1}\log \|Df^{\ell}|_{F(f^{i\ell}(x))}\|\le -\alpha, \quad \forall n\in \NN\big\}.
\end{align*}
%It can be written as a intersection of two blocks $\mathcal{H}_{yp}^{-}(g,\ell, \alpha)$ and $\mathcal{H}_{yp}^{+}(g,\ell, \alpha)$ defined as follows:
%$$
%\mathcal{H}_{yp}^{-}(f,\ell, \alpha)=\big\{x: \frac{1}{n\ell}\sum_{i=0}^{n-1}\log \|Df^{-\ell}|_{E(f^{-i\ell}(x))}\|\le -\alpha, \quad \forall n\in \NN\big\};
%$$
%$$
%\mathcal{H}_{yp}^{+}(f,\ell, \alpha)=\big\{x: \frac{1}{n\ell}\sum_{i=0}^{n-1}\log \|Df^{\ell}|_{F(f^{i\ell}(x))}\|\le -\alpha,\quad \forall n\in \NN\big\}.
%$$
\begin{remark}
Note that the Pesin block defined here is different from the classical one in \emph{Pesin theory} (see \cite{bp07}).
%all points of $\Lambda_{\ell}(f,\alpha,E,F)$ have Lyapunov exponents bounded by $-\alpha$, which is independent of $\ell$. 
\end{remark}
%\begin{claim}
%For every $\alpha>0$, $\ell\in \NN$, and $k\in \NN$, there exist $\alpha'>0$, $\ell'\in \NN$ such that $
%f^k(\Lambda_{\ell}(f,\alpha,E,F))\subset \BB_{\ell'}(f,\alpha',E,F)$.
%\end{claim}

As usual, all points of a Pesin block admit the stable and unstable manifolds with uniform size. 
Moreover, this size can be made to be uniform in a neighborhood of the fixed diffeomorphism.

\begin{lemma}\label{size}
Assume that $f$ is a $C^1$ diffeomorphism with dominated splitting $TM=E\oplus_{\succ} F$. 

Then for every $\alpha>0$ and $\ell\in \NN$, there exists a $C^1$ neighborhood $\U$ of $f$ and constants $\tau:=\tau(\alpha,\ell)\in (0,1)$, $C:=C(\alpha,\ell)>0$, $\delta:=\delta(\ell,\alpha)>0$ such that for every $g\in \U$, for every $x\in \Lambda_{\ell}(g,\alpha,E,F)$, there are local stable manifold $W_{loc}^s(x,g)$ and local unstable manifold $W_{loc}^u(x,g)$, which are $C^1$ embedded disks of radius $\delta$ centered at $x$ with following properties: 
\begin{itemize}
\smallskip
\item $W_{loc}^s(x,g)$ is tangent to $E$ and $W_{loc}^u(x,g)$ is tangent to $F$;
\smallskip
\item for every $n\in \NN$, we have
\begin{itemize}
\smallskip
\item[--] $d(g^{n}(y), g^{n}(z))\le C\tau^nd(y,z)$ for every $y,z \in W_{loc}^s(x,g)$,
\smallskip
\item[--] $d(g^{-n}(y), g^{-n}(z))\le C\tau^{n}d(y,z)$ for every $y,z \in W_{loc}^u(x,g)$.
\end{itemize}
\end{itemize}
\end{lemma}

It can be deduced from Plaque family theorem (Proposition \ref{plaque}), note that the local stable and unstable manifolds are contained in the corresponding plaques, see \cite[Theorem 3.5]{AV17} for a detailed proof. 
The radius $\delta$ is also called the size of stable(unstable) manifolds.
For the rest, we use $R(W_{loc}^s(x,g))$ and $R(W_{loc}^u(x,g))$ to denote the sizes of $W_{loc}^s(x,g)$ and $W_{loc}^u(x,g)$, respectively.
%It is worth noting that this lemma can be also deduced by the Plaque family theorem as stated in Proposition \ref{plaque}, and one has that the local stable and unstable manifolds are contained in the corresponding plaques. 

\smallskip
By Lemma \ref{Gan}, we have the following result which demonstrates the existence of hyperbolic periodic points near Pesin blocks.

\begin{lemma}\label{sp}
Let $f$ be a $C^1$ diffeomorphism with dominated splitting $TM=E\oplus_{\succ} F$. Then for any $\alpha>0$ and $\ell\in \NN$, there exist $\rho_{\alpha,\ell}>0$, $\delta_{\alpha,\ell}>0$ and $L_{\alpha,\ell}>0$ such that if 
$$
\mu\left(B(x,\rho)\cap \Lambda_{\ell}(f,\alpha,E,F)\right)>0,
$$
for some $x\in M$, $\mu\in \M(f)$ and $\rho \le \rho_{\alpha,\ell}$,
then there exists a hyperbolic periodic point $p\in B(x, L_{\alpha,\ell}\cdot\rho)$ of stable index ${\rm dim}F$ such that 
$$
R(W^{s}_{loc}(p,f))\ge \delta_{\alpha,\ell},\quad R(W^{u}_{loc}(p,f))\ge \delta_{\alpha,\ell}.
$$
\end{lemma}

\begin{proof}
Writing $\Lambda_{\ell}:=\Lambda_{\ell}(f,\alpha,E,F)$ for simplicity. Given $n\in \NN$, assume that $y$ and $f^n(y)$ are contained in $\Lambda_{\ell}$, then from the definition of $\Lambda_{\ell}$ one obtains that for every $1\le j \le n$,
\begin{equation}\label{cx}
\prod_{i=0}^{j-1}\left\|Df^{\ell}|_{F(f^{i\ell}(y))}\right\|\le {\rm e}^{-\alpha j\ell} \quad \textrm{and}\quad\prod_{i=0}^{j-1}\left\|Df^{-\ell}|_{E(f^{n-i\ell}(y))}\right\|\le {\rm e}^{-\alpha j\ell}.
\end{equation}

\smallskip
By Lemma \ref{Gan}, there exist $\delta_{0}>0$ and $L_{0}>0$ such that when $d(y,f^n(y))\le \delta_0$, then there exists a hyperbolic periodic point $p$ of periodic $n$ with stable index ${\rm dim}F$, which satisfies
\begin{equation}\label{yy}
d(f^i(y), f^i(p))\le L_{0}d(y, f^n(y)) \quad \textrm{for every}~1\le i \le n.
\end{equation}

Now let us take $\rho_{\alpha,\ell}=\delta_0/2$ and $L_{\alpha,\ell}=2L_0+1$.
For every $0<\rho\le \rho_{\alpha,\ell}$, assume $\mu(B(x,\rho)\cap \Lambda_{\ell})>0$ for some $x\in M$ and $\mu\in \mathcal{M}(f)$. 
By Poincare's recurrent theorem \cite[Theorem 1.4]{wa82}, $\mu$-almost every point of $B(x,\rho)\cap \Lambda_{\ell}$ is recurrent, therefore one can take $y, f^n(y)\in B(x,\rho)\cap \Lambda_{\ell}$ for infinitely many $n$, and thus we have $d(y, f^n(y))<2\rho\le \delta_0$. Consequently, there exists a hyperbolic periodic point $p$ with stable index ${\rm dim}F$, which satisfies (\ref{yy}) and $p\in B(x,L_{\alpha,\ell}\cdot\rho)$. 
%In particular, we have the following distance estimates:
%\begin{itemize}
%\smallskip
%\item[--] $d(p,y)\le 2\rho L_0\le 2 \rho_{\alpha,\ell}$;
%\smallskip
%\item[--] $d(p,x)\le d(p,y)+d(y,x)\le 2\rho L_0+\rho= L_{\alpha,\ell}\cdot\rho.$
%\end{itemize}
Since $y$ exhibits property (\ref{cx}), one can take $\rho_{\alpha,\ell}$ small enough such that for every $j\ge 1$,
\begin{equation*}
\prod_{i=0}^{j-1}\left\|Df^{\ell}|_{F(f^{i\ell}(p))}\right\|\le {\rm e}^{-\frac{\alpha}{2} j\ell} \quad \textrm{and}\quad\prod_{i=0}^{j-1}\left\|Df^{-\ell}|_{E(f^{-i\ell}(p))}\right\|\le {\rm e}^{-\frac{\alpha}{2} j\ell}.
\end{equation*}
Therefore, $p\in \Lambda_{\ell}(f,\alpha/2,E,F)$. According to Lemma \ref{size}, there exists $\delta_{\alpha,\ell}>0$ such that the stable and unstable manifolds of $p$ have size larger than $\delta_{\alpha,\ell}$.
\end{proof}

When the invariant measure is ergodic, we have the following result which will be used for several times.
\begin{lemma}\label{sp1}
Let $f$ be a $C^1$ diffeomorphism with dominated splitting $TM=E\oplus_{\succ} F$. Then for any $\alpha,\alpha'>0$ and $\ell,\ell'\in \NN$, there exist $\delta:=\delta(\alpha,\alpha',\ell,\ell')>0$, $\rho:=\rho(\alpha,\alpha',\ell,\ell')>0$ and $L:=L(\alpha,\alpha',\ell,\ell')>0$ such that if
\begin{equation}\label{lg}
\mu\big(B(x,r)\cap \Lambda_{\ell}(f,\alpha,E,F)\big)>0,\quad \mu\big(B(y,r)\cap \Lambda_{\ell'}(f,\alpha',E,F)\big)>0,
\end{equation}
for some $f$-ergodic measure $\mu$, points $x,y\in M$ and $r\le \rho$, then there exists a hyperbolic periodic point $p\in B(x,Lr)$ such that $f^m(p)\in B(y,Lr)$ for some $m\in \NN$, and both of them admit local stable and unstable manifolds of size larger than $\delta$.
\end{lemma}

\begin{proof}
By Lemma \ref{Gan}, there exist $L_0>0$ and $\rho>0$ such that for any $r\le \rho$, for any $z\in B(x,r)\cap \Lambda_{\ell}(f,\alpha,E,F)$ with $f^n(z)\in B(x,r)\cap \Lambda_{\ell}(f,\alpha,E,F)$, there exists a hyperbolic periodic point $p$ of periodic $n$ such that  
\begin{equation}\label{shad}
d(f^k(p), f^k(z))\le L_0 d(z,f^n(z))\quad \textrm{for every}~1\le k \le n.
\end{equation}

\smallskip
Given $r\le \rho$, $f$-ergodic measure $\mu$ and points $x,y\in M$ satisfying (\ref{lg}).
By Poincare's recurrent theorem, $\mu$-almost every point of $B(x,r)\cap \Lambda_{\ell}(f,\alpha,E,F)$ is recurrent. Moreover, since $\mu(B(y,r)\cap \Lambda_{\ell'}(f,\alpha',E,F))>0$ and $\mu$ is ergodic, the Birkhoff's ergodic theorem \cite[Theorem 1.14]{wa82} implies that for $\mu$-almost point point, its forward orbit enters to $B(y,r)\cap \Lambda_{\ell'}(f,\alpha',E,F)$ for infinitely many times. Therefore, one can take a recurrent point $z$ of $B(x,r)\cap \Lambda_{\ell}(f,\alpha,E,F)$ with some integers $m<n$ such that
\begin{itemize}
\smallskip
\item[--] $z, f^n(z)\in B(x,r)\cap \Lambda_{\ell}(f,\alpha,E,F)$;
\smallskip
\item[--] $f^m(z)\in B(y,r)\cap \Lambda_{\ell'}(f,\alpha',E,F)$.
\end{itemize}
Hence, one can find a hyperbolic periodic point $p$ of periodic $n$ with property (\ref{shad}). Furthermore, by a simple computation, one can take $L>L_0$ such that $p\in B(x,Lr)$ and $f^m(p)\in B(y, Lr)$.
By definition of Pesin blocks, as showed in Lemma \ref{sp}, by reducing $\rho$ if necessary, one can take $\delta>0$ such that both $p$ and $f^m(p)$ admit local stable and unstable manifolds of size larger than $\delta$.
\end{proof}

\smallskip
We are interested in diffeomorphisms in $PH^1_{EC}(M)$, thus for every $f\in PH^1_{EC}(M)$ with partially hyperbolic splitting $TM=E^u\oplus_{\succ}E^{cu}\oplus_{\succ}E^{cs}$, $\alpha>0$ and $\ell \in \NN$, we can rewrite $E^{wu}:=E^u\oplus_{\succ}E^{cu}$ and $\Lambda_{\ell}(f,\alpha):=\Lambda_{\ell}(f,\alpha,E^{wu}, E^{cs})$ for simplicity. 

The result below may be seen as a $C^1$ version of \cite[Lemma 3.4]{AV17}, though we are dealing with a more general system  than \cite{AV17}. For completeness, a proof is provided in Appendix \ref{appendix}.
\begin{proposition}\label{t11}
For every $f\in PH^1_{EC}(M)$, there exists $\alpha>0$ such that for every $\varepsilon>0$, there exists $\ell \in \mathbb{N}$ and a $C^1$ neighborhood $\mathcal{U}$ of $f$, such that for every $g\in \mathcal{U}$, every $\mu\in G^u(f)$, it has
$$
\mu\left(\Lambda_{\ell}(g,\alpha)\right)>1-\varepsilon.
$$
\end{proposition}

Now we give the abundance of hyperbolic periodic points near Pesin blocks described as follows:

\begin{theorem}\label{covering}
For every $f\in PH^1_{EC}(M)$ with partially hyperbolic splitting $TM=E^u\oplus_{\succ} E^{cu}\oplus_{\succ} E^{cs}$, there exists $\alpha>0$ such that for every $\varepsilon>0$ there are $\ell \in \NN$, $\rho_{\ell}>0$, $\delta_{\ell}>0$ and $C^1$ neighborhood $\U$ of $f$ such that for each $g\in \U$, 
we have 
$$
R(W^{s}_{loc}(x,g))\ge \delta_{\ell},\quad R(W^{u}_{loc}(x,g))\ge \delta_{\ell}
$$
for all $x\in \Lambda_{\ell}(g,\alpha)$. In addition, for every $\rho \le \rho_{\ell}$ and $\mu \in G^u(g)$, there is an open covering 
$$
\mathscr{A}_{\mu}( g,\ell, \alpha, \rho)=\Big\{B(x_{i},\rho): x_{i}\in {\rm supp}(\mu|\Lambda_{\ell}(g,\alpha)), 1\le i \le k\Big\}
$$
of ${\rm supp}(\mu|\Lambda_{\ell}(g,\alpha))$ for some $k\in \NN$, which admits following properties:
\begin{itemize}
\item each $B\in\mathscr{A}_{\mu}( g,\ell, \alpha, \rho)$ contains some hyperbolic periodic point, which we denote by $p_B$, so that
$$
R(W^{s}_{loc}(p_B,g))\ge \delta_{\ell},\quad R(W^{u}_{loc}(p_B, g))\ge \delta_{\ell};
$$
\item 
$$
\mu\left(\bigcup_{B\in \mathscr{A}_{\mu}( g,\ell, \alpha, \rho)}\left(B\cap \Lambda_{\ell}(g,\alpha)\right)\right)>1-\varepsilon.
$$
\end{itemize}
\end{theorem}

\begin{proof}
By Proposition \ref{t11}, there exists $\alpha>0$ such that for every $\varepsilon>0$, there are $\ell\in \NN$ and $C^1$ neighborhood $\U$ of $f$, such that $\mu \left(\Lambda_{\ell}(g,\alpha)\right)>1-\varepsilon$ for every $g\in \U$ and $\mu\in G^u(g)$. 
By Lemma \ref{size} and Lemma \ref{sp}, there exist $\rho_{\ell}>0$ and $\delta_{\ell}>0$ such that for every $g\in \U$ and $\mu\in G^u(g)$, we have the following properties: 
\begin{enumerate}
\smallskip
\item\label{x1} every point of $\Lambda_{\ell}(g,\alpha)$ admits local stable and unstable manifolds of size larger than $\delta_{\ell}$;
\smallskip
\item\label{x2} the ball $B(x,\rho)$ for any $\rho\le \rho_{\ell}$ and any $x\in {\rm supp}(\mu|\Lambda_{\ell}(g,\alpha))$ contains hyperbolic periodic points, all of which admit stable and unstable manifolds of size larger than $\delta_{\ell}$.
\end{enumerate}
Due to the compactness, for every $\rho\le \rho_{\ell}$ there exists $k\in \NN$, a sequence of points $x_{1}, \cdots x_{k}$ in ${\rm supp}(\mu|\Lambda_{\ell}(g,\alpha))$ and the family of $\rho$-balls
$$
\mathscr{A}_{\mu}( g,\ell, \alpha, \rho)=\Big\{B(x_{i},\rho): x_i\in {\rm supp}(\mu|\Lambda_{\ell}(g,\alpha)), 1\le i \le k\Big\}
$$
covering ${\rm supp}(\mu|\Lambda_{\ell}(g,\alpha))$. By (\ref{x2}), for every $B\in \mathscr{A}_{\mu}( g,\ell, \alpha, \rho)$ one can fix the hyperbolic periodic point $p_B\in B$ to satisfy
\begin{equation*}\label{x3}
R(W^{s}_{loc}(p_B,g))\ge \delta_{\ell},\quad R(W^{u}_{loc}(p_B, g))\ge \delta_{\ell}.
\end{equation*}
Since ${\rm supp}(\mu|\Lambda_{\ell}(g,\alpha))$ exhibits $\mu$-measure larger than $1-\varepsilon$, by construction of $\mathscr{A}_{\mu}( g,\ell, \alpha, \rho)$ one gets
\begin{equation*}\label{x4}
\mu\left(\bigcup_{1\le i \le k}\left(B(x_{i},\rho)\cap \Lambda_{\ell}(g,\alpha)\right)\right)>1-\varepsilon.
\end{equation*}
Now we complete the proof of Theorem \ref{covering}.
\end{proof}

\section{Cesaro limit measures of iterates}\label{sec4}
For $x\in M$, let $\omega_{\M}(x,f)$ be the set of limit measures of $\{\frac{1}{n}\sum_{i=0}^{n-1}\delta_{f^i(x)}\}_{n\in \NN}$. Similarly, for every $C^1$ embedded sub-manifold $D$, denote by $\omega_{\M}(D,f)$ the set of limit measures of $\{\frac{1}{n}\sum_{j=0}^{n-1}f_{\ast}^j{\rm m}_D\}_{n\in \NN}$, where ${\rm m}_D$ stands for the normalized Lebesgue measure on $D$.

Let $f\in PH^{1}_{EC}(M)$ with partially hyperbolic splitting $TM=E^u\oplus_{\succ} E^{cu}\oplus_{\succ} E^{cs}$, writing $E^{wu}=E^u\oplus E^{cu}$ as before. Let us consider the following subspace charactered by entropy:
$$
G^{cu}(f)=\left\{\mu \in \M(f): \mu \in G^u(f), \quad h_{\mu}(f)= \int \log |{\rm det}Df|_{E^{wu}}| d\mu\right\}.
$$
We call $G^{cu}(f)$ the \emph{Gibbs $cu$-space} of $f$, whose elements are called the \emph{Gibbs $cu$-states} of $f$.

\begin{remark}\label{RK}
Since every $\mu\in G^u(f)$ has only negative Lyapunov exponents along $E^{cs}$,  applying the Ruelle's entropy inequality \cite{ru78}, one knows that 
$
h_{\mu}(f)\le \int \log |{\rm det}Df|_{E^{wu}}| d\mu,
$ 
and therefore the equality in $G^{cu}(f)$ can be replaced by $
h_{\mu}(f)\ge \int \log |{\rm det}Df|_{E^{wu}}| d\mu.
$
\end{remark}

%We are interested in studying the properties of forward orbit of Lebesgue almost every point, and thus let us consider the limit set of the empirical measures of $x\in M$ with respect to $f$ as follows
%$$
%\omega_{\M}(x,f)= \big\{\mu \in \mathcal{M}_{\rm inv}: \exists~\{n_i\}~ \textrm{such that}~\lim_{i\to +\infty}\frac{1}{n_i}\sum_{j=0}^{n_i-1}\delta_{f^jx}=\mu\big\}.
%$$
%Furthermore, we will also study the accumulation of the averaged iterated Lebesgue measures of $C^1$ disk $D$, and define
%$$
%\omega_{\M}(D,f)= \big\{\mu \in \mathcal{M}_f: \exists~\{n_i\}~ \textrm{such that}~\lim_{i\to +\infty}\frac{1}{n_i}\sum_{j=0}^{n_i-1}f_{\ast}^j{\rm Leb}_D=\mu\big\}.
%$$

\paragraph{\bf{Notation}}
For every $f\in PH^{1}_{EC}(M)$, consider $\B_f$ as the set of points whose Cesaro limit measures of iterates are contained in $G^{cu}(f)$, that is
$$
\B_f:=\left\{x\in M: \omega_{\M}(x,f)\subset G^{cu}(f)\right\}.
$$

The main goal of this section is to show the following result.
\begin{theorem}\label{ttd}
Let $f\in PH^{1}_{EC}(M)$ with partially hyperbolic splitting $TM=E^u\oplus_{\succ} E^{cu}\oplus_{\succ} E^{cs}$. If $D$ is a $C^1$ disk transverse to $E^{cs}$, then we have
\begin{enumerate}
\smallskip
\item\label{fde} the intersection of $\B_f$ with $D$ has full ${\rm Leb}_D$-measure;
\smallskip
\item\label{fxe} $\omega_{\M}(D,f)\subset G^{cu}(f)$.
\end{enumerate}
In addition, as a result of (\ref{fde}), we know that $\B_f$ has full Lebesgue measure on $M$.
\end{theorem}

%Given a diffeomorphism $f$ on $M$, denote by $\mathcal{M}_f$ the set of all the $f$-invariant measures.
%We are interested in studying the properties of forward orbit of Lebesgue almost every point, and thus let us consider the limit set of the empirical measures of $x\in M$ with respect to $f$ as follows
%$$
%\omega_{\M}(x,f)= \big\{\mu \in \mathcal{M}_{\rm inv}: \exists~\{n_i\}~ \textrm{such that}~\lim_{i\to +\infty}\frac{1}{n_i}\sum_{j=0}^{n_i-1}\delta_{f^jx}=\mu\big\}.
%$$
%Furthermore, we will also study the accumulation of the averaged iterated Lebesgue measures of $C^1$ disk $D$, and define
%$$
%\omega_{\M}(D,f)= \big\{\mu \in \mathcal{M}_f: \exists~\{n_i\}~ \textrm{such that}~\lim_{i\to +\infty}\frac{1}{n_i}\sum_{j=0}^{n_i-1}f_{\ast}^j{\rm Leb}_D=\mu\big\}.
%$$
Let us recall that $\mu$ is said to be an \emph{SRB measure} of a $C^{1+}$ diffeomorphism if it has positive Lyapunov exponents almost everywhere; and its conditional measures along Pesin unstable manifolds are absolutely continuous w.r.t. Lebesgue measures along these manifolds. As a consequence of Theorem \ref{ttd}, we have the following corollary.

\begin{corollary}\label{TheoA'}
Let $f\in PH^{1+}_{EC}(M)$ with partially hyperbolic splitting $TM=E^u\oplus_{\succ} E^{cu}\oplus_{\succ} E^{cs}$. If $D$ is a $C^1$ disk transverse to $E^{cs}$, then any invariant measure of $\omega_{\M}(D,f)$ is an SRB measure of $f$. Moreover, for ${\rm Leb}_D$-almost every point $x\in D$, any invariant measure of $\omega_{\M}(x,f)$ is an SRB measure.
\end{corollary}

%Since every SRB measure of $f\in \mathcal{PH}^{1+}_{\mathcal{EC}}(M)$ is a Gibbs $u$-sate, it has only negative Lyapunov exponents along $E^{cs}$. Therefore, every SRB measure of $f\in \mathcal{PH}^{1+}_{\mathcal{EC}}(M)$  admits the property of Gibbs $cu$-sates, i.e., it has only positive Lyapunov exponents along $E^{cu}$, whose disintegration along Pesin unstable manifolds tangent to $E^u\oplus E^{cu}$ are absolutely continuous w.r.t. Lebesgue measures.

%Actually, we will prove the following general result on diffeomorphisms in $\mathcal{PH}^{1}_{\mathcal{EC}}(M)$.
%
%%Denote by $\omega_{\M}(x,f)$ the set of limits measures of $\{\frac{1}{n}\sum_{i=0}^{n-1}\delta_{f^i(x)}\}_{n\in \NN}$. Similarly, for every embedded smooth disk $D$, use $\omega_{\M}(D,f)$ to denote the set of limit measures of $\{\frac{1}{n}\sum_{j=0}^{n-1}f_{\ast}^j{\rm Leb}_D\}_{n\in \NN}$. 
%
%\begin{theoremalph}\label{TheoA}
%Let $f\in \mathcal{PH}^{1}_{\mathcal{EC}}(M)$ with partially hyperbolic splitting $TM=E^u\oplus E^{cu}\oplus E^{cs}$. If $D$ is a $C^1$ disk transverse to $E^{cs}$, then every invariant measure $\mu$ of $\omega_{\M}(D,f)$
%is a Gibbs $u$-state of $f$ and satisfies
%\begin{equation}\label{entropy}
%h_{\mu}(f)= \int \log |{\rm det}Df|_{E^{u}\oplus E^{cu}}| d\mu
%\end{equation}
%Moreover, for Lebesgue almost every point $x\in D$, any invariant measure of $\omega_{\M}(x,f)$ is a Gibbs $u$-state and satisfies (\ref{entropy}) as well.
%\end{theoremalph}
%

By assuming Theorem \ref{ttd} we give the proof of Corollary \ref{TheoA'} immediately.
\begin{proof}[Proof of Corollary \ref{TheoA'}]
By Theorem \ref{ttd}, 
we need only to show that any invariant measure of $G^{cu}(f)$ is an SRB measure.  Since every $\mu\in G^{cu}(f)$ is a Gibbs $u$-state, it admits only positive Lyapunov exponents along $E^{cu}$ and negative Lyapunov exponents along $E^{cs}$. This together with 
$$
h_{\mu}(f)=\int \log |{\rm det}Df|_{E^{wu}}| d\mu
$$
implies that $\mu$ must be an SRB measure of $f$, using the classical result \cite[Theorem A]{ly}. 
\end{proof}

%\begin{theoremalph}\label{TheoD}
%For every $C^1$ diffeomorphism $f\in \U(M)$ with dominated splitting $TM=E^u\oplus_{\succ} E^{cu}\oplus_{\succ} E^{cs}$. If $D$ is a $C^1$ disk transverse to $E^{cs}$, then 
%\begin{enumerate}
%\smallskip
%\item\label{fore} $\omega_{\M}(D,f)\subset \G^{cu}(f)$. 
%\smallskip
%\item\label{forp} for Lebesgue almost every point $x\in D$, $\omega_{\M}(x,f)\subset\G^{cu}(f)$.
%\end{enumerate}
%\end{theoremalph}

\subsection{Structure on Gibbs $cu$-spaces}
We have the upper semi-continuity of metric entropy for Gibbs $u$-states and diffeomorphisms in $PH^1_{EC}(M)$.

\begin{theorem}\label{uppp}
Let $\{f_n\}$ be a sequence of diffeomorphisms in $PH^1_{EC}(M)$ which converges to $f\in PH^1_{EC}(M)$ in $C^1$-topology. If $\mu_n\in G^u(f_n)$ for every $n$ and $\mu_n\to \mu$ as $n\to +\infty$, then $\mu\in G^u(f)$ and
$$
\limsup_{n\to +\infty}h_{\mu_n}(f_n)\le h_{\mu}(f).
$$
\end{theorem}

It follows from Theorem \ref{uppp} that we have the following theorem, which can be proved in a standard way as in \cite[Proposition 5.17]{yjg19}.

\begin{theorem}\label{fc}
The map $PH^1_{EC}(M)\ni f\mapsto G^{cu}(f)$ is upper semi-continuous. 
For every $f\in PH^1_{EC}(M)$, the following properties hold:
\begin{enumerate}
\smallskip
\item $G^{cu}(f)$ is compact and convex, all its extreme elements are ergodic;
\smallskip
\item if $f\in PH^{1+}_{EC}(M)$, then all the extreme elements of $G^{cu}(f)$ are physical measures.
\end{enumerate}
\end{theorem}

Now we make effort to prove Theorem \ref{uppp}. Let us recall the definition of \emph{tail entropy}.
\begin{definition}
Given $r>0$ and $x\in M$, define the $r$-tail entropy\footnote{By the result of \cite[Proposition 2.8]{CY16}, $h^*(f,x,r)$ is measurable and $f$-invariant.} at $x$ by 
$$
h^*(f,x,r)=h_{top}(f, B_{\pm \infty}(f,x,r)),
$$
where $B_{\pm \infty}(f, x, r)=\{y: d(f^n(x), f^n(y))\le r, n\in \ZZ\}.$
For any $\mu\in \M(f)$, define the $r$-tail entropy of $\mu$ by
$$
h^*(f,\mu,r)=\int h^*(f,x,r) d\mu(x).
$$
\end{definition}

\begin{lemma}\cite[Theorem 1.2]{cly19}\label{cly19}
Assume that $f$ is a homomorphism on $M$ with finite topological entropy. Then for any $\mu\in \M(f)$ and $r>0$, we have 
$$
h_{\mu}(f)-h_{\mu}(f,\P)\le h^*(f,\mu, r),
$$
where $\mathcal{P}$ is any finite measurable partition with ${\rm diam}(\P)\le r$.
\end{lemma}

By domination and the invariance of plaque families, one has the following result.
\begin{lemma}\cite[Theorem 2.1]{M18S}\label{jj}
If $f$ is a $C^1$ diffeomorphism with dominated splitting $TM=E\oplus _{\succ}F$, then there exists a $C^1$ neighborhood $\U$ of $f$ and constants $0<r<\delta$ such that for every $g\in \U$ and any $g$-ergodic measure $\mu$ one has 
\begin{itemize}
\item either $B_{\pm \infty}(g,x,r)\subset \F^E_{\delta}(g,x)$ for $\mu$-almost every $x$,
\smallskip
\item or $B_{\pm \infty}(g,x,r)\subset \F^F_{\delta}(g,x)$ for $\mu$-almost every $x$.
\end{itemize}
\end{lemma}

Note that this result is also proved in \cite[Lemma 2.3]{cly19} for fake foliations. 
%For every $x\in M$, $r > 0$, let us introduce the following kinds of dynamical balls:
%\begin{itemize}
%\item[--] $B_{+\infty}(f, x, r)=\{y: d(f^n(x), f^n(y))\le r, n\in \NN\};$
%\smallskip
%\item[--] $B_{-\infty}(f, x, r)=\{y: d(f^n(x), f^n(y))\le r, n<0\};$
%\smallskip
%\item[--] $B_{\pm \infty}(f, x, r)=\{y: d(f^n(x), f^n(y))\le r, n\in \ZZ\}.$
%\end{itemize}
%
%\begin{definition}
%Given $r>0$ and $x\in M$, define the $r$-tail entropy at $x$ by 
%$$
%h^*(f,x,r)=h_{top}(f, B_{\pm \infty}(f,x,r)),
%$$
%where $B_{\pm \infty}(f, x, r)=\{y: d(f^n(x), f^n(y))\le r, n\in \ZZ\}.$
%For any $f$-invariant measure $\mu$, define the \emph{$r$-tail entropy} of $\mu$ by
%$$
%h^*(f,\mu,r)=\int h^*(f,x,r) d\mu(x).
%$$
%\end{definition}
%\begin{lemma}\cite[Theorem 1.2]{cly19}\label{cly19}
%Assume that $f$ is a homomorphism on $M$ with finite topological entropy. Then for any $\mu\in \M_{\rm inv}(f)$ and $r>0$, we have 
%$$
%h_{\mu}(f)-h_{\mu}(f,\P)\le h^*(f,\mu, r),
%$$
%where $\mathcal{P}$ is any finite measurable partition with ${\rm diam}(\P)\le r$.
%\end{lemma}
%We have the following result.
\begin{lemma}\label{pes}
For any $f \in PH^{1}_{EC}(M)$, there exists $r>0$ and a $C^1$ neighborhood $\U$ of $f$ such that $h^*(g,\mu,r)=0$ for every $g\in \U$ and $\mu \in G^u(g)$. 
\end{lemma}
\begin{proof}
Let $f\in PH^{1}_{EC}(M)$ with partially hyperbolic splitting $TM=E^u\oplus_{\succ} E^{cu}\oplus_{\succ} E^{cs}$. By Proposition \ref{t11}, there exists a $C^1$ neighborhood $\U$ of $f$, and $\alpha>0$, $\ell\in \NN$ such that $\mu(\Lambda_{\ell}(g,\alpha))>0$ for every $\mu \in G^u(g)$ and $g\in \U$. 

\smallskip
By definition of tail entropy, it suffices to show that there exists $r>0$ such that for every ergodic $\mu \in G^u(g)$, one has $B_{\pm\infty}(g,x,r)=\{x\}$ for $\mu$-almost every $x$. Now we fix $\mu \in G^u(g)$ and $g\in \U$. Following Lemma \ref{jj}, without loss of generality, we assume that for $\mu$-almost every $x$, $B_{\pm\infty}(g,x,r)$ is contained in the $E^{wu}$-direction plaque provided $r$ is small enough. Then using the invariance property of plaques in Proposition \ref{plaque}, we have that all the forward iterates of $B_{\pm\infty}(g,x,r)$ is contained in the corresponding $E^{wu}$-direction plaques.  Since $\mu(\Lambda_{\ell}(g,\alpha))>0$, the Birkhoff's ergodic theorem suggests that for $\mu$-almost every $x$, the forward orbit of $x$ enters to $\Lambda_{\ell}(g,\alpha)$ for infinitely many times. Observe that the local unstable manifold of point in $\Lambda_{\ell}(g,\alpha)$ lies in the $E^{wu}$-direction plaque, by using the backward contracting property of local unstable manifold, one obtains that $B_{\pm\infty}(g,x,r)$ can has only one point $x$. 

\end{proof}

\begin{proof}[Proof of Theorem \ref{uppp}]
Assume that $f_n\to f$ in $C^1$-topology and $\mu_n\to \mu$ in weak$^*$-topology, where $\{f_n\}$ and $f$ are diffeomorphisms in $PH^1_{EC}(M)$, and $\mu_n\in G^u(f_n)$ for every $n\in \NN$.
It follows from Lemma \ref{upp} that $\mu \in G^u(f)$. By Lemma \ref{pes}, there exists $r>0$ so that $h^*(f,\mu,r)=0$ and $h^*(f_n,\mu_n,r)=0$ for every $n$ large enough.
By Lemma \ref{cly19}, for any measurable partition $\P$ for which ${\rm diam}(\P)\le r$ and $\mu(\partial(\P))=0$, we have $h_{\mu_n}(f_n)= h_{\mu_n}(f_n,\P)$ and $h{_\mu}(f)= h_{\mu}(f,\P)$. For this fixed $\P$, we can check by definition that
$$
\limsup_{n\to +\infty}h_{\mu_n}(f_n,\P)\le h_{\mu}(f,\P).
$$
All these together imply
$$
\limsup_{n\to +\infty}h_{\mu_n}(f_n)\le \limsup_{n\to +\infty}h_{\mu_n}(f_n,\P)\le h_{\mu}(f,\P)=h_{\mu}(f).
$$
This proves the result.
\end{proof}

\subsection{Proof of Theorem \ref{ttd}}

The aim of this subsection is to complete the proof of Theorem \ref{ttd}. To begin with, let us recall the notion of \emph{pseudo-physical measure relative to compact disks}, introduced in \cite[$\S$ 2.3]{CYZ18} as an extension of pseudo-physical measures \cite{ce11,cce15}. 

\begin{definition}\label{pds}
Given a compact $C^1$ disk $D$, we call $\mu\in \M(f)$ a pseudo-physical measure relative to $D$ if 
$$
{\rm Leb}_D\left(\{x\in D: d(\omega_{\M}(x,f), \mu )<\varepsilon\}\right)>0\quad \textrm{for every}~\varepsilon>0.
$$
\end{definition}

\begin{lemma}\cite[Theorem 2.3]{CYZ18}\label{pp}
Suppose that $f$ is a homeomorphism on $M$ and $D$ is an embedded $C^1$ compact disk. Then the set of
pseudo-physical measures relative to $D$ is a compact non-empty set. Moreover, for ${\rm Leb}_D$-almost every $x\in D$, every invariant measure $\mu \in \omega_{\M}(x,f)$ is a pseudo-physical measure relative to $D$.
\end{lemma}

%\begin{theorem}\label{k}
%For every $C^1$ diffeomorphism $f\in \U(M)$ with dominated splitting $TM=E^u\oplus E^{cu}\oplus E^{cs}$. Assume that $D$ is a $C^1$ disk transverse to $E^{cs}$, if $\mu\in p\omega(f,D)$, then $\mu \in \G^{cu}(f)$.
%\end{theorem}
%
%\begin{theorem}
%For every $C^1$ diffeomorphism $f\in \U(M)$ with dominated splitting $TM=E^u\oplus E^{cu}\oplus E^{cs}$. If $D$ is a $C^1$ disk transverse to $E^{cs}$, then for Lebesgue almost every point $x\in D$, any invariant measure of $p\omega(f,x)$ belongs to $\G^{cu}(f)$.
%\end{theorem}

%Now we fix a metric compatible with the weak*-topology on the space of Borel probability  measures supported on $M$ as follows: for two probability measures $\mu$ and $\nu$, define
%$$
%{\rm dist}(\mu,\nu)=\sum_{n\in \NN}\frac{|\int \varphi_nd\mu- \int \varphi_n d\nu|}{2^n\sup_{x\in M}|\varphi_n(x)|},
%$$
%where $\{\varphi_n\}_{n\in \NN}$ is a countable dense subset of the space of continuous functions on $M$.
%
Given a homeomorphism $f$ on $M$, for every $\mu \in \M(f)$, for every $\varepsilon>0$ and $n\in \NN$, define the following subset of $M$:
$$
\mathcal{G}_{\mu}(n,\varepsilon)=\left\{x\in M: {\rm dist}\left(\frac{1}{n}\sum_{i=0}^{n-1}\delta_{f^ix},\mu\right)<\varepsilon \right\}.
$$

\begin{lemma}\cite[Theorem C']{CYZ18}\label{DU}
Let $f$ be a $C^1$ partially hyperbolic diffeomorphism with invariant splitting $TM=E^u\oplus_{\succ} E^{cs}$. Then for any $C^1$ disk $D$ transverse to $E^{cs}$, we have $\omega_{\M}(x,f)\subset G^u(f)$
for ${\rm Leb}_D$-almost every $x$ in $D$.
\end{lemma}

%For a $C^1$ embedded sub-manifold $D$, we use ${\rm Leb}_{D}$ to present the induced Lebesgue measure on $D$.
%Recall the notion of pseudo-physical measure relative to compact disks, introduced in \cite[Section 2.3]{CYZ18} as a extension of pseudo-physical measures \cite{ce11,cce15}. 
%
%\begin{definition}
%Given a compact $C^1$ disk $D$, we say an $f$-invariant measure $\mu$ is a pseudo-physical measure relative to  $D$ if for every $\varepsilon>0$, 
%$$
%{\rm Leb}_D\left(\{x: d(p\omega(x,f), \mu )<\varepsilon\}\right)>0
%$$
%\end{definition}
%
%\begin{lemma}\cite[Theorem 2.3]{CYZ18}\label{pp}
%Suppose that $f$ is a homeomorphism on $M$ and $D$ is an embedded $C^1$ compact disk. Then the set of
%pseudo-physical measures relative to $D$ is a compact non-empty set. Moreover, for Lebesgue almost every $x$, every invariant measure $\mu \in p\omega(x,f)$ is a pseudo-physical measure relative to $D$.
%\end{lemma}
%
%
%\begin{lemma}\cite[Theorem E]{CYZ18}\label{U}
%Assume that $f$ is a $C^1$ partially hyperbolic diffeomorphism with strong unstable direction $E^u$, then there exist $\theta>0$, $R>0$ such that for every $\mu\in \M_{\rm inv}(f)$, for any $C^1$ disk $D$ tangent to $\mathcal{C}^{E^u}_{\theta}$ with diameter smaller than $R$, one has
%\begin{equation*}
%\lim_{\varepsilon\to 0}\limsup_{n\to +\infty}\frac{\log({\rm Leb}_D(\mathcal{G}_{\mu}(n,\varepsilon)))}{n}\le h_{\mu}(f,\F^u)-\int \log|{\rm det}Df|_{E^u}|d\mu.
%\end{equation*}
%
%\end{lemma}

\begin{lemma}\cite[Theorem H]{CYZ18}\label{E}
Let $f$ be a $C^1$ diffeomorphism with dominated splitting $TM=E\oplus_{\succ} F$, then there exist $\theta>0$, $R>0$ such that for every $\mu\in \M(f)$, for any $C^1$ disk $D$ tangent to $\mathcal{C}^{E}_{\theta}$ with diameter smaller than $R$, one has
\begin{equation*}
\lim_{\varepsilon\to 0}\limsup_{n\to +\infty}\frac{\log({\rm Leb}_D(\mathcal{G}_{\mu}(n,\varepsilon)))}{n}\le h_{\mu}(f)-\int \log|{\rm det}Df|_{E}|d\mu.
\end{equation*}
\end{lemma}

%\paragraph{\bf{Notation}}
%Consider the set of points whose Cesaro limit measures of iterates are contained in $\G^{cu}(f)$, that is
%$$
%\B_f:=\left\{x\in M: \omega_{\M}(x,f)\subset \G^{cu}(f)\right\}
%$$
%We have the following result.

%\begin{theorem}\label{ttd}
%For every $f\in \mathcal{PH}^{1}_{\mathcal{EC}}(M)$,
%$\B_f$ has a full Lebesgue measure, it intersects every disk transverse to $E^{cs}$ with a full Lebesgue measure.
%\end{theorem}
The following result is essentially contained in \cite[$\S$ 5.2]{CYZ18}.
\begin{lemma}\label{led}
Let $f$ be a $C^1$ diffeomorphism with dominated splitting $TM=E\oplus_{\succ} F$. Let $\M_0$ be a compact subset of $\M(f)$. Then for any $C^1$ disk $D$ transverse to $F$, if
$\omega_{\M}(x,f)\subset \M_0$ for ${\rm Leb}_D$-almost every $x\in D$, then we have $\omega_{\M}(D,f)\subset \M_0$.
\end{lemma}
Now we are in position to give the proof of Theorem \ref{ttd}.

\begin{proof}[Proof of Theorem \ref{ttd}]
Let $\theta$ and $R$ be the positive constants given by Lemma \ref{E}, where we assume $E:=E^{wu}=E^u\oplus E^{cu}$ and $F:=E^{cs}$. Now we take a $C^1$ disk $D$ transverse to $E^{cs}$, and prove the property (\ref{fde}) firstly.
By domination, up to considering iterates we can assume $D$ is compact and tangent to $\mathcal{C}_{\theta}^E$ with diameter smaller than $R$. 
Since $D$ can be foliated by disks transverse to $E^{cu}\oplus E^{cs}$, by Lemma \ref{DU} we know that any measure in $\omega_{\M}(x,f)$ is a Gibbs $u$-state for ${\rm Leb}_D$-almost every $x\in D$. 
Thus, according to Remark \ref{RK}, it suffices to prove that for ${\rm Leb}_D$-almost every point $x\in D$, any invariant measure $\mu\in \omega_{\M}(x,f)$ satisfies
\begin{equation}\label{1}
h_{\mu}(f)\ge \int \log |{\rm det}Df|_{E^{wu}}|d\mu.
\end{equation}
By Lemma \ref{pp}, for ${\rm Leb}_D$-almost every $x\in D$, any $\mu \in \omega_{\M}(x,f)$ is a pseudo-physical measure relative to $D$, denote
$$
P_{\mu}=h_{\mu}(f)-\int \log |{\rm det}Df|_{E^{wu}}|d\mu.
$$

\smallskip
Now we show that $\mu$ satisfies (\ref{1}), that is $P_{\mu}\ge 0$.
From Definition \ref{pds} we know that for every $\varepsilon>0$, there exists $\eta_{\varepsilon}>0$ such that 
$$
{\rm Leb}_D\left(\{x\in D: {\rm dist}(\omega_{\M}(x,f),\mu)<\varepsilon\}\right)>\eta_{\varepsilon},
$$
this together with the fact
$$
\{x\in D: {\rm dist}(\omega_{\M}(x,f),\mu)<\varepsilon\}\subset \bigcap_{k=1}^{\infty}\bigcup_{n=k}^{\infty}\mathcal{G}_{\mu}(n,\varepsilon)
$$
yields the estimates
\begin{equation}\label{psl}
{\rm Leb}_D\left(\bigcup_{n=k}^{\infty}\mathcal{G}_{\mu}(n,\varepsilon)\right)>\eta_{\varepsilon}\quad\textrm{for any}~ k\in \NN.
\end{equation}
By contradiction, there is $\delta>0$ such that 
$P_{\mu}<-2\delta$. 
Using Lemma \ref{E}, there are constants $\varepsilon>0$ and $L>0$ such that for every $n\in \NN$, 
$$
{\rm Leb}_{D}(\mathcal{G}_{\mu}(n,\varepsilon))<L\cdot{\rm exp}(n(P_{\mu}+\delta))\le L\cdot{\rm exp}(-\delta n).
$$
%Then, Lemma \ref{pp} gives that for every $\mu\in \M_{\rm inv}(f)$, for any $\delta>0$, there exist $C>0$ and $\varepsilon>0$ such that 
%$$
%{\rm Leb}_{D}(\mathcal{G}_{\mu}(n,\varepsilon))<C\cdot{\rm exp}(n(P_{\mu}+\delta))\quad \textrm{for every}~n\in \NN.
%$$
%On the other hand, by Lemma \ref{pp}, for Lebesgue almost every $x\in D$, any $\mu \in \omega_{\M}(x,f)$ is a pseudo-physical measure relative to $D$, recalling its definition in Definition \ref{pds} we know that for every $\varepsilon>0$, there exists $\eta_{\varepsilon}>0$ such that 
%$$
%{\rm Leb}_D\left(\{x\in M: {\rm dist}(\omega_{\M}(x,f),\mu)<\varepsilon\}\right)>\eta_{\varepsilon}.
%$$
%This together with the fact
%$$
%\{x\in M: {\rm dist}(\omega_{\M}(x,f),\mu)<\varepsilon\}\subset \bigcap_{k=1}^{\infty}\bigcup_{n=k}^{\infty}\mathcal{G}_{\mu}(n,\varepsilon),
%$$
%yields the estimates
%\begin{equation}\label{psl}
%{\rm Leb}_D\left(\bigcup_{n=k}^{\infty}\mathcal{G}_{\mu}(n,\varepsilon)\right)>\eta_{\varepsilon}\quad\textrm{for any}~ n\in \NN.
%\end{equation}
%Now we show $\mu$ satisfies $(\ref{1})$. 
%We assume by contradiction that $\mu$ is not contained in $\G^{cu}(f)$. Hence, there is $\delta>0$ such that either 
%$P_{\mu}<-2\delta$ or $Q_{\mu}<-2\delta$. For this $\delta>0$, from argument on (\ref{db}) one knows that there are constants $\varepsilon>0$ and $L>0$ such that for every $n\in \NN$, 
%$$
%{\rm Leb}_{D}(\mathcal{G}_{\mu}(n,\varepsilon))<L\cdot{\rm exp}(n\left(\min\{P_{\mu},Q_{\mu}\}+\delta\right))\le L\cdot{\rm exp}(-\delta n).
%$$
Consequently, for every $k\in \NN$,
$$
{\rm Leb}_D\left(\bigcup_{n=k}^{\infty}\mathcal{G}_{\mu}(n,\varepsilon)\right)\le L\sum_{n=k}^{\infty}{\rm exp}(-\delta n),
$$
which converges to zero as $k\to +\infty$. This a contradiction to (\ref{psl}) and thus we obtain that $\mu$ satisfies (\ref{1}). We get property (\ref{fde}) now. 

\smallskip
Property (\ref{fxe}) is a consequence of (\ref{fde}) and Theorem \ref{fc}. Indeed, we have that $G^{cu}(f)$ is compact by Theorem \ref{fc}, and we know $\omega_{\M}(x,f)\subset G^{cu}(f)$ for ${\rm Leb}_D$-almost every $x\in D$ from (\ref{fde}). Therefore, by applying Lemma \ref{led} to $E:=E^{wu}$, $F:=E^{cs}$ and $\M_0:=G^{cu}(f)$, one gets $\omega_{\M}(D,f)\subset G^{cu}(f)$. Observe that $M$ can be covered by finitely many small balls, which can be foliated by disks transverse to $E^{cs}$, thus one can use the Fubini's theorem to show that $\B_f$ has full Lebesgue measure.

The proof of Theorem \ref{ttd} is complete now.
\end{proof}

\subsection{Typical measurable partitions of Gibbs $cu$-states}\label{subty}
Let us take $\delta_M>0$ sufficiently small such that for any $x\in M$ the inverse of
exponential map ${\rm exp}_x^{-1}$ is well defined on the $\delta_M$ open neighborhood of $x$. We then identify this neighborhood with the corresponding neighborhood on tangent space. 

Fix $f\in PH^1_{EC}(M)$ with partially hyperbolic splitting $TM=E^u\oplus_{\succ} E^{cu}\oplus_{\succ} E^{cs}$, let $\mathcal{C}_a^{wu}:=\mathcal{C}_a^{E^{wu}}$ for $a>0$. Denote by $\DD^u$, $\DD^s$ the compact unit disks in the Euclidean space of dimension ${\rm dim}E^{wu}$ and ${\rm dim}E^{cs}$, respectively.

\begin{definition}
A subset $C\subset M$ is called a cylinder if there exists a diffeomorphism $\psi: \DD^u\times \DD^s \to C$. We say that a smooth embedded sub-manifold $\gamma$ of dimension ${\rm dim} E^{wu}$ (resp. ${\rm dim} E^{cs}$) crosses $C$ when $\psi^{-1}(\gamma\cap C)$ is a $C^1$ graph over $\DD^u$ (resp. $\DD^s$).
\end{definition}

By domination, we can choose $a>0$ sufficiently small, and $0<\rho_0<\delta_M$ with following properties: if consider $D$ as a compact smooth disk centered at $x\in M$ with radius $\rho\le \rho_0$, which is tangent to $\mathcal{C}_a^{wu}$ also, then the subset 
$$
C(x,D,\rho)=\bigcup_{y\in D}\left\{{\rm exp}_y(v): v\perp T_yD, \|v\| \le \rho \right\}
$$
is a diffeomorphic image of $\DD^u\times \DD^s$, and thus a cylinder (of size $\rho$). 

\smallskip
Given $\alpha>0$ and $\ell\in \NN$. By Lemma \ref{size}, one can take $\delta_{\ell,\alpha}$ such that all the points of $\Lambda_{\ell}(f,\alpha)$ admit local stable and unstable manifolds of size larger than $\delta_{\ell,\alpha}$. 
Thus, for any $\rho<\delta_{\ell,\alpha}$ sufficiently small, by domination and the choice of $\rho_0$ one can choose $r_{\ell,\alpha, \rho}\in (\rho,\min\{\delta_{\ell,\alpha},\rho_0\})$ such that for each $x\in M$ and disk $D$ tangent to $\mathcal{C}_a^{wu}$ centered at $x$ of radius $r_{\ell,\alpha,\rho}$, we have that both $W_{\delta_{\ell,\alpha}}^u(y,f)$ and $W_{\delta_{\ell,\alpha}}^s(y,f)$ cross $C(x,D, r_{\ell,\alpha,\rho})$ whenever $y\in B(x,\rho)\cap \Lambda_{\ell}(f,\alpha)$.
It follows that one can construct the measurable partition formed by local unstable disks with uniform size as follows:
$$
\P(x,D,\alpha,\ell, \rho)=\left\{W_{\delta_{\ell,\alpha}}^u(y,f)\cap C(x,D,r_{\ell,\alpha,\rho}): y\in B(x,\rho)\cap \Lambda_{\ell}(f,\alpha)\right\}.
$$
For $\mu\in \M(f)$, we say that $\P(x,D,\alpha,\ell,\rho)$ is a \emph{typical measurable partition} associated to $\mu$ if $\mu(B(x,\rho)\cap \Lambda_{\ell}(f,\alpha))>0$.

\smallskip
By Proposition \ref{t11}, we know that for every Gibbs $u$-state, there exists some typical measurable partition associated to it. Moreover, for Gibbs $cu$-states one has the following simple result. 
%Let us recall the Rokhlin's disintegration theorem (e.g. see \cite{R62}, \cite[Appendix C.4]{BDV05}) stated as follows.
%\begin{theorem}\label{rok}
%For any measurable partition $\F$, when denote by $\hat{\mu}$ the quotient measure of $\mu$ w.r.t. $\mathcal{F}$, there exists a unique family of \emph{conditional measures} $\{\mu_{\gamma}: \gamma\in \mathcal{F}\}$ of $\mu$ w.r.t. $\mathcal{F}$ such that 
%\begin{itemize}
%\smallskip
%\item[--] $\mu_{\gamma}(\gamma)=1$ for $\hat{\mu}$-almost every $\gamma\in \mathcal{F}$, and 
%\smallskip
%\item[--] for any measurable set $A$, $\gamma\mapsto \mu_{\gamma}(A)$ is measurable with $\mu(A)=\int \mu_{\gamma}(A){\rm d}\hat{\mu}$.
%\end{itemize}
%\end{theorem}
%
%Now we show the following simple result.
\begin{lemma}\label{ggh}
For every $\mu\in G^{cu}(f)$ and a typical measurable partition $\P(x,D,\alpha,\ell,\rho)$ associated to $\mu$, there exists $D\in \P(x,D,\alpha,\ell,\rho)$ such that 
\begin{itemize}
\item $D\subset {\rm supp}(\mu)$;
\smallskip
\item ${\rm Leb}_D(\Lambda_{\ell}(f,\alpha))>0$.
\end{itemize}
Moreover, if $\mu$ is ergodic for $f$, then we can
require $D$ to satisfy that Lebesgue almost every point of $D$ is contained in $B(\mu,f)$.
\end{lemma}
\begin{proof}
%By Rokhlin's disintegration theorem (see e.g. \cite{R62}, \cite[Appendix C.4]{BDV05}), $\mu$ has the disintegration along $\P$, that is $\mu=\int \mu_D d\hat{\mu}$, where $\{\mu_D\}$ are conditional measure of $\mu$ along $\P$, and $\hat{\mu}$ is the quotient measure. Since $\mu$
%
%there exists $D\in \mathcal{P}$ with positive $\hat{\mu}$-measure such that 
%$$
%\mu_{D}(\Lambda_{\ell}(f,\alpha))>0\quad \textrm{and}\quad \mu_D({\rm supp}\mu)=1.
%$$ 
Since $\mu$ is a Gibbs $cu$-state, the conditional measures of $\mu$ along unstable disks of $\P(x,D,\alpha,\ell,\rho)$ are equivalent to the Lebesgue measures almost everywhere(see e.g. \cite[Theorem 13.1.2]{bp07}). Consequently, there exists $D\in \P(x,D,\alpha,\ell,\rho)$ such that ${\rm Leb}_D(\Lambda_{\ell}(g,\alpha))>0$ and Lebesgue almost every point of $D$ is contained in ${\rm supp}(\mu)$.

Now we show $D\subset {\rm supp}(\mu)$. From the choice of $D$, we can choose $y\in D$ such that $y\in {\rm supp}(\mu)$, and assume by contradiction that there is $z\in D$ which is not in ${\rm supp}(\mu)$. It follows that there exists an open neighborhood $V$ of $z$ such that $V\subset M \setminus{{\rm supp}(\mu)}$, thus all the points in $V\cap D$ are not in ${\rm supp}(\mu)$, which is a contradiction to the fact that Lebesgue almost every point of $D$ is contained in ${\rm supp}(\mu)$. 

If $\mu \in G^{cu}(f)$ is ergodic, by Birkhoff's ergodic theorem one gets that $B(\mu,f)$ has full $\mu$-measure. Thus, one can choose $D$ to satisfy that ${\rm Leb}_D$-almost every point is contained in $B(\mu,f)$. 
\end{proof}

\section{Certain properties of skeletons and pre-skeletons}\label{sec5}

\paragraph{\bf{Notation}} Given a $C^1$ diffeomorphism $f$, let $p,q$ be hyperbolic periodic points of $f$ with the same index, denote 
\begin{itemize}
\smallskip
\item
$p\overset{f}\prec q$ if $W^u({\rm Orb}(p,f))\pitchfork W^s({\rm Orb}(q,f))\neq \emptyset.$
\smallskip
\item
$p\overset{f}\sim q$ if $p\overset{f}\prec q$ and $q\overset{f}\prec p$, or equivalently they are homoclinic related w.r.t. $f$.
\end{itemize}
Sometimes we drop the superscript ``$f$" if there is no ambiguity.

\smallskip
In what follows, let $f$ be a $C^1$ diffeomorphism in $PH^1_{EC}(M)$ with partially hyperbolic splitting $TM=E^u\oplus_{\succ} E^{cu}\oplus_{\succ} E^{cs}$ \footnote{Notice that in the proofs of lemmas in this section, we use neither the mostly expanding behavior nor the mostly contracting behavior on centers.}.
By definition of pre-skeletons, one gets the following result directly.

\begin{lemma}\label{sc}
Given $n\in \NN$, if $\mathcal{S}=\{p_1,\cdots, p_k\}$ is a pre-skeleton of $f^n$, then it is a pre-skeleton of $f$.
\end{lemma}

The proof of the following result is very similar to \cite[Lemma 2.5]{dvy16}, hence we omit the proof.
\begin{lemma}\label{ske3}
Every pre-skeleton of $f$ contains a subset to be a skeleton of $f$.
\end{lemma}

\begin{lemma}\label{ske1}
If both $\{p_i: 1\le i \le k\}$ and $\{q_i: 1\le i \le \ell \}$ are skeletons of $f$, then there exists a bijective map $i\mapsto j(i)$ such that $p_i\sim q_{j(i)}$. Consequently, 
we must have $k=\ell$.
\end{lemma}

\begin{proof}
Consider two skeletons $\mathcal{S}_1=\{p_i: 1\le i \le k\}$ and $\mathcal{S}_2=\{q_i: 1\le i \le \ell \}$ of $f$. For each $p_i\in \mathcal{S}_1$, we know that $W^u_{loc}(p_i,f)$ is a $C^1$ disk transverse to $E^{cs}$. It follows from the condition (S\ref{S1}) in Definition \ref{S}, there is $q_j\in \mathcal{S}_2$ so that $W^u_{loc}(p_i,f)$ intersects $W^s({\rm Orb}(q_j,f))$ transversally.
For this $q_j$, with the same reason we can find $p_l\in \mathcal{S}_1$ such that $W^u_{loc}(q_j,f)$ intersects $W^s({\rm Orb}(p_l,f))$ transversally. As a consequence, inclination lemma \cite[Theorem 5.7.2]{BG02} leads to 
$$
W^u_{loc}(p_i,f)\pitchfork W^s({\rm Orb}(p_l,f))\neq \emptyset.
$$
Thus, condition (S\ref{S2}) in Definition \ref{S} gives $p_i=p_l$. Hence, we have $p_i\sim q_j$, and the choice of $q_j$ is unique. Similarly, each element of $\mathcal{S}_2$ has a unique element of $\mathcal{S}_1$ such that they are homoclinically related. Consequently, we have $k=l$.
\end{proof}

\begin{lemma}\label{ske2}
If $\{p_i: 1\le i \le k\}$ is a skeleton of $f$, then $H(p_i,f)=\overline{W^u({\rm Orb}(p_i,f))}$ for every $p_i$.
\end{lemma}

\begin{proof}
Since we have $H(p_i,f)\subset\overline{W^u({\rm Orb}(p_i,f))}$ by definition, it suffices to show the other direction.  For every point $x\in W^u({\rm Orb}(p_i,f))$, for any $\varepsilon>0$ sufficiently small, let us take the disk $D_{\varepsilon}\subset W^u({\rm Orb}(p_i,f))$ centered at $x$ with radius $\varepsilon$. By condition (S\ref{S1}) in Definition \ref{S}, there is $p_j\in \mathcal{S}$ such that 
$$
D_{\varepsilon}\pitchfork W^s({\rm Orb}(p_j,f))\neq \emptyset.
$$
Observe that there is no heteroclinic intersection between $W^u({\rm Orb}(p_i,f))$ and $W^s({\rm Orb}(p_j,f))$ as long as $i\neq j$. Thus, we must have $p_j=p_i$, and the arbitrariness of $\varepsilon$ implies $x\in \overline{W^u({\rm Orb}(p_i,x))\pitchfork W^s({\rm Orb}(p_i,f))}$. Hence, we have $\overline{W^u({\rm Orb}(p_i,f))}\subset H(p_i,f)$ to end the proof.
\end{proof}

\section{Reduction of $PH^{1}_{EC}(M)$}\label{sec6}
Let $PH^1_{C}(M)$ be the subset of $PH^1_{EC}(M)$ for which if $f\in PH^1_{C}(M)$, then there exists a partially hyperbolic splitting $TM=E^u\oplus_{\succ} E^{cu}\oplus_{\succ} E^{cs}$ such that for every $\mu \in G^u(f)$, we have 
$$
\chi_{\mu}(f^{-1},E^{cu})<0\quad \textrm{and}\quad\chi_{\mu}(f,E^{cs})<0.
$$
%\begin{itemize}
%\smallskip
%\item $\chi_{\mu}(f^{-1},E^{cu})<0$ for every Gibbs $u$-state $\mu$;
%\smallskip
%\item $\chi_{\mu}(f,E^{cs})<0$ for every Gibbs $u$-state $\mu$
%\end{itemize}

\smallskip
Lemma \ref{gibbsp} enables us to have the following result.
The proof follows the same lines as the proof of \cite[Proposition 3.3]{MCY17}, hence omitted.

\begin{theorem}\label{utv}
If $f\in PH^1_{EC}(M)$, then there exists $N\in \NN$ such that $f^N\in PH^1_{C}(M)$.
\end{theorem}
%\begin{lemma}\label{p33}
%If $f\in \mathcal{PH}^1_{\mathcal{EC}}(M)$, then there exists $N_0\in \NN$ and $\alpha_0>0$ such that for every $\mu \in \G^u(f)$, one has
%$$
%\chi_{\mu}(f^{N_0},E^{cs})<-\alpha_0, \quad \chi_{\mu}(f^{-N_0},E^{cu})<-\alpha_0.
%$$
%Furthermore, there is $N_{u}\in \NN$ as a
%multiple of $N_0$ and $\alpha_u\in (0, \alpha_0/N_0)$ such that
%$$
%\chi_{\mu}(f^{N_u},E^{cs})<-\alpha_u,\quad \chi_{\mu}(f^{-N_u},E^{cu})<-\alpha_u$$
%for every $\mu \in \G^u(f^{N_u})$.
%\end{lemma}
%
%It can be deduced by Lemma \ref{gibbsp}. 
%The proof follows the same lines as the proof Proposition 3.3 of \cite{MCY17}, hence omitted.
%As a direct consequence, we have
%
%\begin{theorem}\label{utv}
%If $f\in \mathcal{PH}^1_{\mathcal{EC}}(M)$, then there exists $N\in \NN$ such that $f^N\in \mathcal{PH}^1_{\mathcal{C}}(M)$.
%\end{theorem}

By Theorem \ref{utv}, in order to prove Theorem \ref{TheoB}, we can turn to prove the following parallel result for diffeomorphisms in $PH^{1}_{C}(M)$ firstly. Let $PH^{1+}_{C}(M)$ be the set of $C^{1+}$ diffeomorphisms in $PH^1_{C}(M)$.

\begin{theoremalph}\label{TheoF}
Every $f\in PH^{1}_{C}(M)$ has a skeleton. Moreover, if $f\in PH^{1+}_{C}(M)$, then there exist finitely many physical measures of $f$ with basin covering property, denoted by $\P(f)=\{\mu_1,\cdots,\mu_{k}\}$. If $\mathcal{S}(f)=\{p_1,\cdots, p_{\ell}\}$ is a skeleton of $f$, then there exists a bijective map $i\mapsto j(i)$ such that 
$$
{\rm supp}(\mu_i)=\overline{W^u({\rm Orb}(p_{j(i)},f))}=H(p_{j(i)},f).
$$
for every $\mu_i\in \P(f)$ and the corresponding $p_{j(i)}\in \mathcal{S}(f)$.
In particular, the number of physical measures of $f$ coincides with the cardinality of
its skeleton.
\end{theoremalph}

Theorem \ref{TheoF} will be deduced in Subsection \ref{proofD}. Now we fix some notations for simplicity.
Let $f\in PH^{1}_{C}(M)$ with partially hyperbolic splitting $TM=E^u\oplus_{\succ} E^{cu} \oplus_{\succ} E^{cs}$.
Given $\sigma>0$ and $x\in M$, let $HT(x,\sigma,f)$ be the family of $(E^{wu}, \sigma)$-hyperbolic times of $x$ w.r.t. $f$, 
%that is,
%$$
%HT(x,\sigma,f)=\left\{n: \prod_{i=n-k+1}^{n}\|Df^{-1}|_{E^{wu}(f^i(x))}\|<\sigma^k \quad \textrm{for}~ 1\le k\le n
%\right\}
%$$
where $E^{wu}=E^u\oplus E^{cu}$. For $a>0$, rewriting $\mathcal{C}_a^{wu}:=\mathcal{C}_a^{E^{wu}}$ and $\mathcal{C}_a^{cs}:=\mathcal{C}_a^{E^{cs}}$.

\subsection{The existence of skeletons}
%Let $f\in \mathcal{PH}^{1}_{\mathcal{C}}(M)$ with partially hyperbolic splitting $TM=E^u \oplus_{\succ} E^{cu} \oplus_{\succ} E^{cs}$.
%Given $\sigma>0$ and $x\in M$, let $\mathcal{HT}(x,\sigma,f)$ be the family of $(E^{wu}, \sigma)$-hyperbolic times of $x$ w.r.t. $f$, recalling $E^{wu}=E^u\oplus E^{cu}$.
This subsection is dedicated to show the existence of skeletons for diffeomorphisms in $PH^{1}_{C}(M)$.
\begin{theorem}\label{es}
Every $f\in PH^{1}_{C}(M)$ admits a skeleton $\mathcal{S}(f)=\{p_1,\cdots, p_{k_f} \}$ with some $k_f \in \NN$ depending only on $f$.
\end{theorem}

%Denote by $E_g^{wu}=E_g^u\oplus E_g^{cu}$ for any $g$ $C^1$-close enough to $f$. 
We need the following simple observation:

\begin{lemma}\label{he}
Given $\alpha>0$, and assume that $E$ is a $Df$-invariant sub-bundle of $C^1$ diffeomorphism $f$. If $x\in M$ is the point such that $\chi_{\mu}(f^{-1},E)<-\alpha$ for every $\mu\in \omega_{\M}(x,f)$,
then $x$ is $\alpha$-nonuniformly expanding along $E$.
\end{lemma}

%\begin{proof}
%We argue by contradiction. Suppose that $x$ is not $\alpha$-nonuniformly expanding along $E$, then there is a subsequence $\{n_k\}$ such that
%$$
%\lim_{k \to +\infty}\frac{1}{n_k}\sum_{i=1}^{n_k}\log \|Df^{-1}|_{E(f^i(x))}\|\ge-\alpha.
%$$
%Up to take subsequences, we can assume that 
%$$
%\lim_{k\to +\infty}\frac{1}{n_k}\sum_{i=1}^{n_k}\delta_{f^i(x)}=\mu.
%$$
%Consequently, we have
%$$
%\chi_{\mu}(f^{-1},E)=\lim_{k\to +\infty}\frac{1}{n_k}\sum_{i=1}^{n_k}\log \|Df^{-1}|_{E(f^i(x))}\|\ge-\alpha.
%$$
%This gives a contradiction, since $\mu \in \omega_{\M}(x,f)$ by construction.
%
%\end{proof}

%\begin{lemma}\label{u}
%For every $f\in \V(M)$ with partially hyperbolic splitting $TM=E^u \oplus E^{cu} \oplus E^{cs}$, there exists $\alpha>0$ an a $C^1$-neighborhood $\U$ of $f$ such that 
%$$
%\chi_{\mu}(g^{-1},E^{wu}_g)<-\alpha
%$$
%for every $\mu\in \G^u(g)$ and $g\in \U$.
%\end{lemma}

%Let $f\in \mathcal{PH}^{1}_{\mathcal{C}}(M)$ with partially hyperbolic splitting $TM=E^u \oplus_{\succ} E^{cu} \oplus_{\succ} E^{cs}$.
%Given $\sigma>0$ and $x\in M$, let $\mathcal{HT}(x,\sigma,f)$ be the family of $(E^{wu}, \sigma)$-hyperbolic times of $x$ w.r.t. $f$, where $E^{wu}=E^u\oplus E^{cu}$. For $a>0$, rewriting $\mathcal{C}_a^{wu}:=\mathcal{C}_a^{E^{wu}}$ and $\mathcal{C}_a^{cs}:=\mathcal{C}_a^{E^{cs}}$.

\begin{proposition}\label{uv}
For every $f\in PH^{1}_{C}(M)$ with partially hyperbolic splitting $TM=E^u \oplus_{\succ} E^{cu} \oplus_{\succ} E^{cs}$, there exist $H_f>0$, $D_f\in (0,1)$ and a $C^1$ neighborhood $\U_f$ of $f$ such that for every $g\in \U_f$, every point $x\in \B_g$ is $H_f$-nonuniformly expanding along $E^{wu}$ with the density estimate
\begin{equation}\label{est}
\mathcal{D}_{L}\left(HT(x,{\rm e}^{-H_f/2},g)\right)\ge D_f. 
\end{equation}
\end{proposition}

\begin{proof}
Since $E^u$ is uniformly expanding and $G^u(f)$ is compact, one conclude that there exists $H_f>0$ such that
$$
\chi_{\mu}(f^{-1},E^{wu})<-H_f \quad \textrm{for every}~\mu\in G^u(f).
$$
Applying the upper semi-continuity of Gibbs $u$-space w.r.t. diffeomorphisms (Lemma \ref{upp}), there is a $C^1$ neighborhood $\U_f$ of $f$ such that for any $g\in \U_f$, 
$$
\chi_{\mu}(g^{-1},E^{wu})<-H_f\quad \textrm{for every}~\mu\in G^u(g).
$$
Note that for every $g\in \U_f$, if $x\in \B_g$ then $\omega_{\M}(x,g)\subset G^{cu}(g)\subset G^u(g)$, recalling definitions of $\B_g$ and $G^{cu}(g)$. Thus, we have
$$
\chi_{\mu}(g^{-1},E^{wu})<-H_f\quad \textrm{for every}~\mu\in \omega_{\M}(x,g).
$$
By Lemma \ref{he}, $x$ is $H_f$-nonuniformly expanding along $E^{wu}$. Then, using Proposition \ref{hh}, reducing $\U_f$ if necessary, there exists a density $D_f\in (0,1)$ such that for every $g\in \U_f$, we have
\begin{equation*}
\mathcal{D}_{L}\left(HT(x,{\rm e}^{-H_f/2},g)\right)\ge D_f
\end{equation*}
for every $x\in \B_g$.
\end{proof}

%For every $f\in \U(M)$ and $\sigma>0$, for every $x\in M$, let $\mathcal{HT}(x,\sigma,g)$ be the family of $(E_g^{wu}, \sigma)$-hyperbolic times of $x$ with respect to $g$.
%
%\begin{lemma}
%For every $f\in \V(M)$, there exists $\rho \in (0,1)$, $\alpha>0$ and a $C^1$-open neighborhood $\U$ of $f$ such that for every $g\in \U$, any point of $\B_g$ is $\alpha$-non-uniformly expanding along $E^{wu}=E^u\oplus E^{cu}$. Moreover, we have\begin{equation}\label{est}
%Den^{-}(\mathcal{HT}(x,{\rm e}^{-\alpha/2},g))\ge \rho_H. 
%\end{equation}
%
%\end{lemma}

%Note that for every $f\in \mathcal{V}(M)$, the compactness of $\G^u(f)$ implies that there is $\alpha>0$ such that  
%By Theorem \ref{covering}, for every $f\in \V(M)$ with partially hyperbolic splitting $TM=E^u\oplus_{\succ} E^{cu}\oplus_{\succ} E^{cs}$, for every $\varepsilon>0$, there exist a $C^1$-neighborhood $\U$ of $f$, $\alpha>0$, $\delta_{\ell}>0$, $\rho_{\ell}>0$ for which one can construct  $\mathscr{A}_{\mu}(g,\ell,\alpha,\rho)$ for any $\mu\in \G^u(g)$, $g\in \U$ and $\rho\le \rho_{\ell}$ such that 
%\begin{itemize}
%\item every element $B\in \mathscr{A}_{\mu}(g,\ell,\alpha,\rho)$ contains a hyperbolic periodic point, which we denote by $P_B$. $P_B$ admits 

\begin{theorem}\label{n11}
Let $f\in PH^{1}_{C}(M)$ with partially hyperbolic splitting $TM=E^u\oplus_{\succ} E^{cu}\oplus_{\succ} E^{cs}$.
There exist $\alpha>0$, $\ell\in \NN$, $\theta_f\in (0,1)$, $\rho_f>0$, $H_f>0$ and a $C^1$ neighborhood $\U$ of $f$ such that, for every $\rho\le\rho_f$, for every $g\in \U$, for every $x\in \B_g$ and every $\mu\in \omega_{\M}(x,g)$, 
\begin{enumerate}
\smallskip
\item\label{f1} the forward orbit of $x$ enters to $\mathscr{A}_{\mu}(g,\ell,\alpha,\rho)$ for infinitely many $(E^{wu}, {\rm e}^{-H_f/2})$-hyperbolic times with upper density no less than $\theta_f$. More precisely,
$$
\mathcal{D}_{U}\left(\big\{j\in HT(x,{\rm e}^{-H_f/2},g): g^j(x)\in \bigcup_{B\in \mathscr{A}_{\mu}(g,\ell,\alpha,\rho)}B\big\}\right) \ge \theta_f;
$$
\smallskip
\item\label{f32} there exists $B\in\mathscr{A}_{\mu}(g,\ell,\alpha,\rho)$ such that if $\gamma\ni x$ is a $C^1$ disk transverse to $E^{cs}$, then 
$
\gamma\pitchfork W^s({\rm Orb}(p_B,g))\neq \emptyset.
$
Here $p_B$ is the hyperbolic periodic point in $B$ given by Theorem \ref{covering}.

\end{enumerate}
%\begin{itemize}
%\smallskip
%\item the forward orbit $\{g^n(x)\}$ enters $\mathscr{A}(g,\ell,\alpha,\delta)$ infinitely many times with upper density lager than $\rho_0$. Moreover, we have
%\smallskip
%\item there is $1\le i \le k_{\mu,\delta}$ such that for any $N\in \NN$, there exists $n\ge N$ such that 
%$g^n(x)\in B(x_{\mu,i},\delta)$ and $g^n(\gamma)$ contains a disk of radius $\delta_1$ around $g^n(x)$.
%\end{itemize}
\end{theorem}

\begin{proof}
%By Lemma \ref{u}, we can choose $\alpha_H>0$ and $C^1$ open neighborhood $\U_0\subset \V(M)$ of $f$ such that for every $g\in \U_0$, one has
%$$
%\chi_{\mu}(g^{-1},E_g^{wu})<-\alpha_H\quad \textrm{for every}~\mu\in \G^u(g).
%$$
%For any $g\in \U_0$ and any $x\in \B_g$, for any $\mu \in p\omega(x,g)$, we know $\mu\in \G^u(g)$ by definition of $\B_g$. It follows from Lemma \ref{he} that every point $x$ of $\B_g$ is $\alpha_H$-nonuniformly expanding. Using Proposition \ref{hh}, there exists the density $\rho_H\in (0,1)$ such that 
%\begin{equation}\label{est}
%Den^{-}\left(\mathcal{HT}(x,{\rm e}^{-\alpha_H/2},g)\right)\ge \rho_H. 
%\end{equation}
%for every $x\in \B_g$.
Take $D_f\in (0,1)$, $H_f>0$ and $C^1$ neighborhood $\U_f$ of $f\in PH^{1}_{C}(M)$ as in Proposition \ref{uv}.
Choose $\varepsilon<D_f/2$, by Theorem \ref{covering} there exist $\alpha>0$, $\ell\in \NN$, $\delta_{\ell}$, $\rho_{\ell}>0$ and a $C^1$ neighborhood $\U\subset\U_f$ of $f$ such that for every $\rho \le \rho_{\ell}$, for every $g\in \U$ and every $\mu\in G^u(g)$, one can construct $\mathscr{A}_{\mu}(g,\ell,\alpha,\rho)$, which is consisted of $\rho$-balls satisfying
\begin{itemize}
\smallskip
\item
\begin{equation}\label{ww}
\mu\left(\bigcup_{B\in \mathscr{A}_{\mu}(g,\ell,\alpha,\rho) }B\right)>1-\varepsilon;
\end{equation}
\smallskip
\item
each $B\in  \mathscr{A}_{\mu}(g,\ell,\alpha,\rho)$ contains a hyperbolic periodic point $p_B$ with local stable and unstable manifolds of size larger than $\delta_{\ell}$.
\end{itemize}

\smallskip
By domination together with Lemma \ref{bcd}, one can choose $0<\rho_f<\rho_{\ell}$, $0<\delta_f<\delta_{\ell}$, and $a>0$ such that
\begin{enumerate}[(A)]
\smallskip
\item\label{tde} for any points $x,y$ with $d(x,y)\le 2\rho_{f}$, if $\gamma_x$ and $\gamma_y$ are disks centered at $x$ and $y$ with radius $\delta_{f}$, which are tangent to $\mathcal{C}_a^{wu}$ and $\mathcal{C}_a^{cs}$, respectively, then $\gamma_x$ intersects $\gamma_y$ transversely;
\smallskip
\item\label{bce}
for any $C^1$ embedded disk $\gamma$ tangent to $\mathcal{C}^{wu}_a$ and containing $x\in \B_g$, if $n\in HT(x, {\rm e}^{-H_f/2}, g)$, then 
\begin{equation*}
d_{g^i(\gamma)}(g^{i}(x), g^{i}(y))\le {\rm e}^{(-(n-i)H_f)/4}d_{g^n(\gamma)}(g^n(x), g^{n}(y))
\end{equation*}
for any $y\in \gamma$ satisfying $d_{g^n(\gamma)}(g^n(x), g^n(y))\le \delta_{f}$. 
\end{enumerate}

\smallskip
Now let us fix $g\in \U$, $x\in \B_g$, $\mu \in \omega_{\M}(x,g)$, and $\rho\le \rho_f$. For simplicity,
take $A_{\rho}$ as the union of the $\rho$-balls of $\mathscr{A}_{\mu}(g,\ell,\alpha,\rho)$.
%$$
%A_{\rho}=\bigcup_{B\in \mathscr{A}_{\mu}(g,\ell,\alpha,\rho)}B,
%$$
Due to the choice of $\mu$, there exists a subsequence $\{n_i\}$ such that
$$
\frac{1}{n_i}\sum_{j=0}^{n_i-1}\delta_{g^j(x)}\xrightarrow{weak^{\ast}}\mu,\quad \textrm{as}~i\to +\infty.
$$
This together with (\ref{ww}) yields 
\begin{equation*}\label{s}
\liminf_{i\to +\infty}\frac{1}{n_i}\sum_{j=0}^{n_i-1}\chi_{A_{\rho}}(g^j(x))\ge \mu(A_{\rho})>1-\varepsilon,
\end{equation*}
where we use the fact $\mu\in G^{u}(g)$.
On the other hand, since $x$ is $H_f$-nonuniformly expanding along $E^{wu}$ with density estimate (\ref{est}). According to the choices of $\varepsilon$ and $D_f$, one has
$$
\mathcal{D}_{U}\left(\{j\in HT(x,{\rm e}^{-H_f/2},g): g^j(x)\in A_{\rho}\}\right) \ge D_f-\varepsilon>0.
$$
Therefore, we have complete the proof of Item (\ref{f1}) by taking $\theta_f=D_f-\varepsilon$.

\smallskip
We begin to prove Item (\ref{f32}).
As a consequence of Item (\ref{f1}), one can choose $B\in \mathscr{A}_{\mu}(g,\ell,\alpha,\rho)$ to satisfy
\begin{equation}\label{ff}
\mathcal{D}_{U}\left(\{j\in HT(x,{\rm e}^{-H_f/2},g): g^j(x)\in B\}\right)\ge \frac{\theta_f}{\#\mathscr{A}_{\mu}(g,\ell,\alpha,\rho)}.
\end{equation}
%By domination together with Lemma \ref{bcd}, one can choose $0<\rho_f<\rho_{\ell}$, $0<\delta_f<\delta_{\ell}$, and $a>0$ such that,
%\begin{enumerate}[(A)]
%\smallskip
%\item\label{td} for any points $x,y$ with $d(x,y)\le 2\rho_{f}$, if $\gamma_x$ and $\gamma_y$ are disks tangent to $\mathcal{C}_a^{u}$ and $\mathcal{C}_a^s$ and centered at $x$ and $y$ with radius $\delta_{f}$, respectively, then $\gamma_x$ intersects $\gamma_y$ transversely;
%\smallskip
%\item\label{bc}
%for any $C^1$ embedded disk $\gamma$ tangent to $\mathcal{C}^u_a$ that contains $x\in \B_g$, if $n\in \mathcal{HT}(x, {\rm e}^{-\alpha_H/2}, g)$, then 
%\begin{equation*}
%d_{g^i(\gamma)}(g^{i}(x), g^{i}(y))\le {\rm e}^{(-(n-i)\alpha_H)/4}d_{g^n(\gamma)}(g^n(x), g^{n}(y))
%\end{equation*}
%for any $y\in \gamma$ satisfying $d_{g^n(\gamma)}(g^n(x), g^n(y))\le \delta_{f}$. 
%\end{enumerate}
Let $\gamma$ be a $C^1$ disk transverse to $E^{cs}$, containing the fixed point $x$ in $\B_g$.
By domination, we may assume that every $g^{n}(\gamma)$, $n\ge 1$ is tangent to $\mathcal{C}_a^{wu}$. 
It follows from (\ref{bce}) and  (\ref{ff}) that 
%$g^n(\gamma)$ contains a disk of radius $\delta_f$ around $g^n(x)$, provided $n\in \mathcal{HT}(x, {\rm e}^{-\alpha_H/2}, g)$ is sufficiently large. According to 
% (\ref{ff}), 
there exists $n\in \NN$ such that $g^n(x)\in B$ and $g^n(\gamma)$ contains the disk of radius $\delta_f$ centered at $g^n(x)$. Since $p, g^n(x)\in B$, we know $d(p,g^n(x))<2\rho\le 2\rho_f$, then (\ref{tde}) gives
$
g^n(\gamma)\pitchfork W_{loc}^s(p_B,f)\neq \emptyset.
$
Consequently, we obtain 
$
\gamma\pitchfork W^s({\rm Orb}(p_B,f))\neq \emptyset.
$
Now, we finish the proof of Theorem \ref{n11}.
\end{proof}

We are in position to give the proof of Theorem \ref{es}.
\begin{proof}[Proof of Theorem \ref{es}]
Let $f\in PH^{1}_{C}(M)$ with partially hyperbolic splitting
$
TM=E^u\oplus_{\succ} E^{cu}\oplus_{\succ} E^{cs}.
$
From Lemma \ref{ske1}, it suffices to construct a skeleton of $f$, then it admits the cardinality depending only on $f$, which we denote by $k_f$.

\smallskip
By Theorem \ref{n11}, there exist $\alpha>0$, $\ell \in \NN$, $\rho_f>0$ such that,
for every $x\in \B_f$, when we fix $\mu\in \omega_{\M}(x,f)$, there is $B\in\mathscr{A}(f,\ell,\alpha,\rho_{f})$ with hyperbolic periodic point $p_x:=p_B$ in $B$ so that for every $C^1$ disk $D$ containing $x$ and transverse to $E^{cs}$, we have that $D\pitchfork W^s({\rm Orb}(p_x,f))\neq \emptyset$. 
Denote by $\mathcal{S}_0=\{p_x: x\in \B_f\}$ the set consisting of all the periodic points obtained as above. By Theorem \ref{ttd}, for every $C^1$ disk $D$ transverse to $E^{cs}$, the intersection of $\B_f\cap D$ admits full Lebesgue measure, thus there is $x\in \B_f$ for which $D$ intersects $W^s({\rm Orb}(p_x,f))$ transversally. 
It follows from Theorem \ref{covering} that every hyperbolic periodic point in $\mathcal{S}_0$ admits stable and unstable manifolds of size larger than $\delta_{\ell}$. This implies that the number of periodic points of $\mathcal{S}_0$ which are not homoclinically related to each other is finite only. Therefore, one can take $\mathcal{S}=\{p_i: 1\le i \le l\}\subset\mathcal{S}_0$ such that different elements of $\mathcal{S}$ are not homoclinically related with maximal cardinality, that is, for any $p_x\in\mathcal{S}_0$ there exists 
$p_i\in \mathcal{S}$ so that $p_x\sim p_i$.

\smallskip
Next we show that $\mathcal{S}$ is a pre-skeleton, which implies that there exists a subset of $\mathcal{S}$ to be a skeleton, using Lemma \ref{ske3}. By contradiction, we suppose that there is a $C^1$ disk $D$ transverse to $E^{cs}$ which do not intersect $W^s({\rm Orb}(p,f))$ for any $p\in \mathcal{S}$. On the other hand, by the construction of $\{p_x: x\in \B_f\}$, one can take $p_x$ for some $x\in \B_f\cap D$ such that 
\begin{equation}\label{int}
D\pitchfork W^s({\rm Orb}(p_x,f))\neq \emptyset.
\end{equation}
By the maximality of $\mathcal{S}$, there exists $p\in \mathcal{S}$ such that $p_x\sim p$, which together with (\ref{int}) implies that $D$ intersects $W^s({\rm Orb}(p,f))$ transversally, using the inclination lemma. This is a contradiction to our assumption, thus we get the desired result and complete the proof of Theorem \ref{es}.
\end{proof}

%\begin{notation}
%Since skeletons of the diffeomorphism $f$ must have the same cardinality as shown in Lemma \ref{ske1}, we will use $\kappa_f$ to denote the cardinality of skeleton of $f$ when it make sense.
%\end{notation}

The following theorem focus on the perturbations of the skeleton, which in particular gives the upper semi-continuity of $f\mapsto k_f$ for diffeomorphisms in $PH^{1}_{C}(M)$.

\begin{theorem}\label{ffff}
Let $f\in PH^{1}_{C}(M)$ and $\mathcal{S}(f)=\{p_i: 1\le i \le k_f\}$ be a skeleton of $f$. Then there exists a $C^1$ neighborhood $\U$ of $f$ such that for every $g\in \U$, the continuation $\{p_i(g): 1\le i \le k_f\}$ of $\mathcal{S}(f)$ is a pre-skeleton of $g$. Consequently, we have $k_g \le k_f$ for every $g\in \U$.
\end{theorem}
\begin{proof}
For fixed $f\in PH^{1}_{C}(M)$ with partially hyperbolic splitting $TM=E^u\oplus_{\succ} E^{cu}\oplus_{\succ} E^{cs}$, take the $C^1$ neighborhood $\U$ of $f$ and size $\delta_{\ell}>0$ as in Theorem \ref{n11}.
By Theorem \ref{es}, every diffeomorphism of $\U$ admits a skeleton whose periodic points exhibit local stable and unstable manifolds of size larger than $\delta_{\ell}$. 

\smallskip
Now let $\mathcal{S}(f)=\{p_i: 1\le i \le k_f\}$ be any skeleton of $f$. 
%{\color{red}
%For every disk transverse to $E^{cs}$ of radius larger than $\delta_{\ell}/2$, by definition of skeletons one can find $p_i\in \mathcal{S}(f)$ such that $W^s({\rm Orb}(p_i,f))$ intersects this disk transversely. Then use the compactness of $M$, the continuity of the sub-bundles and stable manifold of periodic point w.r.t. diffeomorphisms, one can reduce $\U$ smaller if necessary such that, for every $g\in \U$ the family $\{p_i(g): 1\le i \le \kappa_f\}$ obtained as the continuation of $\{p_i: 1\le i \le \kappa_f\}$ possesses the following property: for every $C^1$ disk $D$ transverse to $E^{cs}$ with radius larger than $\delta_{\ell}$, there is $1\le i \le \kappa_f$ such that $W^s({\rm Orb}(p_i,f))$ intersects $D$ transversely.
%%$$
%%D\pitchfork W^s({\rm Orb}(p_i(g),g))\neq \emptyset.
%%$$
%For $g\in \U$, let $\mathcal{S}(g)=\{q_i : 1\le i \le \kappa_g \}$ be a skeleton of $g$ such that every $q_i\in \mathcal{S}(g)$ has local unstable manifold of size larger than $\delta_{\ell}$. In particular, for every $q_j\in \mathcal{S}(g)$, there is $1\le i \le \kappa_f$ such that 
%$$
%W^u_{loc}(q_j,g) \pitchfork W^s({\rm Orb}(p_i(g),g))\neq \emptyset.
%$$
%Then, the inclination lemma yields $W^s({\rm Orb}(q_j,g))\subset \overline{W^s({\rm Orb}(p_i(g),g))}$, so $\{p_i(g): 1\le i \le \kappa_f\}$ satisfies condition (S\ref{S1}), and is a pre-skeleton. Therefore, it contains a subset to be a skeleton by Lemma \ref{ske3}, thus $\kappa_g\le \kappa_f$.
%}
%
For every $x\in M$, take a $C^1$ disk transverse to $E^{cs}$ centered at $x$ with radius $\delta_{\ell}/4$, and denote it by $D(x,\delta_{\ell}/4)$. By definition of skeletons, there is $p_i\in \mathcal{S}(f)$ such that $W^s({\rm Orb}(p_i,f))$ intersects $D(x,\delta_{\ell}/4)$ transversely at some point in the interior of $W_{\rho_x}^s({\rm Orb}(p_i,f))$ for some $\rho_x>0$.
Thus, one can take a small neighborhood $U_x$ of $x$ in $M$ such that for every $y\in U_x$, if $\gamma$ is a $C^1$ disk centered at $y$ of radius $\delta_{\ell}/2$ and is transverse to $E^{cs}$, then it intersects $W_{\rho_x}^s({\rm Orb}(p_i,f))$ transversally.
By compactness, we can take an open covering $\{U_{x_i}: 1\le i \le m\}$ of $M$, and put $\rho=\max\{\rho_{x_i}: 1\le i \le m\}$. By the continuity of the sub-bundles and the compact part of stable manifold of periodic point w.r.t. diffeomorphisms, one can reduce $\U$ small enough such that for every $g\in \U$, the family $\{p_i(g): 1\le i \le k_f\}$ obtained as the continuation of $\{p_i: 1\le i \le k_f\}$ possesses the following property: for every $C^1$ disk $D$ transverse to $E^{cs}$ with radius $\delta_{\ell}$ around a point in $M$, there is $1\le i \le k_f$ such that 
$$
D\pitchfork W_{\rho}^s({\rm Orb}(p_i(g),g))\neq \emptyset.
$$

For $g\in \U$, let $\mathcal{S}(g)=\{q_j : 1\le j \le k_g \}$ be a skeleton of $g$ such that every $q_j\in \mathcal{S}(g)$ has local unstable manifold of size larger than $\delta_{\ell}$. In particular, for every $q_j\in \mathcal{S}(g)$, there is $1\le i \le k_f$ such that 
$$
W^u_{loc}(q_j,g) \pitchfork W_{\rho}^s({\rm Orb}(p_i(g),g))\neq \emptyset.
$$
Then, the inclination lemma yields $W^s({\rm Orb}(q_j,g))\subset \overline{W^s({\rm Orb}(p_i(g),g))}$, so $\{p_i(g): 1\le i \le k_f\}$ satisfies condition (S\ref{S1}), and is a pre-skeleton. Therefore, it contains a subset to be a skeleton by Lemma \ref{ske3}, thus $k_g\le k_f$.

\end{proof}

\subsection{Finiteness of ergodic physical measures}
%In this section, we mainly investigate the \emph{ergodic} physical measures by skeletons for $C^{1+\alpha}$  diffeomorphisms in $\U(M)$. 
%Given a diffeomorphism $f$ on $M$, assume that $\mu\in \M_{\rm inv}(f)$ and $p$ a hyperbolic point of $f$, writing $\mu \overset{f}\longleftrightarrow p$ if 
%$$
%{\rm supp}(\mu)=\overline{W^u({\rm Orb}(p,f))}=H(p,f).
%$$
%\paragraph{\bf{Notation}} Given a diffeomorphism $f$ on $M$, assume that $\mu\in \M_{\rm inv}(f)$ and $p$ a hyperbolic point of $f$, writing $\mu \overset{f}\longleftrightarrow p$ if 
%$$
%{\rm supp}(\mu)=\overline{W^u({\rm Orb}(p,f))}=H(p,f).
%$$
For every $f\in PH^{1+}_{C}(M)$, use $\P_{e}(f)$ denotes the family of ergodic physical measures of $f$. Following Theorem \ref{fc}, we know that $\P_e(f)$ is non-empty. Moreover, in terms of the definitions of physical measures and $\B_f$, we have the following observation:

\begin{lemma}\label{kes}
If $f\in PH^{1+}_{C}(M)$, then $\mu \in \P_e(f)$ if and only if it is an ergodic Gibbs $cu$-state.
\end{lemma}

\paragraph{\bf{Notation}} Given a diffeomorphism $f$ on $M$, assume that $\mu\in \M(f)$ and $p$ a hyperbolic periodic point of $f$, writing $\mu \overset{f}\longleftrightarrow p$ if 
$$
{\rm supp}(\mu)=\overline{W^u({\rm Orb}(p,f))}=H(p,f).
$$

We can build the following relationship between ergodic physical measures and the skeleton for any diffeomorphism in $PH^{1+}_{C}(M)$.

\begin{proposition}\label{G}
Assume that $f$ is a diffeomorphism in $PH^{1+}_{C}(M)$ with a skeleton $\mathcal{S}=\{p_i: 1\le i \le k_f\}$. Then there is an injective map $i\mapsto j(i)$ charactered by following property: 
for every $\mu_i\in \P_{e}(f)$, there is $p_{j(i)}\in \mathcal{S}$ and an unstable disk $D_{i}$ such that
\begin{enumerate}
\item\label{p1} $D_i$ is contained in the support of $\mu_i$, and Lebesgue almost every point of it belongs to $B(\mu_i,f)$;
\smallskip
\item\label{p2} $D_i$ intersects $W^s({\rm Orb}(p_{j(i)},f))$ transversally. 
\end{enumerate}
Consequently, we have $\mu_i \overset{f}\longleftrightarrow p_{j(i)}$ for every $1\le i \le k_f$, and $\# \P_e(f)\le k_f$. 
\end{proposition}

\begin{proof} 
By Lemma \ref{ske1}, it is enough to check this result for some skeleton of $f$.
Consider the skeleton $\mathcal{S}(f)=\{p_1,\cdots p_{k_f}\}$ of $f\in PH^{1+}_{C}(M)$ that constructed as in the proof of Theorem \ref{es}, thus each $p_i$, $1\le i \le k_f$ is a hyperbolic periodic point contained in $B(x_i,\rho_{f})\in \mathscr{A}_{\mu}(f,\ell,\alpha,\rho_{f})$ for some $x_i\in {\rm supp}(\mu|\Lambda_{\ell}(f,\alpha))$, $\mu\in G^{cu}(f)$ and fixed $\alpha>0$, $\rho_f>0$ and $\ell\in \NN$. 
Note that by the choice of $\rho_f$, we may assume $\rho_f\ll \min\{\rho_0,\delta_{\ell}\}$. Recalling the construction of cylinders and measurable partitions in Subsection \ref{subty}, there is $r_{\ell}\in (\rho_f,\min\{\rho_0,\delta_{\ell}\})$ such that one can take disks $\gamma_i$ tangent to $\mathcal{C}_a^{wu}$ of radius $r_{\ell}$ centered at $x_i$ to build the cylinders and measurable partitions as follows
$$
C_i=C(x_i,\gamma_i,r_{\ell}),\quad\P_i=\P(x_i,\gamma_i,\alpha,\ell,\rho_f),\quad i=1,\cdots k_f.
$$
We observe that both $W^s_{\delta_{\ell}}(p_i,f)$ and $W^u_{\delta_{\ell}}(p_i,f)$ cross the cylinder $C_i$ for every $1\le i \le k_f$.
%Therefore, following subsection \ref{sub3}, one can construct a measurable partition $\P_i$ attached to $B(x_i,\rho_f)$, which is consisted of unstable manifolds for points in $\mathcal{H}_{yp}(f,\ell,\alpha)$ with uniform sizes.

\smallskip
For every $\mu_i\in \P_{e}(f)$, since it is an ergodic Gibbs $cu$-state guaranteed by Lemma \ref{kes}, using Lemma \ref{ggh} there is a $C^1$ unstable disk $D_{i}\subset {\rm supp}(\mu_i)$ so that ${\rm Leb}_{D_i}$-almost every point of $D_i$ is contained in $B(\mu_i,f)$. Moreover, it follows from condition (S\ref{S1}) of Definition \ref{S} that there exists $p_{j(i)}\in \mathcal{S}(f)$ such that 
\begin{equation}\label{in}
D_{i}\pitchfork W^s({\rm Orb}(p_{j(i)},f))\neq \emptyset.
\end{equation}
Hence, we have found $D_i$ and $p_{j(i)}$ w.r.t. each ergodic physical measure $\mu_i$ to satisfy properties (\ref{p1}) and (\ref{p2}).
 
Now we show $\mu_i \overset{f}\longleftrightarrow p_{j(i)}$ for every $1\le i \le k_f$. Since $D_{i}\subset {\rm supp}(\mu_i)$, and ${\rm supp}(\mu_i)$ is compact and invariant, by applying the inclination lemma, one gets 
\begin{equation}\label{subse}
W^u({\rm Orb}(p_{j(i)},f))\subset {\rm supp}(\mu_i)
\end{equation}
immediately. Observe that by definition each $\P_i$, $1\le i \le k_f$ is a typical cylinder associated to some Gibbs $cu$-state, which together with Lemma \ref{ggh} implies that one can choose $D_{p_{j(i)}}\in \P_{j(i)}$ such that 
\begin{equation}\label{absd}
{\rm Leb}_{D_{p_{j(i)}}}(\Lambda_{\ell}(f,\alpha))>0.
\end{equation}
Note that all the points in $\Lambda_{\ell}(f,\alpha)$ admit local stable manifolds of size larger than $\delta_{\ell}$.  Using inclination lemma to (\ref{in}), we can take $n\in \NN$ large enough such that there is a disk $\widetilde{D}_i\subset f^n(D_{i})$ that is $C^1$-close enough to $W_{\delta_{\ell}}^u(p_{j(i)},f)$ and crosses the cylinder $C_{j(i)}$. 
%This together with (\ref{absd}) and the absolute continuity of stable foliation yields
%\begin{equation}\label{ssg}
%{\rm Leb}_{\gamma}\left(\gamma\cap \bigcup_{x\in D_{p_{j(i)}}\cap \Lambda_{\ell}(f,\alpha) } W^s_{\delta_{\ell}}(x, f)\right)>0,
%\end{equation}
%whenever $\gamma$ is $C^1$ and intersects $W^s_{\delta_{\ell}}(x, f)$ for every $x\in D_{p_{j(i)}}\cap \Lambda_{\ell}(f,\alpha)$. In particular, (\ref{ssg}) holds true for both $\gamma=D_{\mu_i}$ and $\gamma=W^u_{\delta_{\ell}}(p_{j(i)},f)$.
As Lebesgue almost every point of $D_{i}$ is contained in $B(\mu_i,f)$, Lebesgue almost every point of $\widetilde{D}_i$ is contained in $B(\mu_i,f)$ as well. Observe that any stable manifold can only lie in the same basin of some invariant measure, using (\ref{absd}) and the absolute continuity of stable lamination (see e.g. \cite[$\S$8.6]{bp07}), we get
$$
{\rm Leb}_{W^u_{\delta_{\ell}}(p_{j(i)},f)}(B(\mu_i,f))>0.
$$
This implies that one can pick a point $x$ in $W^u_{\delta_{\ell}}(p_{j(i)},f)\cap B(\mu_i,f)$, thus
$$
\frac{1}{n}\sum_{k=0}^{n-1}\delta_{f^k(x)}\xrightarrow{weak^*} \mu_i,\quad \textrm{as}~ n\to +\infty.
$$
Hence, for any $y\in {\rm supp}(\mu_i)$ and any open neighborhood $U$ of $y$, there exist infinitely many times $k$ such that $f^k(x)\in U$, and thus $U\cap W^u({\rm Orb}(p_{j(i)},g))\neq \emptyset$. Therefore, ${\rm supp}(\mu_i) \subset \overline{W^u({\rm Orb}(p_{j(i)},f))}$. This combines with (\ref{subse}) and Lemma \ref{ske2} yields $\mu_i \overset{f}\longleftrightarrow p_{j(i)}$.

\smallskip
Now we show that the map $i\mapsto j(i)$ is injective. Otherwise, there are two different ergodic physical measures $\nu_1$ and $\nu_2$, $p_{j}\in \mathcal{S}(f)$ for some $1\le j \le k_f$ and $C^1$ unstable disks $\{D_i\}_{i=1, 2}$ satisfying properties (\ref{p1})-(\ref{p2}) as follows:
\begin{itemize}
\smallskip
\item[--] $D_i\subset {\rm supp}(\nu_i)$, and ${\rm Leb}_{D_i}$-almost every point belongs to $B(\nu_i,f)$;
\smallskip
\item[--] $D_i\pitchfork W^s({\rm Orb}(p_j,f))\neq \emptyset$.
\end{itemize}
The second property guarantees that, there exist $n_1, n_2\in \NN$ and the disks $D_{n_i}\subset f^{n_i}(D_i)$, $i=1,2$ that are $C^1$-close enough to $W^u_{\delta_{\ell}}(p_j,f)$ and cross the cylinder $C_j$. Then the first property ensures that Lebesgue almost every point of $D_{n_i}$ is contained in $B(\nu_i,f)$, $i=1,2$. Note that $p_j\in B(x_{j},\rho_f)$, for the measurable partition $\mathcal{P}_{j}$, one can take $D_{p_j}\in \mathcal{P}_{j}$ such that $
{\rm Leb}_{D_{p_j}}(\Lambda_{\ell}(f,\alpha))>0$ by Lemma \ref{ggh}. Then using the absolute continuity of stable lamination again, similar to above argument we conclude that $B(\nu_1,f)\cap B(\nu_2,f)\neq \emptyset$, which is a contradiction to our assumption. Thus, we know that the map $i\mapsto j(i)$ is injective. As a direct result, we get $\# \P_e(f)\le k_f$.
\end{proof}

%As a direct result, we have the following corollary.

%\begin{corollary}\label{no}
%We have $\# \P_e(f)\le \kappa_f$ for every $f\in \mathcal{PH}^{1+}_{\mathcal{C}}(M)$.
%\end{corollary}

\subsection{Ergodicity and basin covering property of physical measures}\label{fixedone}
%A sequence of measures $\mu_1,\cdots,\mu_k$ has the \emph{basin covering property} if the union of the basins $B(\mu_i,f)$, $1\le i \le k$ has full Lebesgue measure on $M$.
Recall that in the previous work \cite{MCY17}, for diffeomorphisms in $PH^{1+}_{EC}(M)$, the authors have proved the finiteness and basin covering property of physical measures. In this subsection, by making more accurate analysis on Gibbs $cu$-states, we provide a different proof of these results.
In consideration of Theorem \ref{utv}, it suffices to prove them for diffeomorphisms in $PH^{1+}_{C}(M)$.

\begin{theorem}\label{ky}
For every $f\in PH^{1+}_{C}(M)$, we have 
\begin{enumerate}
\smallskip
\item\label{ps1} all physical measures are ergodic and coincide with the extreme elements of $G^{cu}(f)$;
\smallskip
\item\label{ps2} Lebesgue almost every point of $M$ is contained in some basin of physical measures.
\end{enumerate}
\end{theorem}
%
%
%Let $f\in \U(M)$ of class $C^{1+}$, given $\xi\in (0,1)$ and a $C^1$ disk $D$ tangent to $\mathcal{C}_a^F$, define the Holder curvature as 
%$$
%\mathscr{K}_{\xi}(D)=\inf\{C>0: TD~ \textrm{is}~ (C,\xi)-\textrm{H\"{o}lder continuous}\}, 
%$$
%where we say $TD$ is $(C,\xi)$-H\"{o}lder if $dist(T_xD,T_yD)\le Cd(x,y)$ for all $x,y\in B(x,\delta_M)\cap D$ and $x\in D$. According to \cite[Corollary 4.2]{ABV00}, we know that there exists $\xi>0$ such that for any $C^1$-disk $D$ transverse to $E^{cs}$, all its iterates $f^n(D)$ have the bounded H\"{o}lder curvature w.r.t. $\xi$, for all $n$ sufficiently large. Due to this fact, without loss of generality in the discussion below we will always assume that the $C^1$-disk transverse to $E^{cs}$ having the bounded H\"{o}lder curvature. 

\paragraph{\bf{Fixed constants}}
Given $f\in PH^{1+}_{C}(M)$ with partially hyperbolic splitting $TM=E^u\oplus_{\succ} E^{cu}\oplus_{\succ} E^{cs}$. 
Let us fix $H_f, D_f$ as in Proposition \ref{uv}, and choose $\varepsilon<D_f/2$. For this $\varepsilon>0$, one chooses $\alpha>0$, $\ell\in \NN$, $\rho_{\ell}>0$, $\delta_{\ell}>0$ as in Theorem \ref{covering}. 
%Thus, we have
%$$
%\mu\left(\bigcup_{B\in \mathscr{A}_{\mu}( f,\ell, \alpha, \rho_{\ell})}B\right)>1-\varepsilon\quad \textrm{for every}~\mu \in \G^u(f).
%$$
Furthermore, by Lemma \ref{bcd}, up to shrinking $\delta_{\ell}$, there exists $a>0$ such that we have the following \emph{backward contracting property}:
if $n$ is a $(E^{wu}, {\rm e}^{-H_f/2})$-hyperbolic time of $x$, and $\gamma$ is a $C^1$ disk tangent to $\mathcal{C}_a^{wu}$, then
\begin{equation}\label{bcdd}
d_{f^{n-i}\gamma}(f^{n-i}(y), f^{n-i}(z))\le {\rm e}^{-iH_f/4}d_{f^n\gamma}\left(f^n(y),f^n(z)\right)
\end{equation}
whenever $y,z\in \gamma$ within $f^n(x), f^n(y)$ contained in $B_{f^n(\gamma)}(f^n(x),\delta_{\ell})$.
Let us mention that from the choices of $\rho_{\ell}$ and $\delta_{\ell}$, we can and will assume $\rho_{\ell}\ll \delta_{\ell}$.

%By Theorem \ref{fc}, 
%all the extreme elements of $\G^{cu}(f)$ are ergodic, they are also ergodic physical measures guaranteed by Lemma \ref{kes}. Therefore, applying Corollary \ref{no} we know that there are finitely many extreme element of $\G^{cu}(f)$. Denote these extreme elements as $\mu_1,\cdots,\mu_k$ for some $k\le k_f$.
%Denote
%$$
%\mathbb{B}=\bigcup_{1\le i \le k}\bigcup_{B\in \mathscr{A}_{\mu_i}( f,\ell, \alpha, \rho_{\ell})}B.
%$$
%For $x\in \B_f$, denote the family of hyperbolic times that enters $\BB$ by
%$$
%\mathcal{HT}_{\mathbb{B}}(x, {\rm e}^{-H_f/2},f)=\left\{n: n\in \mathcal{HT}(x, {\rm e}^{-H_f/2},f)~ \textrm{and}~ f^n(x)\in \mathbb{B}\right\}.
%$$

%We will prove the following result. 
%
%\begin{theorem}\label{ky}
%For every $C^{1+}$ diffeomorphism $f\in \V(M)$, we have 
%\begin{enumerate}
%\smallskip
%\item\label{pp1} there are finitely many physical measures, which are ergodic and coincide with the extreme elements of $\G^{cu}(f)$;
%\smallskip
%\item\label{pp2} Lebesgue almost every point of $M$ is contained in some basin of physical measures.
%\end{enumerate}
%\end{theorem}

%\begin{proposition}\label{ky}
%Every physical measure of $C^{1+\alpha}$ diffeomorphism $f\in \U(M)$ is an extreme point of $\G^{cu}(f)$. As a result, we know $\P_e(f)=\P(f)$. 
%\end{proposition}

\bigskip
Following \cite[Corollary 4.2]{ABV00}, for any $C^2$ disk $D$ transverse to $E^{cs}$, the iterate $f^n(D)$ admits the \emph{bounded H\"{o}lder curvature} for all $n$ sufficiently large, see \cite[$\S$ 2.1]{ABV00} for more details. Therefore, without loss of generality, we assume that any $C^2$ disk transverse to $E^{cs}$ considered in the rest of this subsection exhibits the bounded H\"{o}lder curvature.

%By applying the backward contracting property, we have the bounded distortion property at hyperbolic times.
%
%\begin{lemma}\label{bdo}
%There exist constants $\mathcal{C}_0>0, \mathcal{C}_1>0$ such that if $D$ is a $C^1$ disk tangent to $\mathcal{C}_a^u$,
%%and all its iterates have the bounded H\"{o}lder curvature, 
%then for every $n\in \mathcal{HT}(x, {\rm e}^{-H_f/2},f)$, we have
%$$
%\frac{1}{\mathcal{C}_0}\le \frac{\left|{\rm det}Df^n|T_yD\right|}{\left|{\rm det}Df^n|T_zD\right|}\le \mathcal{C}_0,
%$$
%whenever $y,z\in D$ satisfying $f^n(y), f^n(z)\in B_{f^n(D)}(f^n(x),\delta_{\ell})$. Moreover, for any measurable subsets $A, B \subset f^{-n}\left(B_{f^n(D)}(f^n(x), \delta_{\ell})\right)$, 
%$$
%\frac{1}{\mathcal{C}_1}\frac{{\rm Leb}_D(A)}{{\rm Leb}_D(B)}\le \frac{{\rm Leb}_{f^n(D)}(f^n(A))}{{\rm Leb}_{f^n(D)}(f^n(B))}\le \mathcal{C}_1 \frac{{\rm Leb}_D(A)}{{\rm Leb}_D(B)}
%$$
%\end{lemma}

By Theorem \ref{fc}, 
all the extreme elements of $G^{cu}(f)$ are ergodic, they are also ergodic physical measures. Hence, applying Proposition \ref{G} we know that there are at most $k_f$ extreme elements in $G^{cu}(f)$. Denote them by $\mu_1,\cdots,\mu_k$ for some $k\le k_f$.
Rewriting 
$$
\mathscr{A}_i:=\mathscr{A}_{\mu_i}(f,\ell,\alpha,\rho_{\ell}), \quad i=1,\cdots, k,
$$
and 
$$
\BB:=\bigcup_{1\le i \le k}\bigcup_{B\in \mathscr{A}_{i}}B.
$$
For $x\in \B_f$, let us denote 
\begin{equation}\label{htb}
HT_{\BB}(x, {\rm e}^{-H_f/2},f)=\left\{n: n\in HT(x, {\rm e}^{-H_f/2},f)~ \textrm{and}~ f^n(x)\in \mathbb{B}\right\}.
\end{equation}

\begin{lemma}\label{density}
For every $x\in \B_f$, we have $\mathcal{D}_{U}(HT_{\BB}(x, {\rm e}^{-H_f/2},f)) \ge D_f/2$.
\end{lemma}
\begin{proof}
Take $\mu\in \omega_{\M}(x,f)$, which belongs to $G^{cu}(f)$ by definition. Therefore, it can be written as the convex combination of the extreme elements of $G^{cu}(f)$, thus there exists a sequence of real numbers $\xi_i\in [0,1]$ for $1\le i \le k$ satisfying $\xi_1+\cdots +\xi_k=1$ so that
$
\mu=\xi_1 \mu_1+\cdots +\xi_k \mu_k.
$
Due to the construction of $\mathscr{A}_{i}$ in Theorem \ref{covering} we know 
$
\mu_i\left(\cup_{B\in \mathscr{A}_{i}}B\right)>1-\varepsilon
$
for every $1\le i \le k$.
Consequently, 
\begin{equation}\label{fd}
\mu(\BB)>1-\varepsilon.
\end{equation}
As $\mu\in \omega_{\M}(x,f)$, we can choose subsequence $\{n_l\}$ tends to infinity such that 
$$
\frac{1}{n_l}\sum_{i=0}^{n_l-1}\delta_{f^i(x)}\xrightarrow{weak^*} \mu,\quad \textrm{as}~l\to +\infty.
$$
Combined with (\ref{fd}) it follows that the times $n$ for which $f^n(x)$ enters to $\mathbb{B}$ admits upper density estimate as follows:
\begin{equation*}\label{e2}
\mathcal{D}_{U}\left(\{n \in \NN: f^n(x)\in \BB \}\right) \ge 1-\varepsilon.
\end{equation*}
This together with the choices of $\varepsilon$ and $D_f$ yields
$$
\mathcal{D}_{U}\left(HT_{\BB}(x, {\rm e}^{-H_f/2},f)\right) \ge D_f/2.
$$
\end{proof}

By Lemma \ref{density}, we can deduce the following result which is a variation of \cite[Lemma 5.4]{AP08}, thus we omit the proof.
%We have following result for $C^1$ disk that intersects $\B_f$, we prove it by borrowing some ideas from \cite[Lemma 5.4]{AP08} 

\begin{lemma}\label{gtg}
Assume that $D$ is a $C^2$ disk transverse to $E^{cs}$. Let $A$ be a subset of $\B_f$ and $U$ an open subset in $D$ satisfying ${\rm Leb}_D(A\cap U)>0$. For any $\eta>0$, there exists $n\in HT_{\mathbb{B}}(x, {\rm e}^{-H_f/2},f)$ for some $x\in A$, such that $f^n(D)$ contains a ball $B_n$ of radius $\delta_{\ell}/4$ centered at $f^n(x)$, and we have 
$$
\frac{{\rm Leb}_{f^n(D)}\left(f^n(A)\cap B_n\right)}{{\rm Leb}_{f^n(D)}(B_n)}>1-\eta.
$$
\end{lemma}

Now we can give the proof of Theorem \ref{ky}.

\begin{proof}[Proof of Theorem \ref{ky}]
Let $\mu_1,\cdots,\mu_k $ be the extreme elements of $G^{cu}(f)$, and take constants $\alpha>0$, $\ell\in \NN$, $\rho_{\ell}>0$, $a>0$ as fixed after the statement of Theorem \ref{ky}. 
%Rewriting 
%$$
%\mathscr{A}_i:=\mathscr{A}_{\mu_i}(f,\ell,\alpha,\rho_{\ell}), \quad i=1,\cdots, k.
%$$
Recall that $\mathscr{A}_i$ is the collection of finitely many $\rho_{\ell}$-balls that covers ${\rm supp}(\mu_i|\Lambda_{\ell}(f,\alpha))$, and the union of $\rho_{\ell}$-balls from $\mathscr{A}_i$, $1\le i \le k$ is denoted by $\BB$. Note also $HT_{\BB}(x, {\rm e}^{-H_f/2},f)$ is given by (\ref{htb}).
%Let $\mathscr{A}$ be the collection of all the balls from $\mathscr{A}_i:=\mathscr{A}_{\mu_i}(f,\ell,\alpha,\rho_{\ell})$, $1\le i \le k$. 
For this $\rho_{\ell}$ and $\delta_{\ell}$, by domination there exists $r_{\ell}\in (\rho_{\ell},\delta_{\ell})$ such that for each $B:=B(x,\rho_{\ell})\in \mathscr{A}_i$, $1\le i \le k$, 
\begin{enumerate}[(I)]
\smallskip
\item one can take a $C^1$ disk $D$ tangent to $\mathcal{C}_a^{wu}$ centered at $x$ of radius $r_{\ell}$, with which we build cylinder and measurable partition as follows:
$$
C_B:=C(x,D,r_{\ell}),\quad \mathcal{P}_B=\mathcal{P}(x,D,\alpha,\ell,\rho_{\ell});
$$
Recalling their constructions in Subsection \ref{subty}.
\smallskip
\item\label{cro} If $\gamma$ is a $C^1$ sub-manifold tangent to $\mathcal{C}_a^{wu}$ which contains a disk of radius $\delta_{\ell}/4$ centered at a point of $B$, then this disk must cross $C_{B}$.
\end{enumerate}
For each $B\in \mathscr{A}_i$,
since we have $\mu_i(B\cap \Lambda_{\ell}(f,\alpha))>0$, $\mathcal{P}_B$ is a typical measurable partition associated to $\mu_i$, so there exists $\gamma_{B}\in \mathcal{P}_B$ such that 
$$
{\rm Leb}_{\gamma_B}(B(\mu_i,f)\cap\Lambda_{\ell}(f,\alpha))>0.
$$
Note that all points of $\Lambda_{\ell}(f,\alpha)$ admit local stable manifolds of size larger than $\delta_{\ell}$, and the   points of a stable manifold can only lie in the same basin. Then, in view of (\ref{cro}), by using the absolute continuity of stable lamination, we can take $L_0>0$ so that 
%$$
%\frac{{\rm Leb}_{\gamma}(B_{\gamma}(x,\delta_{\ell}/4)\cap B(\mu_i,f))}{{\rm Leb}_{\gamma}(B_{\gamma}(x,\delta_{\ell}/4))}\ge \mathscr{L}_B.
%$$
%for every $C^2$ sub-manifold $\gamma$ that contains the ball of radius $\delta_{\ell}/4$ centered at $x\in B$.
%Consequently, by taking $\mathscr{L}=\max\{\mathscr{L}_B: B\in \mathscr{A}\}$ we have
\begin{equation}\label{mmmm}
\frac{{\rm Leb}_{\gamma}(B_{\gamma}(x,\delta_{\ell}/4)\cap B(\mu_i,f))}{{\rm Leb}_{\gamma}(B_{\gamma}(x,\delta_{\ell}/4))}\ge L_0
\end{equation}
whenever $x\in B \in \mathscr{A}_i$ and $\gamma$ is a $C^2$ sub-manifold containing the ball of radius $\delta_{\ell}/4$ centered at $x$. Here we use the finiteness of balls in $\mathscr{A}_i$, $1\le i \le k$.

\medskip
We are ready to give the proof of Item (\ref{pp1}) firstly.
Let $\mu$ be a physical measure of $f$, then $B(\mu,f)$ exhibits positive Lebesgue measure, and thus $\B_f\cap B(\mu,f)$ has positive Lebesgue measure. Therefore, we know $\mu\in G^{cu}(f)$, which also implies $B(\mu,f)\subset \B_f$ by definition. Next we want to prove $\mu=\mu_i$ for some $1\le i \le k$. 

\smallskip
Pick a density point of $B(\mu,f)$ and foliate the small neighborhood of it with smooth sub-manifolds transverse to $E^{cs}$; then using the Fubini's theorem, one can find a $C^2$ disk $D$ lie in some sub-manifold such that ${\rm Leb}_D(O\cap B(\mu,f))>0$ for some open set $O$ in $D$.
By taking $\eta<L_0$ and $A:=B(\mu,f)$, Lemma \ref{gtg} implies that there exist $x\in B(\mu,f)$ and $n\in HT_{\BB}(x, {\rm e}^{-H_f/2},f)$ sufficiently large such that 
$$
\frac{{\rm Leb}_{f^n(D)}\left(B_{f^n(D)}(f^n(x),\delta_{\ell}/4)\cap B(\mu ,f)\right)}{{\rm Leb}_{f^n(D)}(B_{f^n(D)}(f^n(x),\delta_{\ell}/4))}>1-\eta,
$$
where we use the fact that $B(\mu,f)$ is $f$-invariant.
Note also that $f^n(x)$ is contained in some ball of $\mathscr{A}_{i}$ for some $1\le i \le k$, we have the estimate (\ref{mmmm}) by taking $x:=f^n(x)$ and $\gamma:=f^n(D)$.
%we get
%$$
%\frac{{\rm Leb}_{f^n(D)}\left(B_{f^n(D)}(f^n(x),\delta_{\ell}/4)\cap B(\mu_i ,f)\right)}{{\rm Leb}_{f^n(D)}(B_{f^n(D)}(f^n(x),\delta_{\ell}/4))}\ge C_0.
%$$
By the choices of $\eta$ and $L_0$, one must have
$$
{\rm Leb}_{f^n(D)}\left(B_{f^n(D)}(f^n(x),\delta_{\ell}/4)\cap B(\mu_i ,f)\cap B(\mu,f)\right)>0,
$$ 
which implies $B(\mu,f)\cap B(\mu_i,f)\neq \emptyset$, and thus $\mu=\mu_i$. This shows that $\mu$ is ergodic and is an extreme element of $G^{cu}(f)$, which is an ergodic physical measure following from Theorem \ref{fc}. This completes the proof of Item (\ref{ps1}).

\medskip
From Item (\ref{ps1}), we know that $\mu_1,\cdots \mu_k$ are all the physical measures of $f$. Because $\B_f$ has full Lebesgue measure by Theorem \ref{ttd}, to see Item (\ref{pp2}), it suffices to prove that Lebesgue almost every point of $\B_f$ is contained in $B(\mu_i,f)$ for some $1\le i \le k$. Assume by contradiction that $\mathcal{N}=\B_f \setminus \cup_{1\le i \le k}B(\mu_i,f)$ exhibits positive Lebesgue measure, note that $\mathcal{N}$ is an $f$-invariant subset of $\B_f$. Applying Lemma \ref{gtg} and (\ref{mmmm}) to $A:=\mathcal{N}$ and $\eta<L_0$, with the same argument as in the proof of Item (\ref{ps1}), one obtains 
%that, there is $x\in N$ and some $n\in \mathcal{HT}_{\BB}(x, {\rm e}^{-H_f/2},f)$ such that 
%$$
%\frac{{\rm Leb}_{f^n(D)}\left(B_{f^n(D)}(f^n(x),\delta_{\ell}/4)\cap N\right)}{{\rm Leb}_{f^n(D)}(B_{f^n(D)}(f^n(x),\delta_{\ell}/4))}>1-\eta.
%$$
%On the other hand, since $f^n(x)$ is contained in some $\mathscr{A}_i$ for some $1\le i \le k$, one has
%$$
%\frac{{\rm Leb}_{f^n(D)}\left(B_{f^n(D)}(f^n(x),\delta_{\ell}/4)\cap B(\mu_i,f)\right)}{{\rm Leb}_{f^n(D)}(B_{f^n(D)}(f^n(x),\delta_{\ell}/4))}\ge \mathscr{L}
%$$
%following from (\ref{mmmm}). Therefore, in view of the choice of $\mathscr{L}$ and $\eta$, we obtain 
$B(\mu_i,f)\cap \mathcal{N}\neq \emptyset$ for some $1\le i \le k$. This is a contradiction to the definition of $\mathcal{N}$, thus we complete the proof of Item (\ref{ps2}).
\end{proof}

\subsection{Skeleton generated by physical measures: Proof of Theorem \ref{TheoF}}\label{proofD}
We have given the existence of skeletons for diffeomorphisms in $PH^{1}_{C}(M)$ by Theorem \ref{es}. Moreover, among $PH^{1+}_{C}(M)$, Proposition \ref{G} and Theorem \ref{ky} tell us that each physical measure $\mu_i$ is attached with a periodic point $p_{j(i)}$ in the skeleton, and $i\mapsto j(i)$ is injective. Thus, for proving Theorem \ref{TheoF}, it remains to show the map $i\mapsto j(i)$ is indeed bijective. To this end, we will establish the skeleton in a finer way than those in Theorem \ref{es}.

\begin{proof}[Proof of Theorem \ref{TheoF}]
By Theorem \ref{ky}, consider $\mu_1,\cdots,\mu_k $ for some $k\le \kappa_f$ as the family of physical measures of $f$, which are also the extreme elements of $G^{cu}(f)$. Fixing $H_f>0$, $\alpha>0$, $0<\varepsilon<D_f/2$, $\ell\in \NN$, $0<\rho_{\ell}\ll \delta_{\ell}$ as in Subsection \ref{fixedone}. As before, for each $1\le i \le k$, let $\mathscr{A}_i$ be the family of finitely many $\rho_{\ell}$-balls associated to $\mu_i$. Let $\BB$ be the union of $\rho_{\ell}$-balls from $\mathscr{A}_i$, $1\le i \le k$, and the notion $HT_{\BB}(x, {\rm e}^{-H_f/2},f)$ is given by (\ref{htb}). For the sake of more detailed argument, putting
\begin{itemize}
\smallskip
\item $\#(\mathscr{A}_i)=L_i$;
\smallskip
\item $\mathscr{A}_i=\{B_{i,j}: 1\le j \le L_{i}\}, \quad i=1,\cdots,k$;
\smallskip
\item Let $p_{i,j}:=p_{B_{i,j}}$ be the hyperbolic periodic point in $B_{i,j}$ given by Theorem \ref{covering}.
\end{itemize}
As established in the proof of Theorem \ref{ky}, consider $C_{i,j}:=C_{B_{i,j}}$ and $\mathcal{P}_{i,j}:=\mathcal{P}_{B_{i,j}}$ as the cylinder and measurable partition around $B_{i,j}$, respectively.
%We will prove the following crucial steps:
%\begin{enumerate}[\it{Step} 1:]
%\smallskip
%\item\label{st1} $\mathcal{S}_0=\{p_{i,j}: 1\le i \le k, 1\le j \le L_i\}$ is a pre-skeleton of $f$;
%\smallskip
%\item \label{st2} If $\mathcal{S}_f\subset \mathcal{S}_0$ is a skeleton of $f$, then for each $1\le i \le k$, there exists at least one $1\le j \le L_i$ such that $p_{i,j}\in \mathcal{S}_f$;
%\smallskip
%\item \label{st3} Furthermore, we show for each $1\le i \le k$, there exists only one $1\le j \le L_i$ such that $p_{i,j}\in \mathcal{S}_f$.
%\end{enumerate}

\begin{lemma}\label{st1}
$\mathcal{S}_0=\{p_{i,j}: 1\le i \le k, 1\le j \le L_i\}$ is a pre-skeleton of $f$.
\end{lemma}

\begin{proof}
Let us take any $C^1$ disk $D$ transverse to $E^{cs}$, we need to find $p_{i,j}\in \mathcal{S}_0$ such that $D\pitchfork W^s({\rm Orb}(p_{i,j},f))\neq \emptyset$. From (\ref{fde}) of Theorem \ref{ttd}, one can pick $x\in D\cap \B_f$, using Lemma \ref{density} we have 
%and thus any fixed $\mu\in \omega_{\M}(x,f)$ is a Gibbs $cu$-state, by definition of $\B_f$. Since $\mu_1,\cdots,\mu_k$ are extreme elements of $\G^{cu}(f)$, $\mu$ is a convex combination of them, so there exist $\xi_i\in [0,1]$, $1\le i\le k$, within $\sum_{i=1}^k\xi_i=1$ such that
%\begin{equation}\label{sum}
%\mu=\sum_{i=1}^k\xi_i\mu_i.
%\end{equation}
%Recall the choice of $\mathscr{A}_i$, $1\le i \le k$, we know that the union of $\rho_{\ell}$-balls from $\mathscr{A}_i$ has $\mu_i$-measure no less than $1-\varepsilon$. In view of representation (\ref{sum}) and $\varepsilon<D_f/2$, as in the proof of Theorem \ref{gtg}, one obtains 
$$
\mathcal{D}_{U}\left(HT_{\BB}(x, {\rm e}^{-H_f/2},f)\right)\ge D_f/2.
$$
According to (\ref{bcdd}), there is $n\in HT_{\BB}(x, {\rm e}^{-H_f/2},f)$ such that $f^n(D)$ contains a disk of radius $\delta_{\ell}$ centered at $f^n(x)\in \BB$. Therefore, there is some $B_{i,j}\in \mathscr{A}_i$ such that $f^n(D)$ crosses the cylinder $C_{i,j}$. In particular, we have 
$$
f^n(D)\pitchfork W^s_{\delta_{\ell}}(p_{i,j},f)\neq \emptyset.
$$
Hence, we get $D\pitchfork W^s({\rm Orb}(p_{i,j},f))\neq \emptyset$ and complete the proof of Lemma \ref{st1}.
\end{proof}

Lemma \ref{ske3} guarantees that $\mathcal{S}_0$ contains a subset to be a skeleton. Furthermore, we have:
\begin{lemma}\label{st2}
If $\mathcal{S}_f\subset \mathcal{S}_0$ is a skeleton of $f$, then for each $1\le i \le k$, there exists a unique $1\le j_i \le L_i$ such that $p_{i,j_i}\in \mathcal{S}_f$.
\end{lemma}
\begin{proof}
Let $\mathcal{S}_f\subset\mathcal{S}_0$ be a skeleton of $f$. Now we show firstly that for each $1\le i \le k$, different elements of $\{p_{i,j}: 1\le j \le L_i\}$ are homoclinically related to each other. Given $1\le i \le k$, let us take $1\le m,n\le L_i$ with $m\neq n$. 
Recalling the construction of balls in $\mathscr{A}_i$, 
suppose
$$
B_{i,m}:=B(x_m, \rho_{\ell}),\quad B_{i,n}:=B(x_n, \rho_{\ell})
$$
for some $x_m, x_n\in {\rm supp}(\mu_i|\Lambda_{\ell}(f,\alpha))$.
Therefore, for any $\rho>0$, we have 
$$
\mu_i(B(x_m, \rho)\cap \Lambda_{\ell}(f,\alpha))>0,\quad \mu_i(B(x_n, \rho)\cap \Lambda_{\ell}(f,\alpha))>0.
$$
By Lemma \ref{sp1}, one can find integers $K<N$ and a hyperbolic periodic point $p$ of periodic $n$ with properties as follows:
\begin{itemize}
\smallskip
\item[--] $p\in B_{i,m}$ and $f^K(p)\in B_{i,n}$;
\smallskip
\item[--] both $p$ and $f^K(p)$ have stable and unstable manifolds of size larger than $\delta_{\ell}$.
\end{itemize}
Since $\rho_{\ell}\ll \delta_{\ell}$, the above properties imply $p\sim p_{i,m}$ and $f^K(p)\sim p_{i,n}$, thus $p_{i,n}\sim p_{i,m}$. 
Now, to complete the proof it suffices to show that for each $1\le i\le k$, there exists at least one $1\le j \le L_i$ such that $p_{i,j}\in \mathcal{S}_f$.
Suppose by contradiction that there is $1\le i_0 \le k$ such that $p_{i_0,j}\notin \mathcal{S}_f$ for any $1\le j \le L_{i_0}$. From Proposition \ref{G}, we know that for $\mu_{i_0}$, there exists some $p_{i,j}\in \mathcal{S}_f$ with $i\neq i_0$ and an unstable disk $D_{i_0}$ satisfying
\begin{enumerate}
\smallskip
\item\label{pp1} Lebesgue almost every point of $D_{i_0}$ belongs to $B(\mu_{i_0},f)$;
\smallskip
\item\label{pp2} $D_{i_0}$ intersects $W^s({\rm Orb}(p_{i,j},f))$ transversally. 
\end{enumerate}
On the other hand, by Lemma \ref{ggh}, there exists $D_{i,j}\in \mathcal{P}_{i,j}$ such that ${\rm Leb}_{D_{i,j}}(\Lambda_{\ell}(f,\alpha)\cap B(\mu_i,f))>0$.
Applying inclination lemma to (\ref{pp2}), there is $n\in \NN$ so that $f^n(D_{i_0})$ contains a disk $\widetilde{D}_{i_0}$ which is $C^1$-close to $W_{\delta_{\ell}}^s(p_{i,j},f)$ and crosses $C_{i,j}$. Thus, in view of the choice of $D_{i,j}$, using the absolute continuity of stable lamination one must have 
$
{\rm Leb}_{D_{i_0}}(B(\mu_i, f))>0,
$ 
which is a contradiction to (\ref{pp1}).
\end{proof}

By Lemma \ref{st2}, we know that for each $\mu_i$, there exists a unique $p_{i,j_i}\in \mathcal{S}_f$ and an unstable disk $D_{i}$ satisfying
\begin{itemize}
\smallskip
\item[--] $D_i$ is contained in the support of $\mu_i$, and Lebesgue almost every point of $D_i$ belongs to $B(\mu_i,f)$;
\smallskip
\item[--] $D_i$ intersects $W^s({\rm Orb}(p_{i,j_i},f))$ transversally. 
\end{itemize}
Therefore, we have $\mu_i\overset{f}\longleftrightarrow p_{i,j_i}$, as showed in Proposition \ref{G}. Then, using Lemma \ref{ske1} to get the desired result.
\end{proof}
\section{Proofs of the main results}\label{sec7}
In this section we will provide the proofs of Theorems \ref{TheoB}, \ref{TheC} and \ref{sta}.
Now we show firstly the existence of skeletons and the upper semi-continuity of the cardinality of skeleton w.r.t. diffeomorphisms in $PH^{1}_{EC}(M)$.

\begin{theorem}\label{uppers}
Every $f\in PH^{1}_{EC}(M)$ admits a skeleton. If $\mathcal{S}(f)=\{p_i: 1\le i \le k_f\}$ is a skeleton of $f$, then there exists a $C^1$ neighborhood $\U$ of $f$ such that for every $g\in \U$, the continuation $\{p_i(g): 1\le i \le k_f\}$ of $\mathcal{S}(f)$ is a pre-skeleton of $g$. Consequently, we have $k_g \le k_f$ for every $g\in \U$.
\end{theorem}

\begin{proof}
By Theorem \ref{utv}, there is $n_0\in \NN$ such that $f^{n_0}\in PH^{1}_{C}(M)$. Then by the construction of skeletons in Theorem \ref{es}, there exists a $C^1$ neighborhood $\U_0$ of $f^{n_0}$ such that any diffeomorphism in it admits a skeleton, whose periodic points exhibit local stable and unstable manifolds with uniform size. Take a $C^1$ neighborhood $\U$ of $f$ such that $g^{n_0}\in \U_0$ whenever $g\in \U$.
By applying Lemmas \ref{sc} and \ref{ske3}, we obtain that every diffeomorphism in $\U$ admits a skeleton that formed by periodic points admitting local stable and unstable manifolds with uniform size. Using the similar argument as in the proof of Theorem \ref{ffff}, mainly the continuity of the compact part of stable manifolds of periodic points w.r.t. diffeomorphisms, one can conclude the desired result.
\end{proof}

\paragraph{\bf{Notation}}
Let $f$ be a homeomorphism on $M$ and  $\mu$ a probability on $M$, define 
$$
\mu^{(n)}=\frac{1}{n}\sum_{i=0}^{n-1}f_{\ast}^i\mu,\quad n\in \NN.
$$

By definition, if $\mu$ is $f^n$-invariant, then $\mu^{(n)}$ is $f$-invariant. Moreover, one has the next result.

%
%\begin{definition}
%Let $f$ be a homeomorphism on $M$ and  $\mu$ a probability on $M$, define  
%$$
%\mu^{<n>}=\frac{1}{n}\sum_{i=0}^{n-1}f_{\ast}^i\mu,\quad n\in \NN.
%$$
%\end{definition}

\begin{lemma}\label{ic}
Let $f$ be a homeomorphism on $M$, assume that $\mu$ and $\nu$ are $f^n$-ergodic measures. Then we have
that $\mu^{(n)}=\nu^{(n)}$ if and only if $\mu=f_{\ast}^i\nu$ for some $i\in \{0, 1\cdots,n-1\}$.
\end{lemma}
%
%\begin{proof}
%One can verify by construction that if $\mu=f^k_{\ast}\nu$ for some $k\in \NN$, then $\mu^{<n>}=\nu^{<n>}$. Thus it remains to show the other direction.
%Let $\mu$ be an $f^n$-ergodic measure. By construction, we get $B(\mu, f^n)\cup \cdots \cup B(f_{\ast}^{n-1}\mu,f^n)$ has full $\mu^{<n>}$ measure. Indeed, since $\mu$ is $f^n$-ergodic, for each $1\le i \le n-1$, $f_{\ast}^i\mu$ is $f^n$-ergodic as well, which yields $(f_{\ast}^i\mu) B(f_{\ast}^i\mu, f^n)=1$. Therefore, we have
%\begin{eqnarray*}
%\mu^{<n>}\left(\bigcup_{i=0}^{n-1}B(f_{\ast}^{i}\mu,f^n)\right) &=& \frac{1}{n}\sum_{i=0}^{n-1} f_{\ast}^i\mu\left(\bigcup_{i=0}^{n-1}B(f_{\ast}^{i}\mu,f^n)\right)\\
%&=&  \frac{1}{n}\sum_{i=0}^{n-1} f_{\ast}^i\mu(B(f_{\ast}^{i}\mu,f^n))\\
%&=& 1.
%\end{eqnarray*}
%Let $\nu$ be the another $f^n$-ergodic measure such that $\mu^{<n>}=\nu^{<n>}$. By the same argument as above, we deduce that $\mu^{<n>}(\bigcup_{i=0}^{n-1}B(f_{\ast}^{i}\nu,f^n))=1$. As a consequence, there exist $k,\ell \in \{1,\cdots,n-1\}$ such that $B(f_{\ast}^k\mu,f^n)\cap B(f_{\ast}^{\ell}\nu,f^n)\neq \emptyset$, and thus $f_{\ast}^k\mu=f_{\ast}^{\ell}\nu$, which implies the desired result immediately.
%\end{proof}

%By definition of pre-skeletons, one gets the following result directly.
%
%\begin{lemma}\label{sc}
%Let $f$ be a $C^1$ diffeomorphism, and $n\in \NN$. If $\mathcal{S}=\{p_1,\cdots, p_k\}$ is a skeleton of $f^n$, then it is a pre-skeleton for $f$.
%\end{lemma}

\begin{proposition}\label{kyp}
Given $N\in \NN$ and a $C^{1+}$ diffeomorphism $f$ on $M$ with $f^N\in PH^{1+}_{C}(M)$. Let $\mu_1,\cdots,\mu_{k}$ be physical measures of $f^N$. Assume that $\mathcal{S}(f^N)=\{p_1,\cdots,p_{k}\}$ is a skeleton of $f^N$ satisfying $\mu_i \overset{f^N}\longleftrightarrow p_i$ for every $1\le i \le k$. Then the following properties hold:
\begin{enumerate}
\smallskip
\item \label{f1} $\mu^{(n)}_i=\mu^{(n)}_j$ implies $p_i\overset{f}\sim p_j$.
\smallskip
\item \label{f2} $p_i \overset{f}\prec p_j$ implies $\mu^{(n)}_i=\mu^{(n)}_j$.
\end{enumerate}
\end{proposition}

\begin{proof}
We prove Item (\ref{f1}) firstly. By Lemma \ref{ic}, $\mu^{(n)}_i=\mu^{(n)}_j$ implies $\mu_i=f_{\ast}^{\kappa}\mu_j$ for some $\kappa \in \{0, 1\cdots,N-1\}$. Let us fix $\ell$, $\alpha$, $\delta_{\ell}$, $\rho_{\ell}$ for the diffeomorphism $f^N$ as in Subsection \ref{fixedone}. By our definition of Pesin blocks, one can verify that there exist 
$\ell_0:=\ell_0(\ell, \alpha,\kappa)$, $\alpha_0:=\alpha_0(\ell,\alpha,\kappa)$ 
such that $f^{\kappa}(x)\in \Lambda_{\ell_0}(f^N,\alpha_0)$ whenever $x\in \Lambda_{\ell}(f^N,\alpha)$, that is,
\begin{equation}\label{inc}
f^{\kappa}(\Lambda_{\ell}(f^N,\alpha))\subset \Lambda_{\ell_0}(f^N,\alpha_0).
\end{equation}
For any small $\rho\le \rho_{\ell}$, one can create skeleton $\mathcal{S}_{f^N}=\{P_1,\cdots,P_{k}\}$ as in Lemmas \ref{st1} and \ref{st2}, with properties as follows:
\begin{itemize}
\smallskip
\item[--] $\mu_i\overset{f^N}\longleftrightarrow P_i$ for every $1\le i \le k$;
\smallskip
\item[--] up to shrinking $\delta_{\ell}$, we assume each $f^{j}(P_i)$, $1\le i \le k$, $0\le j\le \kappa$ has stable and unstable manifolds of size larger than $\delta_{\ell}$.
\end{itemize}
Moreover, we assume that each $P_i$ is contained in $B(x_i,\rho)\in \mathscr{A}_{\mu_i}( f^N,\ell, \alpha, \rho)$ for some $x_i\in {\rm supp}(\mu_i|\Lambda_{\ell}(f^N,\alpha))$. As before, let $C_{i}$ be the cylinder around $B(x_i,\rho)$ for every $1\le i \le k$.
By continuity of $f^{\kappa}$, one can take $\delta:=\delta(\rho)>0$ so that
\begin{equation}\label{ninc}
f^{\kappa}B(x,\rho)\subset B(f^{\kappa}(x),\delta)\quad \textrm{for every}~x\in M.
\end{equation}
Using (\ref{inc}), (\ref{ninc}) together with the fact $\mu_i=f_{\ast}^{\kappa}\mu_j$, we obtain
\begin{align*}
&\mu_i\left(B(f^{\kappa}(x_j),\delta)\cap \Lambda_{\ell_0}(f^N,\alpha_0)\right)\\
& \quad\quad \quad \quad\ge \mu_i\left(f^{\kappa}(B(x_j,\rho))\cap f^{\kappa}(\Lambda_{\ell}(f^N,\alpha))\right)\\
& \quad\quad \quad \quad= \mu_j\left(B(x_j,\rho)\cap \Lambda_{\ell}(f^N,\alpha)\right)\\
& \quad\quad \quad \quad> 0.
\end{align*}
Note that one can make $\delta$ small by choosing $\rho$ sufficiently small. 

\smallskip
Recall that we have $\mu_i(B(x_i,\rho)\cap \Lambda_{\ell}(f^N,\alpha))>0$ by construction. With the ergodicity of $\mu_i$, by applying Lemma \ref{sp1}, we know that up to reducing $\delta_{\ell}$ if necessary, by taking $\rho$ small enough, there exist $r\ll \delta_{\ell}$ and $m\in \NN$ such that
\begin{itemize}
\smallskip
\item[--] there exists a hyperbolic periodic point $P$, which is $r$-close to $f^{\kappa}(P_i)$ and exhibits stable and unstable manifolds of size larger than $\delta_{\ell}$.
\smallskip
\item[--] $f^m(P)$ is $r$-close to $P_i$, with local stable and unstable manifolds of size larger than $\delta_{\ell}$.
\end{itemize} 
%
%$$
%\mu_i(B(f^{\kappa}(x_j),\delta)\cap \mathcal{H}_{yp}(f^N,\ell_0,\alpha_0)>0.
%$$
%This implies that $\mu_i$-almost every point of $B(f^{\kappa}(x_j),\delta)\cap \mathcal{H}_{yp}(f^N,\ell_0,\alpha_0)$ is recurrent, by Poincare recurrent theorem. 
%The ergodicity of $\mu_i$ w.r.t. $f^N$ implies that ${\rm Orb}(x)$ enters to $B(x_i,\rho)\cap \mathcal{H}_{yp}(f^N,\ell,\alpha)$ for infinitely many times for $\mu_i$-almost every $x$. Combining these fact with Lemma \ref{Gan}, similar to the proof of Lemma \ref{st3}, we know that by choosing $\rho$ sufficiently small, there exist $r\ll \delta_{\ell}$ and $K\in \NN$ such that 
%\begin{itemize}
%\smallskip
%\item there exists a hyperbolic periodic point $P$, which is $r$-close to $f^{\kappa}(P_j)$ and exhibits stable and unstable manifolds of size larger than $\delta_{\ell}$.
%\smallskip
%\item $f^K(P)$ is $r$-close to $P_i$, with stable and unstable manifolds of size lager than $\delta_{\ell}$.
%\end{itemize}
These two facts imply that $f^{\kappa}(P_{j})\overset{f^N}\sim P_i$, so $f^{\kappa}(P_j)\overset{f}\sim P_i$. Therefore, we get
$P_j\overset{f}\sim P_i$ immediately.
On the other hand, following from Lemma \ref{ske1} and Proposition \ref{G}, we have $p_i\overset{f^N}\sim P_i$ and $p_j\overset{f^N}\sim P_j$, thus $p_i\overset{f}\sim P_i$ and $p_j\overset{f}\sim P_j$. Altogether, we obtain
$p_i\overset{f}\sim p_j$.

\smallskip
To prove Item (\ref{f2}), consider skeleton $\mathcal{S}_{f^N}$ as above, by Proposition \ref{G},
we can choose $C^1$ unstable disks $D_i, D_j$ such that for every $\ast \in \{i,j\}$
\begin{itemize}
\item[--] $D_{\ast}$ intersects $W^s({\rm Orb}(P_{\ast},f^N))$ transversely;
\smallskip
\item[--] Lebesgue almost every point of $D_{\ast}$ is contained in $B(\mu_{\ast},f^N)$.
\end{itemize}
Following from $p_i \overset{f}\prec p_j$, one gets $P_i \overset{f}\prec P_j$ immediately. Therefore, there is $L\in \NN$ such that $W^u(f^L(P_i), f)$ intersects $W^s(P_j,f)$ transversely. Note that Lebesgue almost every point of $f^{L}(D_i)$ is contained in $B(f^L\mu_i, f^N)$. By using inclination lemma, there exists $m$ as a multiple of $N$ and a disk $D_1$ inside $f^{L+m}(D)$ and a disk $D_2$ inside $f^m(D_j)$ such that, both $D_1$ and $D_2$ cross the cylinder $C_j$ around $P_j$. Then Lemma \ref{ggh} gives a disk $D_0$ crossing $C_j$ that admits a subset with positive Lebesgue measure, which possess stable manifolds of size larger than $\delta_{\ell}$. With the absolute continuity of stable lamination, one deduce that $\mu_j=f_{\ast}^L\mu_i$. Hence, we get $\mu^{(n)}_i=\mu^{(n)}_j$ by Lemma \ref{ic}. 
This completes the proof of Proposition \ref{kyp}.
\end{proof}

Now we can give the proof of Theorem of \ref{TheoB} by applying Proposition \ref{kyp} and the previous results.

\begin{proof}[Proof of Theorem \ref{TheoB}]
Let $f\in PH^{1+}_{EC}(M)$ with partially hyperbolic splitting $TM=E^u\oplus_{\succ} E^{cu} \oplus_{\succ} E^{cs}$. By Theorem \ref{utv}, we have $f^N\in PH^{1+}_{C}(M)$ for some $N\in \NN$. By Theorem \ref{ky} and Theorem \ref{TheoF}, there exists $s\in \NN$ with physical measures $\mu_1,\cdots,\mu_{s}$ for $f^N$ and skeleton $\mathcal{S}(f^N)=\{p_1,\cdots, p_{s}\}$ for $f^N$ such that
\begin{enumerate}
\smallskip
\item \label{f11} $\mu_i\overset{f^N}\longleftrightarrow p_i$ for every $1\le i \le s$;
\smallskip
\item \label{f22} the union of $B(\mu_1,f^N),\cdots B(\mu_{s},f^N)$ has full Lebesgue measure.
\end{enumerate}

\begin{claim}\label{claim}
Rearrange if necessary, there is $1\le k\le s$ such that $\{p_1,\cdots,p_k\}$ is a skeleton of $f$, and $\mu^{(N)}_i$, $1\le i \le k$ are (ergodic) physical measures with basin covering property. In addition, $\mu^{(N)}_i\overset{f}\longleftrightarrow p_i$ for every $1\le i \le k$.
%\begin{itemize}
%\item $\{p_1,\cdots,p_k\}$ is a skeleton of $f$;
%\item $\mu^{<N>}_i$, $1\le i \le k$ are (ergodic) physical measures with basin covering property;
%\item $\mu^{<N>}_i\overset{f}\longleftrightarrow p_i$ for every $1\le i \le k$.
%\end{itemize}
\end{claim}
By assuming Claim \ref{claim} and using Lemma \ref{ske1}, we complete the proof of Theorem \ref{TheoB}. 
\end{proof}

Now it remains to give the proof of the claim.

\begin{proof}[Proof of Claim \ref{claim}]
We obtain the ergodicity of $\mu^{(N)}_i$, $1\le i\le s$ by construction. Moreover, we have $B(\mu_i,f^N)\subset B(\mu^{(N)}_i, f)$ for every $1\le i \le s$, hence ${\rm Leb}(B(\mu^{(N)}_i,f))\ge {\rm Leb}(B(\mu_i,f^N))>0$, which means that $\{\mu^{(N)}_i: 1\le i \le s\}$ are physical measures for $f$;
from (\ref{f2}) one gets that the union of basins of $\{\mu^{(N)}_i : 1\le i \le s\}$ w.r.t. $f$ covers a full Lebesgue measure subset. Note that it may be happened that $\mu^{(N)}_i=\mu^{(N)}_j$ for some $i,j\in \{1,\cdots,s\}$ with $i\neq j$.

\smallskip
It follows from (\ref{f11}) and the construction that we have
%By definition, one can verify that for every $1\le i \le k$,
%\begin{equation}\label{ee1}
%{\rm supp}\mu^{<N>}_i=\bigcup_{j=0}^{N-1}{\rm supp}(f_{\ast}^j\mu_i)
%\end{equation}
%and
%\begin{equation}\label{ee2}
%\bigcup_{j=0}^{N-1}W^u({\rm Orb}(f^j(p_i), f^N))=W^u({\rm Orb}(p_i, f)).
%\end{equation}
%It follows from (\ref{f1}) that for every $1\le i \le k$, one has
%\begin{equation}\label{ee3}
%{\rm supp} (f_{\ast}^j\mu_i)=W^u({\rm Orb}(f^j(p_i), f^N))\quad \textrm{for every}~ j\in \NN. 
%\end{equation}
%Combining (\ref{ee1}), (\ref{ee2}) and (\ref{ee3}) yields 
\begin{equation}\label{t}
\mu^{(N)}_i\overset{f}\longleftrightarrow p_i \quad \textrm{for every}~ 1\le i \le s.
\end{equation}
Since $\mathcal{S}(f^N)=\{p_1,\cdots, p_{s}\}$ is a skeleton for $f^N$, by applying Lemma \ref{sc} we know that
$\{p_1,\cdots,p_s\}$ is a pre-skeleton for $f$. Taking $\mathcal{P}\subset \{ \mu^{(N)}_i: 1\le i \le s\}$ such that elements of $\mathcal{P}$ are different to each other and exhibits the basin covering property. By reordering if necessary, we may assume 
$$
\mathcal{P}=\{ \mu^{(N)}_1,\cdots, \mu^{(N)}_k\}\quad \textrm{for some}~1\le k \le s.
$$

\smallskip
It suffices to prove that $\mathcal{S}=\{ p_1,\cdots p_k\}$ is a skeleton of $f$. 
%To see this, we need to show $\mathcal{S}$ admits properties (S\ref{S1}) and (S\ref{S2}). 
We show that $\mathcal{S}$ satisfies (S\ref{S1}) firstly. For every $C^1$ disk $D$ transverse to $E^{cs}$, since $\mathcal{S}(f^N)$ is a pre-skeleton of $f$, there is $p_i\in \mathcal{S}(f^N)$ such that $D$ intersects $W^s({\rm Orb}(p_i,f))$ transversely. In addition, one can make $i$ to satisfy $i\le k$. Indeed, if $k<i\le s$, then by the choice of $\mathcal{P}$, there is $j\le k$ so that $\mu^{(N)}_j=\mu^{(N)}_i$. By using (\ref{f1}) of Proposition \ref{kyp}, we get $p_i\overset{f}\sim p_j$. This together with inclination lemma implies that $D$ must intersect $W^s({\rm Orb}(p_j,f))$ transversely also. This says that $\mathcal{S}$ satisfies condition (S\ref{S1}). Now we show that $\mathcal{S}$ satisfies condition (S\ref{S2}). By absurd we have $p_i \overset{f}\prec p_j$ for some $i, j\in \{1,\cdots, k \}$, then by applying (\ref{f2}) of Proposition \ref{kyp} we get $\mu^{(N)}_i=\mu^{(N)}_j$, this is a contradiction. Now, we can use (\ref{t}) and Lemma \ref{ske2} to complete the proof.
\end{proof}

We are in ready to give the proof of Theorem \ref{TheC}.

\begin{proof}[Proof of Theorem \ref{TheC}]
By Theorem \ref{uppers} and Theorem \ref{TheoB}, we get the upper semi-continuity of the number of physical measures and the cardinality of skeleton in $C^1$-topology among $PH^{1+}_{EC}(M)$. Indeed, as showed in Theorem \ref{uppers}, for each diffeomorphism $f\in PH^{1+}_{EC}(M)$, 
there exists a $C^1$ neighborhood $\V$ of $f$ such that $k_g \le k_f$ for every $g\in \V$. Furthermore, there exists a $C^1$ open and dense subset $\V_0$ of $\V$ on which the cardinality of skeleton is locally constant. To this end, we consider a sequence of subsets of $\V$ as follows:
$$
\U_0=\emptyset; \quad \U_i=\{g\in \V: k_g\le i\}\quad \textrm{when}\quad 1\le i \le k_f.
$$
By definition, it suffices to take $\V_0=\bigcup_{1\le i \le k_f}(\U_i\setminus \overline{\U_{i-1}})$.
Using Theorem \ref{TheoB}, we know that the number of the physical measures for $C^{1+}$ diffeomorphisms of $\V_0$ is also locally constant.

\smallskip
Now we verify the continuity of the physical measures w.r.t. $C^{1+}$ diffeomorphisms of $\V_0$ in following precise way:
%To this end, let us fix any $g\in {\rm Diff}^1(M)\cap\V_0$, assume $g_n$ is a sequence of diffeomorphism in ${\rm Diff}^1(M)\cap\V_0$ that converges to $g$ in $C^1$-topology. By the property of $\V_0$, we can assume further that $k_{g_n}=k_g$ for every $n\in \NN$ for simplicity.
%According to Proposition \ref{ffff} and Theorem \ref{TheoB}, one can take a family of skeletons for $g_n$ and $g$ with properties as follows:
%\begin{itemize}
%\smallskip
%\item[--] Take skeleton $\mathcal{S}(g)=\{p_1,\cdots p_{k_g}\}$ of $g$ such that its continuation $\mathcal{S}(g_n)=\{p_1(g_n),\cdots,p_{k_f}(g_n)\}$ is a skeleton of $g_n$ for every $n\in \NN$.
%\smallskip
%\item[--] Let $\{\mu_i: 1\le i \le k_g\}$ be a sequence of physical measures of $g$ satisfying $\mu_i \leftrightarrow p_i$ for each $1\le i \le k_g$.
%\smallskip
%\item[--] For every $n\in \NN$, denote by $\{\mu_i^{(n)}: 1\le i \le k_g\}$ the physical measures of $g_n$ satisfying $\mu_i^{(n)} \leftrightarrow p_i(g_n)$ for each $1\le i \le k_g$.
%\end{itemize}

\begin{lemma}\label{klm}
Let $g$ be a $C^{1+}$ diffeomorphism in $\V_0$. Assume that $\mathcal{S}(g)=\{p_1,\cdots,p_{k_g}\}$ is a skeleton of $g$ and $\{\mu_i: 1\le i \le k_g\}$ are physical measures of $g$ satisfying $\mu_i \overset{g}\longleftrightarrow p_i$ for each $1\le i \le k_g$. If $g_n$ is a sequence of $C^{1+}$ diffeomorphisms in $\V_0$ that converges to $g$ in $C^1$-topology, then for every $n\in \NN$ sufficiently large, we have
\begin{itemize}
\item the continuation of $\mathcal{S}(g)$, denoted by $\mathcal{S}(g_n)=\{p_1(g_n),\cdots,p_{k_g}(g_n)\}$, is a skeleton of $g_n$. 
\item if we let $\{\mu_{i,n}: 1\le i \le k_g\}$ be the set of physical measures of $g_n$ satisfying $\mu_{i,n}\overset{g_n} \longleftrightarrow p_i(g_n)$ for each $1\le i \le k_g$, then $\mu_{i,n}\xrightarrow{weak^*}\mu_i$ as $n\to+\infty$ whenever $i\in\{1,\cdots, k_g\}$.
\end{itemize}
\end{lemma}

\begin{proof}
As a direct consequence of Theorem \ref{uppers} and the definition of $\V_0$, one knows that the first item is true. Now we make effort to show the second item. By Theorem \ref{utv}, there is $N\in \NN$ such that $f^N\in PH^{1+}_{C}(M)$. Consequently, shrinking $\V$ if necessary, we may assume $g^N\in PH^{1+}_{C}(M)$ with dominated splitting $TM=E^u\oplus_{\succ} E^{cu}\oplus_{\succ} E^{cs}$ for every $g\in \V$. In view of Claim \ref{claim}, without loss of generality we can assume $\V\subset PH^{1}_{C}(M)$.

\smallskip
If the conclusion is not true, then there exists $\varkappa\in \{1,\cdots,k_g\}$ such that $\mu_{\varkappa,n}$ does not converge to $\mu_{\varkappa}$ when $n$ goes to infinite. Note that all physical measures are Gibbs $cu$-states, by the upper semi-continuity of Gibbs $cu$-space w.r.t. diffeomorphisms as stated in Theorem \ref{fc}, up to considering subsequence we assume $\mu_{\varkappa,n}$ converges to $\mu\in G^{cu}(g)$ different to $\mu_{\varkappa}$. It follows from (\ref{ps1}) of Theorem \ref{ky} that $\mu$ is a convex combination of physical measures of $g$, thus there exists $\xi_i\in [0,1]$ such that 
$$
\mu=\sum_{i=1}^{k_g}\xi_{i} \mu_i, \quad \textrm{where}~\sum_{i=1}^{k_g}\xi_i=1.
$$
Since $\mu\neq \mu_{\varkappa}$, there exists $\iota \in \{1,\cdots, k_g\}$ with $\iota\neq\varkappa$ such that $\xi_{\iota}>0$.

\smallskip
Choosing $\varepsilon<\xi_{\iota}/(1+\xi_{\iota})$, Proposition \ref{t11} applies to $g$ and $\varepsilon$, we know that there exist $\alpha>0$, $\ell\in \NN$ and a $C^1$ neighborhood $\W\subset \V$ of $g$ such that for every $h\in \W$ we have
\begin{equation}\label{one}
\mu(\Lambda_{\ell}(h,\alpha))>1-\varepsilon,\quad \forall \mu \in G^{u}(h).
\end{equation}
%Meanwhile, for every $n$ large enough we have
%\begin{equation}\label{two}
%\mu(\mathcal{H}_{yp}(g_n,\ell,\alpha))>1-\varepsilon,\quad \forall \mu \in \G^{u}(g_n).
%\end{equation}
By Lemma \ref{sp}, there exist $\rho_{\ell}>0$ and $\delta_{\ell}>0$ such that
for every $h\in \W$, for every $\mu\in G^u(h)$ and $x\in {\rm supp}(\mu|\Lambda_{\ell}(h,\alpha))$, any ball $B(x,\rho)$($\rho\le \rho_{\ell}$) contains hyperbolic periodic points with stable and unstable manifolds of size larger than $\delta_{\ell}$. 

\smallskip
For fixed $\mu_{\iota}$, which is a Gibbs $u$-state of $g$, take $\rho\le \rho_{\ell}$ within $\rho\ll \delta_{\ell}$, recalling the definition of $\mathscr{A}_{\mu_{\iota}}(g,\ell,\alpha,\rho)$ in Theorem \ref{covering}, and put
$$
A_{\rho}=\bigcup_{B\in \mathscr{A}_{\mu_{\iota}}(g,\ell,\alpha,\rho)}B.
$$
For every $B\in \mathscr{A}_{\mu_{\iota}}(g,\ell,\alpha,\rho)$, let $P_{B,\iota}$ be a fixed hyperbolic periodic point in it which admits stable and unstable manifolds of size larger than $\delta_{\ell}$. Moreover, In view of the proof of Lemma \ref{st2}, we know 
\begin{equation}\label{hom2}
P_{B,\iota}\overset{g}\sim p_{\iota}.
\end{equation}
Using (\ref{one}) to $g$ and $\mu_{\iota}$, we have $\mu_{\iota}(A_{\rho})>1-\varepsilon$, which implies $\mu(A_{\rho})>\xi_{\iota}(1-\varepsilon)$ by representation of $\mu$.
Since $\mu_{\varkappa,n}\xrightarrow{weak^*}\mu$ as $n\to +\infty$, and $A_{\rho}$ is open, we get 
\begin{equation}\label{three}
\mu_{\varkappa,n}(A_{\rho})>\xi_{\iota}(1-\varepsilon)
\end{equation}
for any $n$ large enough. On the other hand, note that $\mu_{\varkappa,n}$ is a Gibbs $u$-state of $g_n$ for every $n$, use (\ref{one}) again we obtain 
$
\mu_{\varkappa,n}(\Lambda_{\ell}(g_n,\alpha))>1-\varepsilon
$
for every $n$ sufficiently large. Combining this with (\ref{three}) and our choice of $\varepsilon$, we see that for every large $n$, 
\begin{equation}\label{four}
\mu_{\varkappa,n}\left(\Lambda_{\ell}(g_n,\alpha)\cap A_{\rho}\right)>0.
\end{equation}

\smallskip
Observe that $\mathscr{A}_{\mu_{\iota}}(g,\ell,\alpha,\rho)$ is consisted of finitely many $\rho$-balls. By (\ref{four}), up to taking subsequence, one may fix $B\in \mathscr{A}_{\mu_{\iota}}(g,\ell,\alpha,\rho)$ such that 
$$
\mu_{\varkappa,n}\left(\Lambda_{\ell}(g_n,\alpha)\cap B\right)>0\quad \textrm{for every large}~n.
$$
This implies that for every $n$ large enough, one can take some point $x_n\in {\rm supp}(\mu_{\varkappa,n}|\Lambda_{\ell}(g_n,\alpha))\cap B$. From the choices of $\rho$ and $\delta_{\ell}$, there exists a hyperbolic periodic point $p_{\varkappa,n}\in B(x_n,\rho)$,
which has stable and unstable manifolds of size larger than $\delta_{\ell}$. 
According Lemma \ref{st2}, for each $n$, one can take a hyperbolic periodic point $p_{\varkappa}^{(n)}$ in a ball $B_{\varkappa,n}\in \mathscr{A}_{\mu_{\varkappa,n}}(g_n,\ell,\alpha,\rho)$ such that $\mu_{\varkappa,n}\overset{g_n} \longleftrightarrow p_{\varkappa}^{(n)}$. Since we have also $\mu_{\varkappa,n}\overset{g_n} \longleftrightarrow p_{\varkappa}(g_n)$ from assumption, we get 
\begin{equation}\label{ss1}
p_{\varkappa}^{(n)}\overset{g_n}\sim p_{\varkappa}(g_n).
\end{equation}
By construction, we have
$$
\mu_{\varkappa,n}(B_{\varkappa,n}\cap\Lambda_{\ell}(g_n,\alpha))>0, \quad \mu_{\varkappa,n}(B(x_n,\rho)\cap\Lambda_{\ell}(g_n,\alpha))>0.
$$
As $\mu_{\varkappa,n}$ is $g_n$-ergodic,
by Lemma \ref{sp1}, similar to the argument in the proof of Lemma \ref{st2}, we get $p_{\varkappa}^{(n)}\overset{g_n} \sim p_{\varkappa,n}$, choosing $\rho$ smaller if necessary. 
%
%that up to shrinking $\delta_{\ell}$ and $\rho$, there is $L>0$ within $L\rho\ll \delta_{\ell}$ such that
%\begin{itemize}
%\item[--] there is a hyperbolic periodic point $P$ with its iterate $g_n^m(P)$ having local stable and unstable manifolds of size larger than $\delta_{\ell}$.
%\item[--] $P$ is $L\rho$-close to $p_{\varkappa}^{(n)}$ and $g_n^m(P)$ is $L\rho$-close to $p_{\varkappa,n}$.
%\end{itemize}
%Note also that both $p_{\varkappa}^{(n)}$ and $p_{\varkappa,n}$ has local stable and unstable manifolds of radius larger than $\delta_{\ell}$.
%These fact implies $p_{\varkappa}^{(n)}\overset{g_n} \sim p_{\varkappa,n}$.
%
% Note that both $p_{\varkappa}^{(n)}$ and $p_{\varkappa,n}$ are contained in some ball which intersect $\mathcal{H}_{yp}(g_n,\ell,\alpha)$ with positive $\mu_{\varkappa}$ measure, by using the ergodicity one can deduce that , similar to the proof of Lemma \ref{st3}.
This together with (\ref{ss1}) yields
\begin{equation}\label{hom1}
p_{\varkappa,n}\overset{g_n}\sim p_{\varkappa}(g_n).
\end{equation}
%Recall that $\mu_{\iota}\overset{g} \longleftrightarrow P_{B,\iota}$, by the assumption $\mu_{\iota}\overset{g} \longleftrightarrow p_{\iota}$, we get
%\begin{equation}\label{hom2}
%P_{B,\iota}\overset{g}\sim p_{\iota}.
%\end{equation}
%In view of Lemma \ref{st3} and the proof of Theorem \ref{TheoB}, we have
%\begin{enumerate}[(H1)]
%\item \label{J}the periodic point $P_{B,\iota}$ is homoclinic related to $p_{\iota}(g)$.
%\smallskip 
%\item \label{H}$p_{\varkappa,n}$ is homoclinic related to $p_{\varkappa}(g_n)$.
%\end{enumerate}

\smallskip
For every $n$ large enough, denote by $P_{B,\iota,n}$ the continuation of $P_{B,\iota}$ w.r.t. $g_n$. Using (\ref{hom2}) and the fact $d(P_{B,\iota}, p_{\kappa,n})<2\rho$, $\rho\ll \delta_{\ell}$, it follows that one can fix $n_0$ large enough such that both $p_{\iota}(g_{n_0})$ and $p_{\varkappa,n_0}$ are homoclinically related to $P_{B,\iota,n_0}$, thus 
$p_{\iota}(g_{n_0})\overset{g_{n_0}} \sim p_{\varkappa,n_0}$. 
This together with (\ref{hom1}) by taking $n=n_0$ yields $p_{\varkappa}(g_{n_0})\overset{g_{n_0}} \sim p_{\iota}(g_{n_0})$, this is a contradiction to the definition of skeletons.

\smallskip
The proof of Lemma \ref{klm} is complete now.
\end{proof}
 
Observe that for each fixed $L>0$, by applying the unstable manifold theorem of hyperbolic periodic point, the compact part $W_L^u({\rm Orb}(p_i(g),g))$ depends continuously on diffeomorphism $g$, which implies that the closure of $W^u({\rm Orb}(p_i(g),g))$ vary lower semi-continuously on $g$ with the Hausdorff topology. By applying Theorem \ref{TheoB}, we have the lower semi-continuity of the supports of physical measures.
Now we have complete the proof of Theorem \ref{TheC}.
\end{proof}

\bigskip
We close this section by finishing the proof of Theorem \ref{sta}.

\begin{proof}[Proof of Theorem \ref{sta}]
We show the weak statistical stability firstly. We start by taking a sequence of diffeomorphisms $f_n\in PH^{1+}_{EC}(M)$ that converges to $f\in PH^{1+}_{EC}(M)$ in $C^1$-topology, and let $\mu_n$ be the physical measure of $f_n$ respectively. Up to considering subsequences, we assume that $\mu_n$ converges to some $\mu\in \M(f)$. Since all ergodic Gibbs $cu$-states are physical measures of diffeomorphisms in $PH^{1+}_{EC}(M)$, by applying Theorem \ref{fc} we know that $\mu$ is a Gibbs $cu$-state of $f$, which is a convex combination of the physical measures of $f$. 

\smallskip
Assume that $f\in PH^{1+}_{EC}(M)$ admits a unique physical measure $\mu_f$. By Theorem \ref{TheC} there exists a $C^1$ neighborhood $\V$ of $f$ such that every $g\in \V\cap PH^{1+}_{EC}(M)$ has a unique physical measure $\mu_g$ also, and one must have $\mu_{f_n}\mapsto \mu_f$ whenever $f_n\in \V\cap PH^{1+}_{EC}(M)$ converges to $f$ in $C^1$-topology.
\end{proof}

\appendix
\section{Proof of Proposition \ref{t11}}\label{appendix}

The following lemma given by \cite[Lemma 5.8]{CMY18} is a variation version of the Pliss-Like Lemma founded by Andersson-Vasquez \cite[Lemma A]{AV17}.
\begin{lemma}\label{pliss}
Given constants $C_1<C_2\le \max\{0,C_2\}<C$, for any $\varepsilon>0$, there is $\rho=\rho(C_1,C_2,C,\varepsilon)>0$ such that for any any sequence $\{a_n\}_{n\in\NN}\subset \mathbb{R}$ satisfying:

\begin{itemize}

\item $|a_n|\le C$, $\forall n\in \NN$,
\smallskip
\item there is a subset $P\subset \NN$ such that $\mathcal{D}_{L}(P)>1-\rho$ and $a_n\le C_1$ for any $n\in P$,
\end{itemize}
then there is a subset $Q \subset\NN$ with $\mathcal{D}_{U}(Q)>1-\varepsilon$ such that for any $j\in Q$, one has that
$$\sum_{i=0}^{n-1}a_{i+j}\le nC_2,~~~\forall n\in\NN.$$
\end{lemma}

The following lemma plays a key role in the proof of Proposition \ref{t11}. 
\begin{lemma}\label{plg}
Let $f$ be a $C^1$ diffeomorphism with dominated splitting $TM=E\oplus F$. Given $0<\alpha<\alpha_0$, if $\mu$ is an $f$-invariant measure such that
\begin{equation}\label{formu}
\lim_{n\to +\infty}\frac{1}{n}\log \|Df^{-n}|_{E(x)}\|<-\alpha_0, \quad \mu-a.e.  ~x\in M.
\end{equation}
\begin{equation}\label{forms}
\lim_{n\to +\infty}\frac{1}{n}\log \|Df^{n}|_{F(x)}\|<-\alpha_0, \quad \mu-a.e.  ~x\in M.
\end{equation}
Then, for every $\varepsilon>0$ there exists $\ell\in \mathbb{N}$ and a neighborhood $\U$ of $f$ such that for every $g\in \U$ and $g$-invariant measure $\nu$, we have
$$
\nu\left(\Lambda_{\ell}(g,\alpha,E,F)\right)>1-\varepsilon.
$$
\end{lemma}
\begin{proof}
For simplicity, for every $g$ that is $C^1$-close to $f$, we introduce the two blocks $\Lambda_{\ell}^{-}(g,\alpha)$ and $\Lambda_{\ell}^{+}(g, \alpha)$ defined as follows:
$$
\Lambda_{\ell}^{-}(g, \alpha)=\left\{x: \frac{1}{n\ell}\sum_{i=0}^{n-1}\log \|Dg^{-\ell}|_{E(f^{-i\ell}(x))}\|\le -\alpha, \quad \forall n\in \NN\right\};
$$
$$
\Lambda_{\ell}^{+}(g,\alpha)=\left\{x: \frac{1}{n\ell}\sum_{i=0}^{n-1}\log \|Dg^{\ell}|_{F(f^{i\ell}(x))}\|\le -\alpha,\quad \forall n\in \NN\right\}.
$$
Hence, $\Lambda_{\ell}(g,\alpha)=\Lambda_{\ell}^{-}(g, \alpha)\cap \Lambda_{\ell}^{+}(g,\alpha)$.
%Since domination is a $C^1$ open property, any $g$ $C^1$ close to $f$ admits the dominated slitting 
%$TM=E_g\oplus F_g$. Note that $E_g, F_g$ depend continuously on $g$, and thus in particular $E_f=E, F_f=F$. 
Given $n\in \NN$, let us define 
$$
\Lambda^E_{g,n}=\big\{x: \frac{1}{n}\log\|Dg^{-n}|_{E(x)}\|<-\alpha_0 \big\};
$$
$$
\Lambda^F_{g,n}=\big\{x: \frac{1}{n}\log\|Dg^{n}|_{F(x)}\|<-\alpha_0 \big\}.
$$
Take $\Lambda_{g,n}=\Lambda^E_{g,n}\cap \Lambda^F_{g,n}$.
%Define
%$$
%B^{E}(g,\ell, \alpha)=\big\{x: \frac{1}{n\ell}\sum_{i=0}^{n-1}\log \|Dg^{-\ell}|E_g(g^{-i\ell}(x))\|\le -\alpha\};
%$$
%$$
%B^{F}(g,\ell, \alpha)=\big\{x: \frac{1}{n\ell}\sum_{i=0}^{n-1}\log \|Dg^{\ell}|F_g(g^{i\ell}(x))\|\le -\alpha\}.
%$$
%Clearly, we have
%$
%B(g,\ell,\alpha,E,F)=B^{E}(g,\ell, \alpha)\cap B^{F}(g,\ell, \alpha).
%$

\smallskip
Choose $\varepsilon'>0$ satisfying $(1-\varepsilon')^2>1-\varepsilon/2$ and fix constants
$$
C_1=-\alpha_0, \quad C_2=-\alpha, \quad C=\max_{x\in M}|\log \|Df^{\pm 1}|_{E(x)}\||+1.
$$
Then take $\rho=\rho(C_1,C_2,C,\varepsilon)$ as in Lemma \ref{pliss}.
It follows from (\ref{formu}) and (\ref{forms}) that for $0<\rho \varepsilon'<1$, there exists $\ell \in \mathbb{N}$ such that $\mu(\Lambda_{f,\ell})>1-\rho\varepsilon'$.
By the continuity of diffeomorphisms and \cite[Lemma 3.2]{AV17}, there exists an open neighborhood $\UU$ of $(f,\mu)$ such that  
\begin{equation}\label{forB}
\nu(\Lambda_{g,\ell})>1-\rho \varepsilon'
\end{equation}
and $\max_{x\in M}|\log \|Dg^{\pm 1}|_{E(x)}\||\le C$
for every $(g,\nu)\in \UU$. 
Consequently, we have $\nu(\Lambda^E_{g,\ell})>1-\rho \varepsilon'$, $\nu(\Lambda^F_{g,\ell})>1-\rho \varepsilon'$ for every $(g,\nu)\in \UU$.
This implies that there exists a $C^1$ open neighborhood $\U$ of $f$ satisfying $(g,\nu)\in \UU$ for every $g\in \U$ and $g$-invariant measure $\nu$. 

\smallskip
Since $\nu$ is $g^{\ell}$-invariant, by Birkhoff's ergodic theorem we get that for $\nu$-almost every $x\in M$, there exists the limit
$$
\phi(x)=\lim_{n\to +\infty}\frac{1}{n}\# \left \{ i: 0 \le i \le n-1, g^{-\ell i}(x) \in \Lambda^E_{g,\ell} \right\}.
$$
Furthermore, applying (\ref{forB})  we have
\begin{equation}\label{forG}
\int \phi(x)d\nu= \nu (\Lambda^E_{g,\ell})>1-\rho \varepsilon'.
\end{equation}
Let
$
G=\left\{x: 1-\rho<\phi(x)\le 1\right\}
$, we claim that 
$\nu(G)>1-\varepsilon'$.
Indeed, observe that
$
M\setminus G = \{x: \rho \le 1-\phi(x) \le 1\},
$
this together with (\ref{forG}) imply
$$
\rho \varepsilon'> \int (1-\phi)d\nu \ge \int (1-\phi)\chi_ {M\setminus G}d\nu \ge \rho \nu(M\setminus G).
$$
Thus, $\nu(G)>1-\varepsilon'$.
For any $x\in G$, we define
$$
a_i=\frac{1}{\ell}\log \|Dg^{-\ell}|_{E(g^{-i\ell}x)}\|,\quad \forall i\ge 0,
$$
and
$$
P=\{i\ge 0: g^{-i\ell}(x)\in \Lambda^E_{g,\ell}\}.
$$
Then, by properties of  $\Lambda^E_{g,\ell}$ and $G$, one has that $a_i<-\alpha$ for every $i\in P$, and 
$\mathcal{D}_{L}(P)>1-\rho$.
Thus, by applying Lemma \ref{pliss}, there exists a subset $Q\subset \mathbb{N}$ with density estimate
\begin{equation}\label{d}
\mathcal{D}_{U}(Q)>1-\varepsilon'.
\end{equation}
Moreover, for every $j\in Q$, we have
$$
\sum_{i=0}^{n-1}a_{i+j}\le -n\alpha, \quad  \forall n\in \NN,
$$
which can be rewrite as
$$
\prod_{i=0}^{n-1}\|Dg^{-\ell}|_{E(g^{-{i+j}\ell}(x))}\|\le {\rm e}^{-n\ell \alpha}, \quad \forall n\in\mathbb{N}.
$$
Therefore, by combining (\ref{d}) with Birkhoff's ergodic theorem we know that for $\nu$-almost every $x\in G$,
$$
\lim_{n\to +\infty}\frac{1}{n}\#\left\{i\in [0,n-1]:  g^{-i\ell}(x)\in \Lambda_{\ell}^{-}(g,\alpha)\right\}>1-\varepsilon'.
$$
Therefore, there exists $n_0\in \NN$ and a subset
$$
G_{0}=\left\{x\in G: \frac{1}{n_0}\#\left\{i\in [0, n_0-1]: g^{-i\ell}(x)\in \Lambda_{\ell}^{-}(g,\alpha)\right\}>1-\varepsilon' \right\}
$$
satisfying $\nu(G_0)>1-\varepsilon'$.
Then, we have
\begin{eqnarray*}
\nu(\Lambda_{\ell}^{-}(g,\alpha))&=&\int \frac{1}{n_0}\sum_{i=0}^{n_0-1}\chi_{\Lambda_{\ell}^{-}(g,\alpha)}(g^{-i\ell}x)d\nu\\
&\ge &\int_{G_0} \frac{1}{n_0}\sum_{i=0}^{n_0-1}\chi_{\Lambda_{\ell}^{-}(g,\alpha)}(g^{-i\ell}x)d\nu\\
&=&\int_{G_0} \frac{1}{n_0}\#\left\{i\in [0, n_0-1]: g^{-i\ell}(x)\in \Lambda_{\ell}^{-}(g,\alpha)\right\}d\nu\\
&\ge & (1-\varepsilon')\nu(G_0)\\
&>&(1-\varepsilon')^2>1-\varepsilon/2.
\end{eqnarray*}
Similarly, for $F$ direction one can get $\nu(\Lambda_{\ell}^{-}(g,\alpha))>1-\varepsilon/2$. Therefore, 
$$
\nu(\Lambda_{\ell}(g,\alpha))>1-\varepsilon.
$$
Now the proof of Lemma \ref{plg} is completed.
\end{proof}
%The following result asserts that for diffeomorphims in $\U(M)$, the central Lyapunov exponents for Gibbs $u$-sates can be estimated uniformly.
%\begin{lemma}\label{p33}
%If $f\in \U(M)$, then there exists $N_0\in \NN$ and $\alpha_0>0$ such that for every $\mu \in \G^u(f)$, one has
%$$
%\chi_{\mu}(f^{N_0},E^{cs})<-\alpha_0, \quad \chi_{\mu}(f^{-N_0},E^{cu})<-\alpha_0.
%$$
%Furthermore, there is $N_{u}\in \NN$ as a
%multiple of $N_0$ and $\alpha_u\in (0, \alpha_0/N_0)$ such that
%$$
%\chi_{\mu}(f^{N_u},E^{cs})<-\alpha_u,\quad \chi_{\mu}(f^{-N_u},E^{cu})<-\alpha_u$$
%for every $\mu \in \G^u(f^{N_u})$.
%\end{lemma}
%
%This can be concluded by Lemma \ref{gibbsp}, it is an adapted version of Lemma 3.4 and Proposition 3.3 of \cite{MCY17}, though the definition of Gibbs $u$-states are different. 
%As a result, we have the following corollary.
%
%\begin{corollary}\label{utv}
%If $f\in \U(M)$, then there exists $N\in \NN$ such that $f^N\in \V(M)$.
%\end{corollary}
%
By Lemma \ref{gibbsp}, one can conclude the following result. It can be seen as a combination of \cite[Lemma 3.4 ]{MCY17} and \cite[Lemma 3.5]{MCY17}, though the definition of Gibbs $u$-state is different from that in \cite{MCY17}. 

\begin{lemma}\label{p33}
If $f\in PH^1_{EC}(M)$, then there exist $N\in \NN$ and $\alpha>0$ such that for every $\mu \in G^u(f)$, one has
$$
\lim_{n\to +\infty} \frac{1}{n}\sum_{i=0}^{n-1}\log \|Df^{-N}|_{E^{cu}(f^{-iN}(x))}\|<-\alpha
$$
for $\mu$-almost every $x$.
\end{lemma}

Now we can finish the proof of Proposition \ref{t11}.

\begin{proof}[Proof of Proposition \ref{t11}] 
By Lemma \ref{p33}, there exist $N\in \NN$ and $\alpha>0$ such that for every $\mu\in G^u(f)$ and $\mu$-almost every $x$ we have
$$
\lim_{n\to +\infty} \frac{1}{n}\sum_{i=0}^{n-1}\log \|Df^{-N}|_{E^{cu}(f^{-iN}(x))}\|<-\alpha,
$$
by taking $\alpha_0=\alpha/N$, we get that for $\mu$-almost every $x$,
$$
\lim_{n\to +\infty}\frac{1}{n}\log \|Df^{-n}|_{E^{cu}(x)}\|<-\alpha_0.
$$
Similarly, we have that
$$
\lim_{n\to +\infty}\frac{1}{n}\log \|Df^{n}|_{E^{cs}(x)}\|<-\alpha_0.
$$
for $\mu$-almost every $x$.
Now fix $\alpha\in (0,\alpha_0)$. By Lemma \ref{plg}, for every $\varepsilon>0$, for any $\mu \in G^u(f)$, there exists $\ell_{\mu}\in \NN$ and neighborhood $\mathcal{U}_{\mu}\times \mathcal{V}_{\mu}$ of $(f,\mu)$ in $G^u(PH^{1}_{EC}(M))$\footnote{$G^u(PH^{1}_{EC}(M))=\left\{(f,\mu): \mu\in G^u(f), f\in PH^{1}_{EC}(M)\right\}$.} such that
\begin{equation}\label{fg}
\nu(\Lambda_{\ell_{\mu}}(g, \alpha))>1-\varepsilon
\end{equation}
for every $(g,\nu)\in \mathcal{U}_{\mu}\times \mathcal{V}_{\mu}$.
By compactness of $G^u(f)$ as stated in Item (1) of Lemma \ref{gibbsp}, there exist finitely many Gibbs $u$-states $\mu_1,\cdots,\mu_j$ such that the corresponding neighborhoods $\mathcal{V}_i:=\mathcal{V}_{\mu_i}, 1\le i \le j$ form a covering of $G^u(f)$. Writting $\mathcal{U}_i:=\mathcal{U}_{\mu_i}$ and $\ell_i:=\ell_{\mu_i}$ for every $1\le i \le j$.
Let
$$
\mathcal{U}=\bigcap_{1\le i \le j} \mathcal{U}_i; \quad \mathcal{V}=\bigcup_{1\le i\le j}\mathcal{V}_i;\quad \ell=\prod_{i=1}^j \ell_j.
$$
Note that by definition of $\ell$, we have $\Lambda_{\ell_i}(g,\alpha)\subset \Lambda_{\ell}(g,\alpha)$ for every $1\le i \le j$.
By Lemma \ref{upp}, up to shrinking $\mathcal{U}$ we assume $G^u(g)\subset \mathcal{V}$ for every $g\in \mathcal{U}$. Therefore, for every $(g,\nu)\in \mathcal{U}\times \mathcal{V}$, there is some $1\le i\le j$ such that $(g,\nu)\in \mathcal{U}_i\times \mathcal{V}_i$, and thus (\ref{fg}) implies $\nu(\Lambda_{\ell}(g,\alpha))>1-\varepsilon$.
\end{proof}

\end{document}